\documentclass[11pt]{article}
\usepackage{amssymb, amsmath}
\renewcommand{\theequation}{\arabic{equation}}

\usepackage[margin=1.2in]{geometry}

\numberwithin{equation}{section}

\usepackage{mathptmx}

\newtheorem{theorem}{Theorem}[section]
\newtheorem{lemma}{Lemma}[section]

\newtheorem{corollary}{Corollary}[section]
\newtheorem{conjecture}{Conjecture}[section]

\newenvironment{proof}[1][Proof]{\begin{trivlist}
\item[\hskip \labelsep {\bfseries #1}]}{\end{trivlist}}
\newenvironment{definition}[1][Definition]{\begin{trivlist}
\item[\hskip \labelsep {\bfseries #1}]}{\end{trivlist}}

\newenvironment{remark}[1][Remark]{\begin{trivlist}
\item[\hskip \labelsep {\bfseries #1}]}{\end{trivlist}}

\newcommand{\qed}{\nobreak \ifvmode \relax \else
      \ifdim\lastskip<1.5em \hskip-\lastskip
      \hskip1.5em plus0em minus0.5em \fi \nobreak
      \vrule height0.75em width0.5em depth0.25em\fi}

\renewcommand{\Re}{\operatorname{Re}}
\renewcommand{\Im}{\operatorname{Im}}

\begin{document}
\title{A Review of Conjectured Laws of Total Mass of Bacry-Muzy GMC Measures on the Interval and Circle and Their Applications}

\author{Dmitry Ostrovsky}

\date{}

\maketitle
\noindent

\begin{abstract}
\noindent
Selberg and Morris integral probability distributions are long conjectured to be distributions of the total mass of the Bacry-Muzy Gaussian
Multiplicative Chaos measures with non-random logarithmic potentials on the unit interval and circle, respectively. The construction and properties of these distributions 
are reviewed from three perspectives: analytic based on several representations of the Mellin transform, asymptotic based on low intermittency expansions, and probabilistic based on the theory of Barnes beta probability distributions. In particular, positive and negative integer moments, infinite factorizations 
and involution invariance of the Mellin transform, analytic and probabilistic proofs
of infinite divisibility of the logarithm, factorizations into products of Barnes beta distributions, and Stieltjes moment problems of these distributions are presented in detail. 
Applications are given in the form of conjectured mod-Gaussian limit theorems, laws of derivative martingales, distribution of extrema of $1/f$ noises, and calculations of inverse participation ratios in the Fyodorov-Bouchaud model. 

\end{abstract}

\tableofcontents

\section{Introduction}
\noindent In this review we study conjectured laws of the total mass of the Bacry-Muzy \cite{MRW}, \cite{BM1}, \cite{BM} Gaussian
Multiplicative Chaos (GMC) measures on the unit interval and circle with non-random logarithmic potentials. These measures are the most
tractable examples of GMC measures in dimension one and serve as a paradigm of all GMC measures due
to their very natural logarithmic covariance structures and connections with $1/f$ noises.  
While the positive integer moments of the total mass of all multiplicative chaos measures can be written in the form of multiple integrals, cf.  \cite{MeIntRen},
\cite{MeLMP},
the tractability of the Bacry-Muzy measures is inextricably tied to the fact that they are the only GMC measures that have explicitly known moments
given by the Selberg integral on the interval and the Morris integral on the circle. A key challenge of the problem of
computing the law of the total mass is that its moments diverge at any level of intermittency (inverse temperature) thereby rendering the associated
Stieltjes moment problem indeterminate. Moreover, while the negative integer moments are finite, their Stieltjes moment problem
can still be indeterminate as is the case of the Bacry-Muzy measure on the interval, for example. From this perspective, the simplest
measure is the Bacry-Muzy measure on the circle, whose negative integer moments do capture the distribution uniquely. Nonetheless,
the computation of the negative moments from first principles is very difficult and has only been achieved in the simplest
case of the Bacry-Muzy measure on the circle with the zero logarithmic potential corresponding to the positive moments being given by the special case of the Morris integral known as the Dyson integral, cf. the recent announcement in \cite{Remy}. Overall, the study of the law of the total mass is a highly non-classical moment problem
that requires novel mathematical techniques. In the rest of this section we will briefly explain the interest in GMC measures  and then review possible approaches to the problem of the total mass and summarize our key contributions. 

\subsection{GMC and Total Mass Problem} 
The theory of Gaussian Multiplicative Chaos (GMC) measures was conceived by Mandelbrot \cite{secondface}, 
who introduced the key ingredients of what is now known as GMC under the name of the limit lognormal measure, cf. also his review \cite{Lan}.
The mathematical foundation of the subject was laid down by Kahane \cite{K2}, who created a comprehensive, mathematically rigorous theory of (not necessarily gaussian) multiplicative chaos measures. The theory was advanced further
around 2000 with the introduction of the conical set construction by Barral and Mandelbrot \cite{Pulses} and Schmitt and Marsan \cite{Schmitt} and took its modern form with the theory of infinitely divisible multiplicative chaos measures of Bacry and Muzy \cite{BM1}, \cite{BM}. The theory of Bacry and Muzy was limited to multiplicative chaos on a finite interval. 
It has since been extended  to multiple dimensions by Robert and Vargas \cite{RVrevis}, who also relaxed Kahane's $\sigma-$positivity 
condition and proved universality of GMC, to other geometric shapes such as the circle by Fyodorov and Bouchaud \cite{FyoBou} and Astala \emph{et. al} \cite{Jones}, to complex GMC by Lacoin {\it et. al.} \cite{LRV}, 
as well as to critical GMC by Duplantier \emph{et. al} \cite{dupluntieratal} and Barral \emph{et. al.} \cite{barraletal}, and most recently to super-critical GMC by Madaule \emph{et. al.} \cite{Mad}. Most recently, Berestycki \cite{B}, Junnila and Saksman \cite{JS}, and Shamov \cite{Shamov} found new re-formulations and further extended the existing theory. 

The interest in GMC derives from its remarkable properties of multifractality and multiscaling, from inherent interest in Gaussian logarithmically 
correlated fields, cf. \cite{RVLog}, \cite{YO}, upon which GMC measures are built, from the complexity of
mathematical problems that their stochastic dependence poses, and from the many applications in mathematical
and theoretical physics and pure mathematics, in which GMC naturally appears. Without aiming for comprehension, we can mention applications to five areas: (1) conformal field theory and Liouville quantum gravity \cite{Jones}, \cite{BenSch}, \cite{DS}, \cite{RV3}, \cite{RV1},  (2) statistical mechanics of disordered energy landscapes \cite{Cao}, \cite{Cao2}, \cite{ClD}, 
\cite{Fyo10}, \cite{FyoBou},  \cite{YO}, \cite{FLDR},  \cite{FLDR2}, \cite{Me16}, 
(3) random matrix theory  \cite{FKS}, \cite{FyodSimm},  \cite{Hughes}, \cite{joint},  \cite{Webb}, (4)  statistical modeling of fully intermittent turbulence  \cite{CD95}, \cite{CGH}, \cite{Chainais}, \cite{Frisch}, \cite{Novikov}, \cite{RVrevis}, 
(5) conjectured \cite{FHK}, \cite{YK}, \cite{Menon} and some rigorous \cite{SW} applications to the behavior of the Riemann zeta function on the critical line. 

A fundamental open problem in the theory of GMC is to calculate the distribution of the total mass of the chaos measure and, more generally,
understand its stochastic dependence structure, \emph{i.e.} the joint distribution of the measure of several subsets of the set, on which it is defined. 
The significance of this problem is due to the fact that in most of the aforementioned 
areas the objects of interest can be reduced to questions about the total mass. We will illustrate it with two examples. 
As it was first discovered in \cite{FyoBou} and \cite{FLDR},
conjectured laws of the total mass of the Bacry-Muzy GMC measures on the circle and interval yield, under the hypothesis of freezing, precise asymptotic distributions of extremes of the underlying gaussian fields, which are restrictions of the 2D Gaussian Free Field to these geometries.  
It was subsequently proved rigorously in \cite{Madmax} that the distribution of the maximum of the logarithmically correlated gaussian field underlying the general GMC construction is determined by the law of the total mass of the corresponding critical GMC. As a second example,
fluctuations of mesoscopic counting statistics that converge to $\mathcal{H}^{1/2}$ gaussian fields can be rigorously quantified by means of the law of the total mass of the corresponding GMC measure, thereby connecting random matrix and GMC theories, cf. \cite{joint}.  We refer the interested reader to \cite{RV2} for a general review of GMC.

\subsection{Three Known Approaches}
\noindent There are three known approaches to the problem of the total mass. 
The first approach, pioneered by Fyodorov and Bouchaud \cite{FyoBou} and Fyodorov \emph{et. al.} \cite{FLDR} and
presented in the general case by Fyodorov and Le Doussal \cite{FLD}, is to heuristically extend the known positive integer moments to the complex plane,
\emph{i.e.} construct a function of a complex variable whose restriction to the finite interval of positive integers, where the moments are finite, coincides with the moments, and thereby
guess the Mellin transform of the total mass. This is particularly simple for the Bacry-Muzy GMC on the circle with the zero potential as the extension of the Dyson integral from the integer to complex dimension is elementary. This method also works on the interval although the extension of the Selberg
integral from the integer to complex dimension is substantially more difficult. The primary theoretical limitation of this approach is that it operates on the moments
directly and the moment problem is not determinate. The secondary limitation is that it is not obvious that the so-constructed Mellin
transform is in fact the Mellin transform of a probability distribution. While it is easy to show that the analytic extension of the Dyson
integral obtained in \cite{FyoBou} is the Mellin transform of a Fr\'echet distribution, it is much more difficult to establish the probabilistic property of
the analytic extension of the Selberg integral found in \cite{FLDR} so that the authors of that work limited themselves to numerical evidence. 
Nonetheless, from the computational standpoint, this method is particularly efficient. 

The second approach, which we introduced in \cite{Me2} and developed in \cite{Me3}--\cite{Me16}, is based on the formalism of intermittency differentiation and renormalization.
The rule of intermittency differentiation expresses the intermittency derivative of a general class of functionals of the total mass in the form
of an exact, non-local equation (or infinite hierarchy of local equations). It allows to compute the full high-temperature (low intermittency) expansion of the Mellin transform of the total mass of a general GMC measure
in terms of the expansion of positive integer moments in intermittency and effectively reconstruct the Mellin transform by summing the intermittency expansion. In the special cases of the Bacry-Muzy measures on the interval and circle with a logarithmic potential we carried out these computations explicitly by means of
Hardy's moment constant method
and proved that the resulting expressions\footnote{Hardy's method produces expressions for the logarithm of
the Mellin transform in the integral form, which are more cumbersome than the expressions for the Mellin transform itself in terms
of Barnes double gamma factors that are produced by the method of Fyodorov \emph{et. al.} \cite{FLDR}. The integral expressions
however are easier to bring to the L\'evy-Khinchine form and thus prove the probabilistic property of the construction. 
This constitutes the analytical proof of existence. It is possible to give a purely probabilistic proof that is based on the Mellin transform directly,
which requires the machinery of Barnes beta distributions. We review both proofs in this  paper. }
are Mellin transforms of valid probability distributions, known as the Selberg and Morris integral probability distributions, respectively. These distributions have the properties that their positive integer moments are given by the Selberg and Morris integrals,
\emph{i.e.} match the moments of the total mass, and that the asymptotic expansions of their Mellin transforms in intermittency coincide with the intermittency expansion. The principal computational limitation of this approach is that it too requires the explicit knowledge of the moments.
It is unknown whether the intermittency expansion of the Mellin transform captures the distribution of the total mass uniquely. If true, this approach
would provide a solution to the moment problem of the total mass. The answer depends on detailed, non-perturbative analysis of intermittency differentiation equations and requires novel mathematical tools. We refer the interested reader to \cite{MeIntRen} for an in-depth discussion
of theoretical aspects of our approach. We also note that the intermittency differentiation approach is not limited to 1D or Bacry-Muzy measures or even GMC measures but in fact applies to all infinitely divisible multiplicative chaos measures, cf. \cite{Me17}. 

In summary, both the first and second approaches succeeded in constructing good candidates for the distribution 
of the total mass of the Bacry-Muzy measures in the sub-critical regime.  The formulas for the Mellin transform that are 
produced by both methods are known to be the same. 

The third and most recent approach introduced by Remy \cite{Remy} is based on the connection between GMC and 
Liouville conformal field theory that was established by David \emph{et. al.} \cite{David}. 
It interprets negative integer moments of the total mass of the Bacry-Muzy GMC on the circle 
in terms of one-point correlation functions of Liouville conformal field theory on the unit disk and thereby computes these moments from first principles. Unlike the first two approaches, the approach of Remy \cite{Remy} is mathematically rigorous. 
A priori, it appears that this approach is limited to low-dimensional GMC measures that are connected
with the Liouville theory and further have a determinate Stieltjes moment problem for the negative moments such as the Bacry-Muzy measure 
on the circle. In particular, in its current form it does not apply to the Bacry-Muzy measure on the interval as the Stieltjes moment problem
for the negative moments of its total mass is indeterminate. It should be stressed that the
result of Remy \cite{Remy} has not resolved our conjecture about the distribution of the total mass of the Bacry-Muzy GMC on the circle because
we allow for a non-random logarithmic potential. Remy \cite{Remy} sets it to zero, which greatly simplifies the 
total mass distribution as it reduces it to the Fr\'echet factor as predicted by  Fyodorov and Bouchaud \cite{FyoBou}, whereas the full distribution 
was conjectured in \cite{Me16} to have a completely non-trivial Barnes beta factor, cf. Sections \ref{CirAnalytical} and \ref{Probabilistic}. 


\subsection{Summary of Results and Plan of the Paper}
The scope of this paper is the review of mathematically rigorous results on the existence and properties of the Selberg and Morris integral probability distributions. We review both the analytical and probabilistic constructions of these distributions in the sub-critical and critical regimes in detail. 
In particular, we review the theory of Barnes beta distributions that provide basic building blocks of the Selberg and Morris integral distributions.
We also review the analytic continuation of the complex Selberg integral (Dotsenko-Fateev integral). 
The theory is illustrated with several conjectured applications. 

The analytical construction is based on the three known representations of the Mellin transform: a finite product of ratios of Barnes double gamma
factors, a regularized infinite product of ratios of Euler's gamma factors, and a L\'evy-Khinchine integral representation of the logarithm of the Mellin
transform. The Barnes double gamma and infinite product representations lead to the 
computation of the negative moments and furnish simple proofs that the Mellin transform is the analytic continuation of the Selberg/Morris integrals. The Barnes double gamma representation also provides the asymptotic expansion of the Mellin transform, which is shown to match the intermittency expansion that follows from the known formulas for the positive integer moments. 
The integral representation of the logarithm of the Mellin
transform furnishes the analytical proof of the fact that the Mellin transform is in fact the Mellin  transform of a log-infinitely divisible probability distribution, whose gaussian and L\'evy components are computed explicitly. 

The method of deriving the analytic continuation of the Selberg integral in the approach of Fyodorov \emph{et. al.} \cite{FLDR}
is more computationally efficient than our method of summing the intermittency expansion, cf. \cite{Me4}. For this reason, we
use their approach in the construction of the analytic continuation, 
and also apply it in deriving the analytic
continuations of the Morris and complex Selberg integrals. Afterwards, to be consistent with our approach, we check that the high-temperature asymptotic expansions of the analytic continuations of the Selberg and Morris integrals coincide with the corresponding intermittency expansions.
The approach of Fyodorov \emph{et. al.} is based on the novel idea of using Barnes-like double gamma function that is popular in the physics
literature. As this function is not standard in mathematics, we chose, beginning with \cite{MeIMRN}, to use the standard double gamma
function and its cousins instead. We  believe that the wealth of its known properties makes  its use more advantageous.

The probabilistic construction is based on the theory of Barnes beta probability distributions, which we introduced in \cite{MeIMRN} and developed in  
\cite{Me13}, \cite{Me14}, and \cite{Me16}.
The Barnes beta distributions constitute a novel family of log-infinitely divisible probability distributions having the property that their Mellin transform is defined in the form of an intertwining product of ratios of multiple Barnes gamma factors. We review their remarkable properties and show
that in the special case of Barnes beta distributions corresponding to the double gamma function these distributions provide
the second, purely probabilistic proof of the existence of Selberg and Morris integral probability distributions. Moreover, we obtain their
explicit factorizations: the Selberg integral distribution is the product of independent lognormal, Fr\'echet, and three Barnes beta distributions,
the Morris integral distribution is the product of independent Fr\'echet and a single Barnes beta. 

The two constructions rely on special properties of the double gamma function such as functional equations, Barnes and Shintani infinite
factorizations, scaling invariance, Barnes multiplication, integral representations of its logarithm, and asymptotic expansions. We review these 
properties in some detail to make the paper accessible to a wider audience, especially as the double gamma function has multiple known normalizations (classical, Alexeiewsky-Barnes, Ruijsenaars) that are all useful in different contexts. Nonetheless, we restrict ourselves to giving the relevant
formulas without providing detailed proofs from the theory of  multiple gamma functions as that would take us too far afield. 

We illustrate our theoretical results on the Morris and Selberg integral distributions with three types of applications. Following
Fyodorov and Bouchaud \cite{FyoBou} and Fyodorov \emph{et. al.} \cite{FLDR}, who discovered the connection between
the asymptotic distribution of the maximum of the 2D  Gaussian Free Field restricted to the circle and interval and the laws
of the total mass of Bacry-Muzy measure for these geometries, we give a probabilistic re-formulation\footnote{In the circle case
we also extend the original conjecture in \cite{FyoBou}, which was restricted to the Dyson integral.}
 of their results in terms
of the conjectured laws of the corresponding derivative martingales. Our second application has to do with the computation
of the inverse participation ratios of the Fyodorov-Bouchaud model that were known previously only by means of a heuristic analytic continuation 
of the Morris integral to negative dimensions, cf. \cite{Fyo09}. Our result on the conjectured law of the Bacry-Muzy
measure on the circle with a logarithmic potential allows us to treat this continuation rigorously. In the third application we conjecture
several mod-Gaussian limit theorems. The concept of mod-Gaussian convergence as a means of precisely quantifying divergent sequences
of random variables was first introduced by Keating and Snaith \cite{KeaSna} and was formalized and developed into 
a powerful mathematical tool by Jacod \emph{et. al.} \cite{Jacod}, see also \cite{Feray} and \cite{Meliot}. 
The idea of associating such theorems with GMC was first introduced in \cite{Menon}
in the context of mesoscopic statistics of Riemann zeroes. We review some of those results and show how they
can be combined with the methods of Fyodorov \emph{et. al.} to conjecture a mod-Gaussian limit theorem
for the distribution of the maximum of the centered Gaussian Free Field on the circle and interval, also known as the Fractional Brownian 
motion with Hurst index $H=0,$ cf. \cite{FKS}.

This paper is largely a review of results that have already appeared elsewhere. The only new results are those on the analytic continuation
of the complex Selberg integral, 
the computation of the inverse participation ratios of the Fyodorov-Bouchaud model, and the conjectured 
mod-Gaussian theorems for the centered Gaussian Free Field. 

The plan of the rest of the paper is as follows. In Section 2 we briefly recall Bacry-Muzy measures and then state the problem of 
total mass for them. In Section 3 we review multiple gamma functions. In Sections 4 and 5 we state our analytical results on the Morris
and Selberg integral probability distributions, respectively, followed by the proofs in Section 6. In Section 7 we present the theory of Barnes beta probability distributions. In Section 8 we state our probabilistic results on the Morris and Selberg integral probability distributions, followed by
the proofs in Section 9. In Section 10 we give our results on the critical Morris and Selberg integral distributions. In Section 11 we give the 
analytic continuation of the complex Selberg integral. In Section 12 we give some applications of our results on conjectured laws of Bacry-Muzy measures. Conclusions are given in Section 13. The Appendix gives proofs of results on Barnes beta distributions. 

Our results are mathematically rigorous except in Section \ref{SomeApplications}.


\section{Bacry-Muzy GMC and Total Mass Problem}\label{Problem}
In this section we will informally recall the construction of Bacry-Muzy measures on the interval and circle and pose the specific version
of the total mass problem that we will be considering in the rest of the paper. 

Following \cite{MRW}, define a centered gaussian process with the covariance
\begin{align}
{\bf{Cov}}\left[V_{\varepsilon}(u), \,V_{\varepsilon}
(v)\right]  = & 
\begin{cases}\label{covk}
 -
2 \, \log|u-v|, \, \varepsilon < |u-v|\leq 1,  \\
- 2\log\varepsilon,\, u=v,
\end{cases} 
\end{align}
Let $0\leq \beta<1.$ 
The theorem of Bacry and Muzy states that the regularized exponential functional of this field converges weakly a.s. as $\varepsilon\rightarrow 0$ to
a non-degenerate limit random measure, called the Bacry-Muzy GMC measure on the interval,
\begin{gather}
e^{\beta^2\log\varepsilon}\int_a^b e^{\beta V_{\varepsilon}(u)} \, du\longrightarrow M_{\beta}(a, b), \label{chaosinterval} \\ 
{\bf{E}}[M_{\beta}(a, b)]=|b-a|.
\end{gather}
It is worth emphasizing that the choice of covariance regularization for $|u-v|\leq \varepsilon$ has no effect on the law of the total mass,
see the proof of universality in \cite{RVrevis}.
The moments of the total mass of the limit measure with a logarithmic potential are given by the Selberg integral: let $n<1/\beta^2,$ 
\begin{equation}
{\bf{E}} \Bigl[\Bigl(\int_0^1 s^{\lambda_1}(1-s)^{\lambda_2} \,
M_\beta(ds)\Bigr)^n\Bigr] =  \int\limits_{[0,\,1]^n} \prod_{i=1}^n
s_i^{\lambda_1}(1-s_i)^{\lambda_2} \prod_{i<j}^n
|s_i-s_j|^{-2\beta^2} ds_1\cdots ds_n.
\end{equation}
Define the quantity 
\begin{equation}
\tau = \frac{1}{\beta^2} > 1.
\end{equation}
In statistical physics one thinks of $\beta$ as the inverse temperature. 
Recall the classical Selberg integral,
\begin{equation}
\int\limits_{[0,\,1]^n} \prod_{i=1}^n s_i^{\lambda_1}(1-s_i)^{\lambda_2}\, \prod\limits_{i<j}^n |s_i-s_j|^{-2/\tau} ds_1\cdots ds_n =  \prod_{k=0}^{n-1}\frac{\Gamma(1-(k+1)/\tau)
\Gamma(1+\lambda_1-k/\tau)\Gamma(1+\lambda_2-k/\tau)}
{\Gamma(1-1/\tau)\Gamma(2+\lambda_1+\lambda_2-(n+k-1)/\tau)}, \label{Selberg}
\end{equation}
cf. \cite{ForresterBook} for a modern treatment. We will assume for simplicity that $\lambda_i \geq 0.$ The integral is convergent for $n<\tau.$
The Bacry-Muzy GMC on the circle is a periodized version of the Bacry-Muzy measure on the interval. It was first considered heuristically
in \cite{FyoBou} and formalized in \cite{Jones}. Let $V_\varepsilon(\psi)$ be a centered gaussian process with the covariance
\begin{align}
{\bf{Cov}}\left[V_{\varepsilon}(\psi), \,V_{\varepsilon}
(\xi)\right]  = &
\begin{cases}\label{covkc}
 -
2 \, \log|e^{2\pi i\psi}-e^{2\pi i\xi}|, \,  |\xi-\psi|> \varepsilon,  \\
-2\log\varepsilon, \psi=\xi,
\end{cases} 
\end{align}
Once again,
the regularized exponential functional of this field converges weakly a.s. as $\varepsilon\rightarrow 0$ to
a non-degenerate limit random measure, which we refer to as the Bacry-Muzy GMC measure on the circle.
Let $0\leq \beta<1.$ 
\begin{gather}
e^{\beta^2 \log\varepsilon}\int_\phi^\psi e^{\beta V_{\varepsilon}(\theta)} \, d\theta\longrightarrow M_{\beta}(\phi, \psi), \label{chaoscircle} \\
{\bf{E}}[M_{\beta}(\phi, \psi)]=|\psi-\phi|.
\end{gather}
The moments of the total mass of the limit measure with a logarithmic potential are given by the Morris integral: let $n<1/\beta^2,$ 
\begin{align}
{\bf{E}} \Bigl[\Bigl( \int_{[-\frac{1}{2},\,\frac{1}{2}]} e^{2\pi i\psi \frac{\lambda_1-\lambda_2}{2}} \, |1+e^{2\pi i\psi}|^{\lambda_1+\lambda_2}\, M_{\beta}(d\psi)\Bigr)^n\Bigr] = & \int\limits_{[-\frac{1}{2},\,\frac{1}{2}]^n} \prod\limits_{l=1}^n  e^{2\pi i \theta_l\frac{\lambda_1-\lambda_2}{2}}  |1+e^{2\pi  i\theta_l}|^{\lambda_1+\lambda_2}\times \nonumber\\
& \times
\prod\limits_{k<l}^n |e^{2\pi  i \theta_k}-e^{2\pi i\theta_l}|^{-2\beta^2} \,d\theta.
\end{align}
Recall the Morris integral, 
see Chapter 4 of \cite{ForresterBook}.
\begin{gather}
\int\limits_{[-\frac{1}{2},\,\frac{1}{2}]^n} \prod\limits_{l=1}^n  e^{ \pi i \theta_l(\lambda_1-\lambda_2)}  |1+e^{2\pi  i\theta_l}|^{\lambda_1+\lambda_2} \,
\prod\limits_{k<l}^n |e^{2\pi i \theta_k}-e^{2\pi i\theta_l}|^{-2/\tau} \,d\theta, \nonumber \\  = \prod\limits_{j=0}^{n-1} \frac{\Gamma(1+\lambda_1+\lambda_2- \frac{j}{\tau})\,\Gamma(1-\frac{(j+1)}{\tau})}{\Gamma(1+\lambda_1- \frac{j}{\tau})\,\Gamma(1+\lambda_2- \frac{j}{\tau})\,\Gamma(1-\frac{1}{\tau})} \label{morris2}.
\end{gather}
We will restrict our attention to a special case of the total mass problem on the circle corresponding to
\begin{equation}
\lambda_1=\lambda_2=\lambda\geq 0.
\end{equation}
In this case the moments of the total mass are given by 
\begin{align}
{\bf{E}} \Bigl[\Bigl( \int_{[-\frac{1}{2},\,\frac{1}{2}]} |1+e^{2\pi i\psi}|^{2\lambda}\, M_{\beta}(d\psi)\Bigr)^n\Bigr] = & \int\limits_{[-\frac{1}{2},\,\frac{1}{2}]^n} \prod\limits_{l=1}^n  |1+e^{ 2\pi i\theta_l}|^{2\lambda} \,
\prod\limits_{k<l}^n |e^{2\pi  i \theta_k}-e^{2\pi  i\theta_l}|^{-2/\tau} \,d\theta, \nonumber \\
= & \prod\limits_{j=0}^{n-1} \frac{\Gamma(1+2\lambda- \frac{j}{\tau})\,\Gamma(1-\frac{(j+1)}{\tau})}{\Gamma(1+\lambda - \frac{j}{\tau})^2\,\Gamma(1-\frac{1}{\tau})}.
\label{momlambda}
\end{align}
In the special case of $\lambda=0$ this integral is known as the Dyson integral and has a much simpler evaluation corresponding to the moments of
the Fyodorov-Bouchaud model. 

In the rest of the paper we will construct probability distributions having the properties
that they match the positive integer moments
specified in Eqs. \eqref{Selberg} and \eqref{momlambda} and \emph{intermittency expansions} of the total mass distributions. To state this precisely, we need to 
briefly remind the reader of the key results of the intermittency renormalization formalism, cf. \cite{Me2}, \cite{Me3}, \cite{MeIMRN}, \cite{MeIntRen}.
Let us write the total mass in the form
\begin{equation}
M_\beta[\varphi](\mathcal{D}) = \int\limits_\mathcal{D} \varphi(x) \, M_\beta(dx)\, dx,
\end{equation}
so that
\begin{align}
\mathcal{D} & = [0,1],\; \varphi(x)  = x^{\lambda_1}(1-x)^{\lambda_2}\; (interval), \\
\mathcal{D} & = [-\frac{1}{2},\,\frac{1}{2}],\; \varphi(x)  = |1+e^{2\pi ix}|^{2\lambda}\; (circle),
\end{align}
and introduce the quantity 
\begin{equation}
\bar{\varphi} \triangleq \int\limits_\mathcal{D} \varphi(x) \,dx.
\end{equation}
Finally, we need to introduce the quantity called the intermittency by
\begin{equation}
\mu = 2\beta^2 = \frac{2}{\tau}.
\end{equation}
Then, the log-moments of the total mass have the representation in the form, 
\begin{equation}\label{cpl}
\log{\bf{E}} \Bigl[\Bigl( \int_\mathcal{D} \varphi(x)\, M_{\beta}(dx)\Bigr)^n\Bigr] = n\log \bar{\varphi}+ \sum\limits_{p=1}^\infty c_p(n) \mu^p,
\end{equation}
for some coefficients $c_p(n)$ that are known to be \emph{polynomial} in the moment order $n,$ cf. \cite{MeIntRen}.
One of the key results of the intermittency renormalization formalism is that the full intermittency expansion (formal power series expansion)
of the Mellin transform of the total mass of a general GMC measure can be calculated in closed form in terms of 
the expansion of the log-moments in intermittency by the following formula, 
\begin{equation}\label{intermittencyMellin}
{\bf E}\Bigl[\Bigl(\int_\mathcal{D} \varphi(x)
\,M_\beta(dx)\Bigr)^q\Bigr]=\bar{\varphi}^q \exp\Bigl(\sum_{p=1}^\infty
\mu^{p} \,c_p(q)\Bigr), \; \Re(q)<\tau. 
\end{equation}
It is naturally interpreted as the asymptotic expansion in the limit of low intermittency (high temperature). 
The coefficients $c_p(n)$ are known explicitly for the Bacry-Muzy measures.
We have on the interval,
\begin{gather}
c_p(n) =
\frac{1}{p2^p} 
\Bigl[ \bigl(\zeta(p,
1+\lambda_1)+\zeta(p,
1+\lambda_2)\bigr)\Bigl(\frac{B_{p+1}(q)-B_{p+1}}{p+1}\Bigr)
-\zeta(p)n + \zeta(p)\times \nonumber \\ \times
\Bigl(\frac{B_{p+1}(n+1)-B_{p+1}}{p+1}\Bigr) - \zeta(p,
2+\lambda_1+\lambda_2)
\Bigl(\frac{B_{p+1}(2q-1)-B_{p+1}(q-1)}{p+1}\Bigr)\Bigr], \label{cpninterval}
\end{gather}
and on the circle,
\begin{gather}
c_p(n) = \frac{1}{p2^p} \Bigl[\bigl(\zeta(p,\,1+2\lambda)-2\zeta(p, 1+\lambda)\bigr)\frac{B_{p+1}(n)-B_{p+1}}{p+1}+
\zeta(p)\frac{B_{p+1}(n+1)-B_{p+1}}{p+1}-n\zeta(p)\Bigr]. \label{cpncircle}
\end{gather}
As usual, $B_n(s)$ denotes the $n$th Bernoulli polynomial, $\zeta(s,
a)$ the Hurwitz zeta function, $\zeta(s)$ the Riemann zeta function, and $\zeta(1, a)\triangleq -\psi(a)$
the digamma function. These formulas are elementary corollaries of  Eqs. \eqref{Selberg} and \eqref{momlambda} and
the following summation formulas,
\begin{align}
\log\Gamma(a+z) = & \log\Gamma(a)+\sum\limits_{p=1}^\infty \frac{(-z)^p}{p} \zeta(p,a), \\
\sum\limits_{j=x}^y j^p = &\frac{B_{p+1}(y+1)-B_{p+1}(x)}{p+1}.
\end{align}

Given these preliminaries, we can now give a precise statement of the problem that is reviewed in the rest of the paper, which is our version
of the moment problem for the total mass. The problem is to construct and describe properties of positive probability distributions that have positive integer moments given by Eqs. \eqref{Selberg} and \eqref{momlambda} and whose asymptotic expansion of the Mellin transform in intermittency coincides with the expansion in Eq. \eqref{cpl} with the $c_p(n)$ coefficients specified in Eqs. \eqref{cpninterval} and \eqref{cpncircle}, respectively.

\section{A Review of Barnes Double Gamma Function}\label{BarnesReview}
In this section we review several formulations of the multiple gamma functions of Barnes with an emphasis
on the double gamma function. 

In general, let $a=(a_1,\cdots, a_M),$ $M\in \mathbb{N},$ $a_i>0 $ $\forall i=1\cdots M.$
The multiple gamma function of Barnes $\Gamma_{M}(z\,|\,a)$ is a meromorphic function of $z\in\mathbb{C}$ that satisfies $M$ functional equations,
\begin{equation}\label{feq}
\Gamma_{M}(z\,|\,a) = \Gamma_{M-1}(z\,|\,\hat{a}_i)\,\Gamma_M\bigl(z+a_i\,|\,a\bigr),\,i=1\cdots
M, 
\end{equation}
$\hat{a}_i = (a_1,\cdots, a_{i-1},\,a_{i+1},\cdots, a_{M}),$ and
\begin{align}
\Gamma_0(z) = & \frac{1}{z},  \\
\Gamma_1(z\,|\,\tau) = & \frac{\tau^{z/\tau-1/2}}{\sqrt{2\pi}} \,\Gamma\bigl(\frac{z}{\tau}\bigr), \label{gamma1}
\end{align}
where $\Gamma(z)$ is Euler's gamma function, cf. \cite{mBarnes}.
By iterating Eq. \eqref{feq} one sees that $\Gamma_M(z\,|\,a)$ is
meromorphic over $z\in\mathbb{C}$ having no zeroes and poles at
\begin{equation}\label{poles}
z=-(k_1 a_1+\cdots + k_M a_M),\; k_1\cdots k_M\in\mathbb{N},
\end{equation}
with multiplicity equal the number of $M-$tuples $(k_1, \cdots,
k_M)$ that satisfy Eq. \eqref{poles}.
The case of $M=2$ is referred to as the double gamma function. The fundamental functional equations then
take on the form, cf. \cite{Double},
\begin{gather}
\frac{\Gamma_2(z\,|\,a_1,a_2)}{\Gamma_2(z+a_1\,|\,a_1,a_2)}
= \frac{a_2^{z/a_2-1/2}}
{\sqrt{2\pi}}\Gamma\bigl(\frac{z}{a_2}\bigr), \\
\frac{\Gamma_2(z\,|\,a_1,a_2)}{\Gamma_2(z+a_2\,|\,a_1,a_2)}
= \frac{a_1^{z/a_1-1/2}}
{\sqrt{2\pi}}\Gamma\bigl(\frac{z}{a_1}\bigr),
\end{gather}
which have the following useful corollary that is used repeatedly below. Let $a=(1,\tau)$ and $k\in\mathbb{N}.$ Then,
\begin{align}
\frac{\Gamma_2(z+1-k\,|\,1,\tau)}{\Gamma_2(z+1\,|\,1,\tau)} = & \prod\limits_{j=0}^{k-1}
\Gamma_1\bigl(z-j\,|\,1,\tau\bigr), \nonumber \\
= & 
\bigl(\frac{1}{2\pi\tau}\bigr)^{k/2} \tau^{\sum\limits_{j=0}^{k-1} (z-j)/\tau}\;
\prod\limits_{j=0}^{k-1} \Gamma\bigl(\frac{z}{\tau}-\frac{j}{\tau}\bigr). \label{repeated}
\end{align}
The multiple gamma function is defined classically by
\begin{equation}\label{barnes}
\Gamma^{-1}_M(z\,|\,a) = e^{P(z\,|\,a)}\,w\prod\limits_{n_1,\cdots , n_M=0}^\infty
{}' \Bigl(1+\frac{z}{\Omega}\Bigr)\exp\Bigl(\sum_{k=1}^M \frac{(-1)^k}{k}\frac{z^k}{\Omega^k}\Bigr), 
\end{equation}
where $\Re(z)>0,$ $P(z\,|\,a)$ is a polynomial in $z$ of degree $M$ that depends on one's choice of
normalization, 
\begin{equation}\label{Omega}
\Omega\triangleq \sum_{i=1}^M n_i \, a_i,
\end{equation}
and the prime indicates that the product is over all indices except $n_1=\cdots =n_M=0.$

The \emph{classical} normalization condition for the double gamma function that was used by Barnes is
\begin{equation}\label{classical}
\lim\limits_{z\rightarrow 0}\,
\Bigl[z\,\Gamma_2(z\,|\,a_1,a_2)\Bigr] = 1.
\end{equation}
In this normalization the double gamma function is closely related to the so-called Alexeiewsky-Barnes $G(z\,|\,\tau)$ function, which was historically
introduced first, cf. \cite{A} and \cite{Genesis}.
The function $G(z\,|\,\tau)$ is defined for
$z\in\mathbb{C}$ and $\tau\in\mathbb{C}$ such that
$|\arg(\tau)|<\pi$ 
and 
satisfies the following
normalization and functional equations.
\begin{align}
G(z=1\,|\,\tau)= & 1, \\
G(z+1\,|\,\tau) = & \Gamma\Big(\frac{z}{\tau}\Bigr)\,G(z\,|\,\tau), \label{Gfunct1} \\ 
G(z+\tau\,|\,\tau) = &
(2\pi)^{\frac{\tau-1}{2}}\,\tau^{-z+\frac{1}{2}}\,\Gamma(z)\,
G(z\,|\,\tau).
\end{align}
$G(z\,|\,\tau)$ is a meromorphic function of $z$ with no poles and
roots at $z=-(m\tau+n),$ $m,n=0, 1, 2,\cdots.$ 
The relationship between the double gamma function in the classical normalization and the Alexeiewsky-Barnes
function was established by Barnes in \cite{Double}. 
Using Barnes' notation, define the function 
\begin{equation}
{}_2 S_0 (z\,|\,a_1, a_2) \triangleq \frac{z^2-z(a_1+a_2)}
{2a_1 a_2}.
\end{equation}
Then,
\begin{equation}
\Gamma_2^{-1}(z\,|\,a_1,a_2) = (2\pi)^{-z/2a_1}
a_2^{1+{}_2 S_0 (z\,|\,a_1, a_2)}
G\Bigl(\frac{z}{a_1}\,\Big|\,\frac{a_2}{a_1}\Bigr).
\end{equation}
Equivalently, we can write
\begin{equation}\label{GfromG2}
G(z\,|\,\tau) = (2\pi)^{z/2} \tau^{-\bigl(1+{}_2
S_0(z,|\,1,\tau)\bigr)}\,\Gamma_2^{-1}(z\,|\,1,\tau),
\end{equation}
so that $G(z\,|\,\tau)$ is, up to normalization, the reciprocal of the double gamma function with parameters $(1,\tau).$
It is an elementary exercise to check that the normalization conditions and functional equations of the double gamma function are equivalent to
those of the $G(z\,|\,\tau)$ function. Moreover, $\Gamma_2(z\,|\,a_1,a_2)$ is symmetric in $(a_1, a_2).$

The function $G(z\,|\,\tau)$ has the following useful properties. The first is an immediate corollary of the functional equation
in Eq. \eqref{Gfunct1}.
\begin{equation}\label{Grepeated}
\frac{G(1+z\,|\,\tau)}{G(1+z-k\,|\,\tau)} = \prod\limits_{j=0}^{k-1} \Gamma\bigl(\frac{z-j}{\tau}\bigr).
\end{equation}
We note that the use of the physicists' equivalent of Eq. \eqref{Grepeated} in constructing analytic continuations was pioneered by
Fyodorov \emph{et. al.} \cite{FLDR} and constitutes the core of their method.  
The second property is the integral representation of the logarithm in the form of 
a Malmst$\acute{\text{e}}$n-type formula due to Lawrie and King \cite{LawKing}, see also \cite{MeIMRN} for an elementary derivation
based on \cite{Genesis}. Given $\Re(z),\,\Re(\tau)>0,$
\begin{equation}
\log G(z\,|\,\tau) = \int\limits_0^\infty \frac{dt}{t} \Bigl[
\frac{1-z}{e^{t\tau}-1}+(1-z)
e^{-t\tau}+(z^2-z)\frac{e^{-t\tau}}{2\tau}+\frac{1-e^{-t(z-1)}}{(e^t-1)(1-e^{-t\tau})}\Bigr]. \label{LK}
\end{equation}
The third is the asymptotic expansion due to Billingham and King
\cite{BillKing}. We cite it here in a simplified form, which is sufficient for our needs.
\begin{equation}
\log G(z\,|\,\tau) = \frac{z^2}{2\tau}\log(z) - \frac{z^2}{\tau}
\left(\frac{3}{4}+\frac{\log\tau}{2}\right) -
\frac{1}{2}\left(\frac{1}{\tau}+1\right)z\log z + O(z),\,\,
z\rightarrow\infty, \,\,|\arg(z/\tau)|<\pi. \label{BillKingAsymp}
\end{equation}
The fourth property is the integral representation and asymptotic expansion of the logarithm of the ratio of two $G(z\,|\,\tau)$ functions due to 
\cite{MeIMRN}.
Let $\Re(q)<1+a+\tau,$ $a>-\tau,$ and $\tau>0.$ Define the function
\begin{equation}
I(q\,|\,a, \tau) \triangleq \int\limits_0^\infty
\frac{dx}{x}\frac{e^{-ax}}{e^{x\tau}-1}
\Bigl[\frac{e^{xq}-1}{e^{x}-1}
 -q-\frac{(q^2-q)}{2}x\Bigr].\label{Iintegral}
\end{equation}
Then,  
\begin{equation}\label{IfuncG}
I(q\,|\,a, \tau) =
\log\frac{G(1+a+\tau\,|\,\tau)}{G(1-q+a+\tau\,|\,\tau)}
-q\log\Bigl[\Gamma\bigl(1+\frac{a}{\tau}\bigr)\Bigr]+
\frac{(q^2-q)}{2\tau}\psi\bigl(1+\frac{a}{\tau}\bigr),
\end{equation}
and $I(q\,|\, a\tau, \tau)$ has the asymptotic expansion 
\begin{equation}\label{IfuncGAsymptotic}
I(q\,|\,a\tau, \tau) \thicksim \sum\limits_{r=1}^\infty \frac{\zeta(r+1,
\,1+a)}{r+1}\Bigl(\frac{B_{r+2}(q)-B_{r+2}}{r+2}\Bigr)/\tau^{r+1}
\end{equation}
in the limit $\tau\rightarrow +\infty,$ which follows from a slight extension of Ramanujan's generalization
of Watson's lemma, confer Lemma 10.2 in Chap. 38 of \cite{Berndt}, see \cite{MeIMRN}  for details.
As a corollary of Eq. \eqref{IfuncG} and the classical identities, cf. \cite{Temme},
\begin{equation}\label{Malmsten} 
\log\Gamma(1+s) =\int\limits_0^\infty
\Bigl(\frac{e^{-ts}-1}{e^t-1}+se^{-t}\Bigr) \frac{dt}{t}
\,\,\,\,\text{(Malmst$\acute{\text{e}}$n)}, \,\,\Re(s)>-1,
\end{equation}
\begin{equation}\label{Frullani} 
\log(s) = \int\limits_0^\infty
\bigl(e^{-t}-e^{-ts}\bigr)\frac{dt}{t}
\,\,\,\,\text{(Frullani)},\,\,\Re(s)>0.
\end{equation}
we established in \cite{MeIMRN} the following result.
Given  $b, c, d>0$ and $\Re(q)<b,$ there holds the identity
\begin{align}
\exp\left(\int\limits_0^\infty
\frac{dx}{x}\frac{e^{-bx}(1-e^{-cx})(1-e^{-dx})}{(1-e^{-x})(1-e^{-x\tau})}
(e^{xq}-1)\right) = &
\frac{G(b\,|\,\tau)}{G(b-q\,|\,\tau)}\frac{G(b-q+c\,|\,\tau)}{G(b+c\,|\,\tau)}\times \nonumber\\
&\times \frac{G(b-q+d\,|\,\tau)}{G(b+d\,|\,\tau)}
\frac{G(b+c+d\,|\,\tau)}{G(b-q+c+d\,|\,\tau)}.\label{Gratioidentity}
\end{align}
Finally, the last property is the Shintani factorization, cf. \cite{Shi}.
Given $z\in\mathbb{C}$ and
$|\arg(\tau)|<\pi,$ then
\begin{equation}\label{ShintaniG}
G(z+\tau\,|\,\tau) = (2\pi)^{\frac{(\tau-1)}{2}} \tau^{-\frac{1}{2}}
e^{\gamma\frac{(z-z^2)}{2\tau}} \prod\limits_{m=1}^\infty
(m\tau)^{z-1} e^{\frac{(z^2-z)}{2m\tau}} \frac{\Gamma(1+m\tau)}
{\Gamma(z+m\tau)},
\end{equation}
or, equivalently, in terms of the double gamma function in its classical renormalization,
\begin{equation}\label{ShintaniGamma}
\Gamma_2(z\,|\,1,\tau) = (2\pi)^{\frac{z}{2}}
\tau^{\frac{(z-z^2)}{2\tau} - \frac{z}{2}}
e^{\gamma\frac{(z^2-z)}{2\tau}} \Gamma(z) \prod\limits_{m=1}^\infty
(m\tau)^{1-z} e^{\frac{(z-z^2)}{2m\tau}} \frac{\Gamma(z+m\tau)}
{\Gamma(1+m\tau)}.
\end{equation}
We refer the interested reader to \cite{KataOhts} for a further discussion of the double and multiple gamma functions
in the classical normalization. 

We note that the well-known Barnes $G(z)$ function, confer \cite{Genesis} and \cite{SriCho}, is a special case of
the Alexeiewsky-Barnes
function. 
\begin{equation}\label{Gdef}
G(z) \triangleq G(z\,|\,\tau=1).
\end{equation}
It satisfies the functional equation
\begin{equation}
G(z+1) = \Gamma(z)\,G(z).
\end{equation}
It is known that $G(z)$ is entire with zeroes at $z=-n, \,n\in\mathbb{N},$ of
multiplicity $n+1.$

The \emph{modern} normalization of the double and, more generally, multiple gamma functions is due to Ruijsenaars \cite{Ruij}.
Its advantages over the classical approach is that it simplifies some formulas, notably Barnes multiplication and scaling, and streamlines integral representations and asymptotic expansions.
The approach of Ruijsenaars is based on multiple zeta functions. Define the function
\begin{equation}\label{fdef}
f(t) = t^M \prod\limits_{j=1}^M (1-e^{-a_j t})^{-1}
\end{equation}
for some integer $M\geq 0$ and parameters $a_j>0,$ $j=1\cdots M.$
Slightly
modifying the definition in \cite{Ruij}, we define multiple
Bernoulli polynomials for $a=(a_1,\cdots, a_M)$ by
\begin{equation}\label{Bdefa}
B_{M, m}(x\,|\,a) \triangleq \frac{d^m}{dt^m}\Big|_{t=0} \bigl[f(t)
e^{-xt}\bigr].
\end{equation}
The generalized zeta function is defined by
\begin{equation}\label{zdef}
\zeta_M(s, \,w\,|\,a) \triangleq \frac{1}{\Gamma(s)} \int\limits_0^\infty
t^{s-1} e^{-wt}\,f(t) \,\frac{dt}{t^M}, \,\,\Re(s)>M,\,\Re(w)>0.
\end{equation}
It is shown in \cite{Ruij} that $\zeta_M(s, \,w)$ has the analytic
continuation to a meromorphic function in $s\in\mathbb{C}$
with simple poles at $s=1, 2, \cdots M.$ The generalized log-multiple gamma
function is then defined by
\begin{equation}\label{Ldef}
L_M(w\,|\,a) \triangleq \partial_s \zeta_M(s, \,w\,|\,a)|_{s=0}, \,\,\Re(w)>0.
\end{equation}
It can be analytically continued to a function that is holomorphic
over $\mathbb{C}-(-\infty, 0].$ The key result of \cite{Ruij} 
is the following  Malmst\'en-type formula for $L_M(w\,|\,a).$
Let $\Re(w)>0.$
\begin{equation}\label{key}
L_M(w\,|\,a) = \int\limits_0^\infty \frac{dt}{t^{M+1}} \Bigl(
e^{-wt}\,f(t) - \sum\limits_{k=0}^{M-1} \frac{t^k}{k!}\,B_{M,k}(w\,|\,a)
- \frac{t^M\,e^{-t}}{M!}\, B_{M,M}(w\,|\,a)\Bigr).
\end{equation}
$L_M(w\,|\,a)$ satisfies the asymptotic expansion,
\begin{gather}
L_M(w\,|\,a) = -\frac{1}{M!} B_{M, M}(w\,|\,a)\,\log(w) + \sum\limits_{k=0}^M
\frac{B_{M,k}(0\,|\,a) (-w)^{M-k}}{k!(M-k)!}\sum\limits_{l=1}^{M-k}
\frac{1}{l} + R_M(w\,|\,a), \label{asym}\\
R_M(w\,|\,a) = O(w^{-1}), \,|w|\rightarrow\infty, \, |\arg(w)|<\pi.
\label{asymremainder}
\end{gather}
Now, it is not difficult to show that Eq.
\eqref{zdef} implies
\begin{equation}\label{mzdef}
\zeta_M\bigl(s,\,w\,|\,a\bigr) =
\sum\limits_{k_1,\cdots,k_M=0}^\infty \bigl(w+k_1 a_1+\cdots+k_M
a_M\bigr)^{-s},\,\,\Re(s)>M,\,\Re(w)>0,
\end{equation}
which is the formula given originally by
Barnes \cite{mBarnes} for the multiple zeta function.
Following \cite{Ruij}, define the Barnes multiple gamma function by
\begin{equation}\label{mgamma}
\Gamma_M(w\,|\,a) \triangleq \exp\bigl(L_M(w\,|\,a)\bigr).
\end{equation}
It follows from Eqs. \eqref{mzdef} and \eqref{mgamma} that
$\Gamma_M(w\,|\,a)$ satisfies the fundamental functional equation
in Eq. \eqref{feq}.
The multiple gamma function in the modern normalization has the following properties, cf. \cite{Me14} for a detailed review.
The first property is scaling invariance. 
Let \(\Re(w)>0,\) \(\kappa>0\) and \((\kappa\,a)_i\triangleq\kappa\,a_i,\;i=1\cdots M.\)
\begin{equation}\label{scale}
\Gamma_M(\kappa w\,|\,\kappa a) = \kappa^{-B_{M,M}(w\,|\,a)/M!}\,\Gamma_M(w\,|\,a).
\end{equation}
In the case of classical normalization, this result appears to be due to \cite{KataOhts}.
It was re-discovered in \cite{Kuz} in the special case of $M=2.$
The second property is Barnes multiplication.
Let $\Re(w)>0$ and $k=1,2,3,\cdots.$
\begin{equation}\label{multiplic}
\Gamma_M(kw\,|\,a) = k^{-B_{M,
M}(kw\,|\,a)/M!}\,\prod\limits_{p_1,\cdots,p_M=0}^{k-1}\Gamma_M\Bigl(w+\frac{\sum_{j=1}^M
p_j a_j}{k}\,\Big|\,a\Bigr).
\end{equation}
In the classical case, this result is due to \cite{mBarnes}. We will be particularly interested in the special case of $M=2$
and record the formula for $B_{2,2}(w\,|\,a)$ for future convenience.
\begin{equation}\label{B22}
B_{2,2}(w\,|\,a) =
\frac{w^2}{a_1a_2}-\frac{w(a_1+a_2)}{a_1a_2}+\frac{a_1^2+3a_1a_2+a_2^2}{6a_1a_2}.
\end{equation}
Finally, we state the general Shintani factorization in a slightly simplified way that is sufficient for our needs.
Given arbitrary $x>0,$ there exist functions \(\Psi_{M+1}(w,y\,|\,a)\) and 
   $\phi_{M+1}\bigl(w,y\,|\,a, a_{M+1}\bigr)$ such that
\begin{align}
\Gamma_{M+1}\bigl(w\,|\,a,a_{M+1}\bigr) = & \prod\limits_{k=1}^\infty \frac{\Gamma_M(w+ka_{M+1}\,|\,a)}{\Gamma_M(x+ka_{M+1}\,|\,a)}e^{\Psi_{M+1}(x,\,ka_{M+1}\,|\,a)-\Psi_{M+1}(w,\,ka_{M+1}\,|\,a)} \times \nonumber \\
&\times \exp{\bigl(\phi_{M+1}(w,x\,|\,a, a_{M+1})\bigr)}\, \Gamma_{M}(w\,|\,a). \label{generalfactorization}
\end{align}
\(\Psi_{M+1}(w,y\,|\,a)\) and \(\phi_{M+1}(w,y\,|\,a, a_{M+1})\) are \emph{polynomials} in $w$ of degree $M+1.$
Explicit formulas for these functions are given in \cite{Me14}. In the classical normalization this type of factorization, cf. Eq. \eqref{ShintaniGamma} above, was discovered in \cite{Shi} for $M=1$ and 
extended to general $M$ in \cite{KataOhts}.  We gave new proofs for $M=1$  using the classical normalization in \cite{MeIMRN}  and of the general case in the modern normalization in \cite{Me14}.

For concreteness, in what follows we will write $\Gamma_M(z\,|\,a)$ to mean the multiple gamma function in the modern as opposed to classical normalization, unless specifically stated to the contrary. As we will see, the key formulas are invariant of the choice of normalization as the
fundamental functional equation Eq. \eqref{feq} is the same in both normalizations and our formulas involve \emph{ratios}. This is in particular true of  the identity in Eq. \eqref{repeated}.

We conclude this section with a brief mention of the multiple sine function, which occurs in the context of ratios of Barnes beta distributions
as well as in the analytic continuation of the complex Selberg integral.
The multiple sine function \cite{KurKoya} is defined by
\begin{equation}\label{msinedef}
S_M(w|a) \triangleq \frac{\Gamma_M(|a|-w|a)^{(-1)^M}}{\Gamma_M(w|a)},
\end{equation}
where $M=0,1,2\cdots,$ $|a|=\sum_{i=1}^M a_i,$ and $a=(a_1,\cdots, a_M)$ are fixed positive
constants. 
It satisfies the same functional equation as the multiple gamma function,
\begin{equation}\label{feqsine}
S_{M}(w\,|\,a) = S_{M-1}(w\,|\,\hat{a}_i)\,S_M\bigl(w+a_i\,|\,a\bigr),\,i=1\cdots
M, 
\end{equation}
$\hat{a}_i = (a_1,\cdots, a_{i-1},\,a_{i+1},\cdots, a_{M}).$ In particular, when $M=1,$ we recover the classical sine function,
\begin{equation}
S_1(w\,|\,a) = 2\sin(\pi w/a).
\end{equation}
Let $a=(1, \, \tau).$ Given the functional equation, instead of Eq. \eqref{repeated}, we have the identity,
\begin{align}
\frac{S_2(w+1-k\,|\,1,\tau)}{S_2(w+1\,|\,1,\tau)} = & \prod\limits_{j=0}^{k-1}
S_1\bigl(w-j\,|\,1,\tau\bigr), \nonumber \\
= & 
\prod\limits_{j=0}^{k-1} 2 \sin\pi\bigl(\frac{w}{\tau}-\frac{j}{\tau}\bigr). \label{Srepeated}
\end{align}

In most of the applications of the double gamma function below its second argument is $a=(1,\,\tau).$ 
In all such cases we will write $\Gamma_2(\cdot\,|\,\tau)$ as an abbreviation for $\Gamma_2(\cdot\,|\,1, \tau).$


\section{Morris Integral Distribution: Analytical Approach}\label{CirAnalytical}
\noindent
In this section we will construct a positive probability distribution having the properties that it matches the moments and intermittency
expansion of the total mass of the Bacry-Muzy GMC on the circle as specified in Eqs. \eqref{momlambda} and \eqref{cpncircle} in Section \ref{Problem}.
Our approach in this section is analytical and focused on the Mellin transform of what we call the Morris integral probability distribution.
Its moments match the full Morris integral in Eq. \eqref{morris2}. 
The fine probabilistic structure of this distribution is described in Section \ref{Probabilistic} below. The proofs of all results
in this section are given in Section \ref{proofsanalytical}.

Define the function 
\begin{align}
\mathfrak{M}(q\,|\tau,\,\lambda_1,\,\lambda_2)=&\frac{\tau^{\frac{q}{\tau}}}{\Gamma^q\bigl(1-\frac{1}{\tau}\bigr)}
\frac{\Gamma_2(\tau(\lambda_1+\lambda_2+1)+1-q\,|\,\tau)}{\Gamma_2(\tau(\lambda_1+\lambda_2+1)+1\,|\,\tau)}
\frac{\Gamma_2(-q+\tau\,|\,\tau)}{\Gamma_2(\tau\,|\,\tau)}\times \nonumber \\ & \times
\frac{\Gamma_2(\tau(1+\lambda_1)+1\,|\,\tau)}{\Gamma_2(\tau(1+\lambda_1)+1-q\,|\,\tau)}
\frac{\Gamma_2(\tau(1+\lambda_2)+1\,|\,\tau)}{\Gamma_2(\tau(1+\lambda_2)+1-q\,|\,\tau)}. \label{thefunctioncircle}
\end{align}
Equivalently, we have the identity in terms of the Alexeiewsky-Barnes $G-$function.
\begin{align}
\mathfrak{M}(q\,|\tau,\,\lambda_1,\,\lambda_2)=&\frac{1}{\Gamma^q\bigl(1-\frac{1}{\tau}\bigr)}
\frac{G(\tau(\lambda_1+\lambda_2+1)+1\,|\,\tau)}{G(\tau(\lambda_1+\lambda_2+1)+1-q\,|\,\tau)}
\frac{G(\tau\,|\,\tau)}{G(-q+\tau\,|\,\tau)}\times \nonumber \\ & \times
\frac{G(\tau(1+\lambda_1)+1-q\,|\,\tau)}{G(\tau(1+\lambda_1)+1\,|\,\tau)}
\frac{G(\tau(1+\lambda_2)+1-q\,|\,\tau)}{G(\tau(1+\lambda_2)+1\,|\,\tau)}. \label{MG}
\end{align}
Then, we have the following results. 
\begin{theorem}\label{InfinFacCir}
The function $\mathfrak{M}(q\,|\tau,\,\lambda_1,\,\lambda_2)$ has the infinite product factorization, 
\begin{align}
\mathfrak{M}(q\,|\tau,\,\lambda_1,\,\lambda_2) = & \frac{\Gamma(1-q/\tau)}{\Gamma^q(1-1/\tau)} \prod\limits_{m=1}^\infty \Bigl[
\frac{\Gamma(1-q+m\tau)}{\Gamma(1+m\tau)}  \frac{\Gamma(1+\tau\lambda_1+m\tau)}{\Gamma(1-q+\tau\lambda_1+m\tau)} 
\frac{\Gamma(1+\tau\lambda_2+m\tau)}{\Gamma(1-q+\tau\lambda_2+m\tau)} \nonumber \times \\  & \times \frac{\Gamma(1-q+\tau(\lambda_1+\lambda_2)+m\tau)}{\Gamma(1+\tau(\lambda_1+\lambda_2)+m\tau)}\Bigr].
\end{align}
\end{theorem}
\begin{theorem}\label{CirMoments}
The function $\mathfrak{M}(q\,|\tau,\,\lambda_1,\,\lambda_2)$ reproduces the product in Eq. \eqref{morris2} when $q=n<\tau.$
The values of $\mathfrak{M}(q\,|\tau,\,\lambda_1,\,\lambda_2)$ at the negative integers, $q=-n,$ $n\in\mathbb{N},$ are
\begin{equation}\label{CirNegMoments}
\mathfrak{M}(-n\,|\tau,\,\lambda_1,\,\lambda_2)  =  \prod\limits_{j=0}^{n-1} \frac{\Gamma(1+\lambda_1+\frac{(j+1)}{\tau}) \,\Gamma(1+\lambda_2+\frac{(j+1)}{\tau})\Gamma(1-\frac{1}{\tau})}{\Gamma(1+\lambda_1+\lambda_2+\frac{(j+1)}{\tau})\,\Gamma(1+\frac{j}{\tau})}. 
\end{equation}
\end{theorem}
\begin{theorem}[Morris Integral Probability Distribution]\label{CircleExist}
The function $\mathfrak{M}(iq\,|\tau,\,\lambda_1,\,\lambda_2),$ $q\in\mathbb{R},$ $\tau>1,$
is the Fourier transform of an
infinitely divisible, absolutely continuous probability distribution on $\mathbb{R}$ 
with the L\'evy-Khinchine decomposition
\begin{align}
\log \mathfrak{M}(iq\,|\tau,\,\lambda_1,\,\lambda_2) = & \int\limits_0^\infty
\frac{dx}{x} (e^{ixq}-1) \frac{e^{-\tau x}\bigl(1-e^{-(1+\tau\lambda_1)x}\bigr)\bigl(1-e^{-(1+\tau\lambda_2)x}\bigr)}{(1-e^{-x})(1-e^{-x\tau})}
+ \nonumber \\ 
& + \int\limits_0^\infty
\frac{dx}{x} (e^{ixq}-1-ixq) \frac{e^{-x\bigl(1+\tau(\lambda_1+\lambda_2)\bigr)}}{e^{x\tau}-1} + iq \,const . \label{CirLKh}
\end{align}
\end{theorem}

\begin{corollary}\label{CirExistRV}
Denote the density of the probability distribution in Theorem \ref{CircleExist} by $f(x).$ The function
$q\rightarrow\mathfrak{M}(q\,|\,\tau,\lambda_1,\lambda_2)$ for
$\Re(q)<\tau$ and $\tau>1$ is the Mellin transform of the
probability density function $f(\log
y)/y,$ $y\in (0,\,\infty).$
\end{corollary}
Denote the probability distribution corresponding to $f(\log y)/y$ by $M_{(\tau, \lambda_1,
\lambda_2)}.$ We call it the Morris integral probability distribution. Thus, we have established the identity
\begin{equation}
{\bf E} [M_{(\tau, \lambda_1,
\lambda_2)}^q] = \mathfrak{M}(q\,|\,\tau,\lambda_1,\lambda_2), \; \Re(q)<\tau.
\end{equation}
\begin{theorem}\label{CirAsymp}
The function $\log\mathfrak{M}(q\,|\,\tau,\lambda_1,\lambda_2)$ has
the asymptotic expansion as $\tau\rightarrow +\infty,$
\begin{gather}
\log\mathfrak{M}(q\,|\tau,\,\lambda_1,\,\lambda_2)\thicksim q\Bigl(\log\Gamma(1+\lambda_1+\lambda_2)-\log\Gamma(1+\lambda_1)-\log\Gamma(1+\lambda_2)\Bigr) + 
\nonumber \\ + \sum\limits_{p=1}^\infty 
\frac{1}{p\tau^p} \Bigl[\bigl(\zeta(p,\,1+\lambda_1+\lambda_2)-\zeta(p, 1+\lambda_1)-\zeta(p, 1+\lambda_2)\bigr)\frac{B_{p+1}(q)-B_{p+1}}{p+1}+\nonumber \\ +
\zeta(p)\frac{B_{p+1}(q+1)-B_{p+1}}{p+1}-q\zeta(p)\Bigr].
\end{gather}
\end{theorem}
\begin{theorem}\label{CicInvolution}
The Mellin transform is involution invariant under
\begin{equation}\label{invtranscircle}
\tau\rightarrow \frac{1}{\tau},\; q\rightarrow \frac{q}{\tau}, \; \lambda_i\rightarrow \tau\lambda_i.
\end{equation}
\begin{equation}
\mathfrak{M}\bigl(\frac{q}{\tau}\,\Big|\,\frac{1}{\tau},\tau\lambda_1,\tau\lambda_2\bigr) \,\Gamma^{\frac{q}{\tau}}(1-\tau) \Gamma(1-\frac{q}{\tau}) = 
\mathfrak{M}(q\,|\,\tau,\lambda_1,\lambda_2) \Gamma^{q}(1-\frac{1}{\tau}) \Gamma(1-q). \label{invcircle}
\end{equation}
\end{theorem}

It is worth pointing out that in the special case of $\lambda_1=\lambda_2=0,$ we have
\begin{equation}
\mathfrak{M}(q\,|\tau, 0, 0)=\frac{\Gamma\bigl(1-\frac{q}{\tau}\bigr)}{\Gamma^q\bigl(1-\frac{1}{\tau}\bigr)}, 
\end{equation}

\begin{conjecture}[Law of Total Mass]\label{ourmainconjcircle}
Let $M_{(\tau,\lambda_1,\lambda_2)}$ be as constructed in Theorem \ref{CircleExist}. Let 
\begin{equation}
\lambda_1=\lambda_2=\lambda,\footnote{This restriction is necessary as $M_{(\tau,\lambda_1,\lambda_2)}$ is real-valued whereas  
$\int_{-1/2}^{1/2} e^{2\pi i\psi \frac{\lambda_1-\lambda_2}{2}} \, |1+e^{2\pi i\psi}|^{\lambda_1+\lambda_2}\, M_{\beta}(d\psi)$ is not in general, unless $\lambda_1=\lambda_2.$  The problem of determining the law of $\int_{-1/2}^{1/2} e^{2\pi i\psi \frac{\lambda_1-\lambda_2}{2}} \, |1+e^{2\pi i\psi}|^{\lambda_1+\lambda_2}\, M_{\beta}(d\psi)$ for $\lambda_1\neq\lambda_2$ is left to future research.}
\end{equation}
\begin{equation}
M_{(\tau,\lambda,\lambda)} \overset{{\rm in \,law}}{=}   \int_{-1/2}^{1/2} |1+e^{2\pi i\psi}|^{2\lambda}\, M_{\beta}(d\psi),\; \tau=1/\beta^2>1.
\end{equation}
\end{conjecture}
This conjecture for $\lambda=0$ is due to \cite{FyoBou}. It was recently verified in \cite{Remy}. The general case is due to
\cite{Me16}.

\section{Selberg Integral Distribution: Analytical Approach}\label{IntAnalytical}
\noindent
In this section we will construct a positive probability distribution having the properties that it matches the moments and intermittency
expansion of the total mass of the Bacry-Muzy GMC on the interval as specified in Eqs. \eqref{Selberg} and \eqref{cpninterval} in Section \ref{Problem}.
Our approach in this section is analytical and focused on the Mellin transform of what we call the Selberg integral probability distribution.
The fine probabilistic structure of this distribution is described in Section \ref{Probabilistic} below. The proofs of all results
in this section are given in Section \ref{proofsanalytical}.

Define the function
\begin{align}
\mathfrak{M}(q\,|\,\tau,\lambda_1,\lambda_2)  \triangleq &
\Bigl(\frac{2\pi\,\tau^{\frac{1}{\tau}}}{\Gamma\bigl(1-1/\tau\bigr)}\Bigr)^q\;
\frac{\Gamma_2(1-q+\tau(1+\lambda_1)\,|\,\tau)}{\Gamma_2(1+\tau(1+\lambda_1)\,|\,\tau)}
\frac{\Gamma_2(1-q+\tau(1+\lambda_2)\,|\,\tau)}{\Gamma_2(1+\tau(1+\lambda_2)\,|\,\tau)}\times
\nonumber \\ & \times
\frac{\Gamma_2(-q+\tau\,|\,\tau)}{\Gamma_2(\tau\,|\,\tau)}
\frac{\Gamma_2(2-q+\tau(2+\lambda_1+\lambda_2)\,|\,\tau)}{\Gamma_2(2-2q+\tau(2+\lambda_1+\lambda_2)\,|\,\tau)}
\label{thefunctioninterval}
\end{align}
for $\Re(q)<\tau.$
Equivalently, we have the identity in terms of the Alexeiewsky-Barnes $G-$function.
\begin{gather}
\mathfrak{M}(q\,|\,\tau,\lambda_1,\lambda_2) =
\Gamma^{-q}\bigl(1-1/\tau\bigr)
\frac{G(1+\tau(1+\lambda_1)\,|\,\tau)}{G(1-q+\tau(1+\lambda_1)\,|\,\tau)}
\frac{G(1+\tau(1+\lambda_2)\,|\,\tau)}{G(1-q+\tau(1+\lambda_2)\,|\,\tau)}\times
\nonumber \\ \times \frac{G(1+\tau\,|\,\tau)}{G(-q+\tau\,|\,\tau)}
\frac{G(2-2q+\tau(2+\lambda_1+\lambda_2)\,|\,\tau)}{G(2-q+\tau(2+\lambda_1+\lambda_2)\,|\,\tau)}. \label{MGI}
\end{gather}
\begin{theorem}\label{InfinFacInt}
Given $\Re(q)<\tau,$ 
\begin{gather}
\mathfrak{M}(q\,|\,\tau,\lambda_1,\lambda_2) =
\frac{\tau^{q}\Gamma\bigl(1-q/\tau\bigr)\Gamma\bigl(2-2q+\tau(1+\lambda_1+\lambda_2)\bigr)}{
\Gamma^{q}\bigl(1-1/\tau\bigr)\Gamma\bigl(2-q+\tau(1+\lambda_1+\lambda_2)\bigr)}
\prod\limits_{m=1}^\infty \left(m\tau\right)^{2q}
\frac{\Gamma\bigl(1-q+m\tau\bigr)}{\Gamma\bigl(1+m\tau\bigr)}
 \times \nonumber \\
\times
\frac{\Gamma\bigl(1-q+\tau\lambda_1+m\tau\bigr)}{\Gamma\bigl(1+\tau\lambda_1+m\tau\bigr)}
\frac{\Gamma\bigl(1-q+\tau\lambda_2+m\tau\bigr)}{\Gamma\bigl(1+\tau\lambda_2+m\tau\bigr)}
\frac{\Gamma\bigl(2-q+\tau(\lambda_1+\lambda_2)+m\tau\bigr)}{\Gamma\bigl(2-2q+\tau(\lambda_1+\lambda_2)+m\tau\bigr)}.\label{InfiniteSelberg}
\end{gather}
\end{theorem}
\begin{theorem}\label{IntMoments}
The function $\mathfrak{M}(q\,|\,\tau,\lambda_1,\lambda_2)$ reproduces the product in Eq. \eqref{Selberg} for $q=n<\tau.$ 
The values of $\mathfrak{M}(q\,|\,\tau,\lambda_1,\lambda_2)$ at the negative integers, $q=-n,$ $n\in\mathbb{N},$ are
\begin{equation}
\mathfrak{M}(-n\,|\,\tau,\lambda_1,\lambda_2) = \prod_{k=0}^{n-1}
\frac{\Gamma\bigl(2+\lambda_1+\lambda_2+(n+2+k)/\tau\bigr)
\Gamma\bigl(1-1/\tau\bigr) }{
\Gamma\bigl(1+\lambda_1+(k+1)/\tau\bigr)\Gamma\bigl(1+\lambda_2+(k+1)/\tau\bigr)
\Gamma\bigl(1+k/\tau\bigr) }. \label{IntNegMoments}
\end{equation}
\end{theorem}
\begin{theorem}[Selberg integral probability distribution]\label{IntExist}
The function
$q\rightarrow\mathfrak{M}(iq\,|\,\tau,\lambda_1,\lambda_2),$
$q\in\mathbb{R},$ $\tau>1,$ is the Fourier transform of an
infinitely divisible probability distribution on $\mathbb{R}$ 
with the L$\acute{\text{e}}$vy-Khinchine decomposition
\begin{equation}
\log\mathfrak{M}(iq\,|\,\tau,\lambda_1,\lambda_2) =
iq\,\mathfrak{m}(\tau) - \frac{1}{2} q^2
\sigma^2(\tau)+\int_{\mathbb{R}\setminus \{0\}} \bigl(e^{iq u}-1-iq
u/(1+u^2)\bigr) d\mathcal{M}_{(\tau, \lambda_1, \lambda_2)}(u)
\end{equation} 
for
some $\mathfrak{m}(\tau)\in\mathbb{R}$ and the following gaussian
component and spectral function
\begin{gather}
\sigma^2(\tau) = \frac{4\,\log 2}{\tau}, \\
\mathcal{M}_{(\tau, \lambda_1,
\lambda_2)}(u)\!\!=\!\!-\!\!\int\limits_u^\infty \!\!
\Bigl[\frac{\bigl(e^x+e^{-x\tau\lambda_1}+e^{-x\tau\lambda_2}+e^{-x(1+\tau(1+\lambda_1+\lambda_2))}\bigr)}{\bigl(e^x-1\bigr)(e^{x\tau}-1)}
\!-\!\frac{e^{-x(1+\tau(1+\lambda_1+\lambda_2))/2}}{\bigl(e^{x/2}-1\bigr)(e^{x\tau/2}-1)}\Bigr]\!\!
\frac{dx}{x} \label{IntLKh}
\end{gather}
for $u>0,$ and $\mathcal{M}_{(\tau, \lambda_1, \lambda_2)}(u)=0$ for
$u<0.$ $d\mathcal{M}_{(\tau, \lambda_1, \lambda_2)}(u)/du > 0$ for
$u>0.$ 
\end{theorem}
\begin{corollary}\label{IntExistRV}
The probability distribution in Theorem \ref{IntExist} has a bounded,
continuous, zero-free density function $f(x),$ $x\in\mathbb{R}.$ The function
$q\rightarrow\mathfrak{M}(q\,|\,\tau,\lambda_1,\lambda_2)$ for
$\Re(q)<\tau$ and $\tau>1$ is the Mellin transform of the
probability density function $f(\log
y)/y,$ $y\in\mathbb{R}_+.$
\end{corollary}
Denote the probability distribution corresponding to $f(\log y)/y$ by $M_{(\tau, \lambda_1,
\lambda_2)}.$ We call it the Selberg integral probability distribution.
Thus, we have established the identity
\begin{equation}
{\bf E} [M_{(\tau, \lambda_1,
\lambda_2)}^q] = \mathfrak{M}(q\,|\,\tau,\lambda_1,\lambda_2), \; \Re(q)<\tau.
\end{equation}
\begin{theorem}\label{IntAsymp}
The function $\log\mathfrak{M}(q\,|\,\tau,\lambda_1,\lambda_2)$ has
the asymptotic expansion as $\tau\rightarrow +\infty,$
\begin{gather}
\log\mathfrak{M}(q\,|\,\tau,\lambda_1,\lambda_2) \thicksim
q\log\Bigl(\frac{\Gamma(1+\lambda_1)\Gamma(1+\lambda_2)}{\Gamma(2+\lambda_1+\lambda_2)}\Bigr)+
\sum\limits_{p=1}^\infty \Bigl(\frac{1}{\tau}\Bigr)^{p}
\frac{1}{p}\Bigl[-\zeta(p)q+\nonumber
\\+\bigl(\zeta(p, 1+\lambda_1)+\zeta(p,
1+\lambda_2)\bigr)\Bigl(\frac{B_{p+1}(q)-B_{p+1}}{p+1}\Bigr)
 + \zeta(p) \times \nonumber \\
\times\Bigl(\frac{B_{p+1}(q+1)-B_{p+1}}{p+1}\Bigr) - \zeta(p,
2+\lambda_1+\lambda_2)
\Bigl(\frac{B_{p+1}(2q-1)-B_{p+1}(q-1)}{p+1}\Bigr)\Bigr].\label{logMellinAsympI}
\end{gather}
\end{theorem}

\begin{theorem}\label{Mtransforminvol}
The Mellin transform is involution invariant under
\begin{equation}
\tau\rightarrow \frac{1}{\tau},\; q\rightarrow \frac{q}{\tau}, \; \lambda_i\rightarrow \tau\lambda_i.
\end{equation}
\begin{align}
\mathfrak{M}\bigl(\frac{q}{\tau}\,\Big|\,\frac{1}{\tau},\tau\lambda_1,\tau\lambda_2\bigr) (2\pi)^{-\frac{q}{\tau}}\,\Gamma^{\frac{q}{\tau}}(1-\tau) \Gamma(1-\frac{q}{\tau}) = &
\mathfrak{M}(q\,|\,\tau,\lambda_1,\lambda_2) (2\pi)^{-q} \times \nonumber \\ & \times \Gamma^{q}(1-\frac{1}{\tau}) \Gamma(1-q).\label{involutionint}
\end{align}
\end{theorem} 
This result can be proved by both analytic and probabilistic means. We prefer the latter, cf. Corollary \ref{mycorollary}, so that the proof of Theorem \ref{Mtransforminvol} is given in Section \ref{ProbProofs} below. The analytic proof is similar to the proofs of Theorem \ref{CicInvolution} in Section \ref{proofsanalytical} and Corollary \ref{MComplextransforminvol} in Section \ref{AnalyticalComplexSelberg}.


\begin{conjecture}[Law of Total Mass]\label{ourmainconjinterval}
Let $M_{(\tau,\lambda_1,\lambda_2)}$ be as constructed in Theorem \ref{IntExist}. Then, 
\begin{equation}
M_{(\tau,\lambda_1,\lambda_2)} \overset{{\rm in \,law}}{=}   \int_0^1 \,s^{\lambda_1}(1-s)^{\lambda_2} \,
M_\beta(ds),\; \tau=1/\beta^2>1.
\end{equation}
\end{conjecture}

The expression for the negative moments in Eq. \eqref{IntNegMoments} with $\lambda_1=\lambda_2=0$  first appeared in \cite{Me3}.
Theorem \ref{IntMoments} first appeared in \cite{FLDR}, who gave an equivalent expression for the right-hand
side of Eq. \eqref{MGI} and verified that it matches Eq. \eqref{Selberg} without proving analytically that their formula corresponds to the Mellin transform of a probability distribution. The special case of $\lambda_1=\lambda_2=0$ of Theorems  \ref{InfinFacInt}, \ref{IntExist}, and
\ref{IntAsymp} first appeared in \cite{Me4} and the general case in \cite{MeIMRN}. 
The involution invariance of the Mellin transform in the equivalent form of self-duality, see Eq. \eqref{selfdual} below,  was first discovered in the special
case of $\lambda_1=\lambda_2=0$ in \cite{FLDR}. We extended it to the general case in the form of Eq. \eqref{involutionint} in \cite{Me14},
followed by the general form of self-duality in \cite{FLD}. 
The interested reader can find additional information about the Selberg integral
distribution such as  functional equations, an analytical approach to the moment problem based on Eq. \eqref{BillKingAsymp}, and the tails  
in \cite{Me4} and \cite{MeIMRN}.

It is worth pointing out that in our original derivation of the Selberg integral probability distribution in \cite{Me4} in the
special case of $\lambda_1=\lambda_2=0$ and in \cite{MeIMRN} in general 
we first established Eq.  \eqref{MGI} by
summing the asymptotic series in Theorem \ref{IntAsymp} using Hardy's moment
constant method, cf. Sections 4.12 and 4.13 in \cite{Hardy}, 
 and Ramanujan's generalization of Watson's lemma, cf. Lemma 10.2 in Chap. 38 of \cite{Berndt}, 
\emph{i.e.} we summed
the intermittency expansions in closed form as explained in Section \ref{Problem}. This asymptotic
series is divergent in general, however it is convergent for finite ranges of positive and negative 
integer $q,$ see \cite{Me4}.


\section{Proofs of Analytical Results}\label{proofsanalytical}
\noindent In this section we will give proofs of all
results that were stated in Sections \ref{CirAnalytical} and \ref{IntAnalytical}. All the proofs rely on various properties
of the Alexeiewsky-Barnes $G$-function and the double gamma function, which were reviewed in
Section \ref{BarnesReview}. Recall that $\tau=2/\mu$ as defined in Section \ref{Problem}.

\subsection{The Circle} 
\noindent 
We start by noting that the equivalence between the formulas for $\mathfrak{M}(q\,|\,\tau,\lambda_1,\lambda_2)$
in Eqs. \eqref{thefunctioncircle} and \eqref{MG} follows from Eq. \eqref{GfromG2}.
\begin{proof}[Proof of Theorem \ref{InfinFacCir}.]
Using the functional equation of the $G(z\,|\,\tau)$ function in Eq. \eqref{Gfunct1}, we can re-write the product in Eq. \eqref{MG} in the
form
\begin{align}
\mathfrak{M}(q\,|\tau,\,\lambda_1,\,\lambda_2)=&\frac{\Gamma(1-q/\tau)}{\Gamma^q\bigl(1-\frac{1}{\tau}\bigr)}
\frac{G(\tau(\lambda_1+\lambda_2+1)+1\,|\,\tau)}{G(\tau(\lambda_1+\lambda_2+1)+1-q\,|\,\tau)}
\frac{G(1+\tau\,|\,\tau)}{G(1-q+\tau\,|\,\tau)}\times \nonumber \\ & \times
\frac{G(\tau(1+\lambda_1)+1-q\,|\,\tau)}{G(\tau(1+\lambda_1)+1\,|\,\tau)}
\frac{G(\tau(1+\lambda_2)+1-q\,|\,\tau)}{G(\tau(1+\lambda_2)+1\,|\,\tau)}. 
\end{align}
It now remains to apply the Shintani identity in Eq. \eqref{ShintaniG} to each of the $G-$factors and notice the cancellations of all
terms except for the $\Gamma-$factors. \qed
\end{proof}

\begin{proof}[Proof of Theorem \ref{CirMoments}.]
These formulas are immediate from Eq.  \eqref{repeated} or equivalently Eq. \eqref{Grepeated}. \qed
\end{proof}
\begin{proof}[Proof of Theorem \ref{CircleExist}.] 
The proof is based on the identities in Eqs. \eqref{Malmsten}  and \eqref{Gratioidentity}. We start by re-writing Eq. \eqref{MG} by means of
Eq. \eqref{Gfunct1} in the form
\begin{align}
\mathfrak{M}(q\,|\tau,\,\lambda_1,\,\lambda_2)=&\frac{1}{\Gamma^q\bigl(1-\frac{1}{\tau}\bigr)}
\frac{\Gamma(\frac{ \tau(\lambda_1+\lambda_2+1)+1-q}{\tau})}{\Gamma(\frac{ \tau(\lambda_1+\lambda_2+1)+1}{\tau})}
\frac{G(\tau(\lambda_1+\lambda_2+1)+2\,|\,\tau)}{G(\tau(\lambda_1+\lambda_2+1)+2-q\,|\,\tau)}
\frac{G(\tau\,|\,\tau)}{G(-q+\tau\,|\,\tau)}\times \nonumber \\ & \times
\frac{G(\tau(1+\lambda_1)+1-q\,|\,\tau)}{G(\tau(1+\lambda_1)+1\,|\,\tau)}
\frac{G(\tau(1+\lambda_2)+1-q\,|\,\tau)}{G(\tau(1+\lambda_2)+1\,|\,\tau)}. 
\end{align}
We now make two observations. The four ratios of $G-$factors are in the same functional form as in Eq. \eqref{Gratioidentity} with
$b=\tau,$ $c=1+\tau\lambda_1,$ and $d=1+\tau\lambda_2.$ The ratio of the $\Gamma-$ factors has the functional form
$\Gamma(1+(z-q)/\tau)/\Gamma(1+z/\tau)$ with $z=1+\tau(\lambda_1+\lambda_2)$ and can be represented by means of Eq. \eqref{Malmsten}.
We have the identity
\begin{equation}
\log\Gamma\bigl(1+\frac{z-q}{\tau}\bigr) - \log\Gamma\bigl(1+\frac{z}{\tau}\bigr) = \int\limits_{0}^\infty \frac{dt}{t}
\Bigl[ e^{-tz}\bigl(\frac{e^{tq}-1}{e^{t\tau}-1}\bigr) - \frac{q}{\tau} e^{-t\tau}\Bigr].
\end{equation}
The formula in Eq. \eqref{CirLKh} follows. Thus, we have reduced $\log \mathfrak{M}(iq\,|\tau,\,\lambda_1,\,\lambda_2)$ to
the L\'evy-Khinchine form. It follows from the general theory of infinitely divisible probability distributions on the real line
that $\mathfrak{M}(iq\,|\tau,\,\lambda_1,\,\lambda_2)$ is the Fourier transform of an infinitely divisible distribution that is supported
on the real line, cf. Proposition 8.2 in Chapter 4 of \cite{SteVHar}, and is absolutely continuous,  cf. Theorem 4.23 in Chapter 4 of \cite{SteVHar}. 
\qed
\end{proof}
\begin{proof}[Proof of Corollary \ref{CirExistRV}.]
Denote the random variable corresponding
to the L\'evy-Khinchine decomposition 
by $X$ and its probability density by $f(x),$ $x\in\mathbb{R}.$ The Fourier transform of
$f(x)$ is
$\mathfrak{M}(iq\,|\,\mu,\lambda_1,\lambda_2)$ by construction,
which is analytic as a function of $q$ in the strip $\Im(q)>-\tau.$
Then, by the fundamental theorem of analytic characteristic
functions, confer Theorem 7.1.1 in Chap. 7 of \cite{Lukacs}, we have
for all $q$ in the strip of analyticity $\Im(q)>-\tau$
\begin{equation}
\mathfrak{M}(iq\,|\,\mu,\lambda_1,\lambda_2) =
\int\limits_\mathbb{R} e^{iqx}\,f(x)\,dx.
\end{equation}
On the other hand, the random variable
$M_{(\mu,\lambda_1,\lambda_2)}\triangleq \exp(X)$ has the density $f_{(\mu, \lambda_1,
\lambda_2)}(\log y)/y,$ $y\in (0, \infty)$ so that the right-hand
side is precisely its Mellin transform for
$\Re(q)<\tau$
\begin{equation}
\mathfrak{M}(q\,|\,\mu,\lambda_1,\lambda_2) = \int\limits_0^\infty
y^q\,f_{(\mu, \lambda_1, \lambda_2)}(\log y)\,\frac{dy}{y} = {\bf
E}\bigl[M_{(\mu,\lambda_1,\lambda_2)}^q\bigr]
\end{equation}
as seen upon relabeling $iq\rightarrow q$ and changing variables $y
= e^x.$
\qed
\end{proof}
\begin{proof}[Proof of Theorem \ref{CirAsymp}.]
This is an immediate corollary of Eqs. \eqref{IfuncG} and \eqref{IfuncGAsymptotic}.
It remains to apply these equations to each of the four ratios of the $G-$factors in Eq. \eqref{MG}
and collect the terms. \qed
\end{proof}
\begin{proof}[Proof of Theorem \ref{CicInvolution}.]
To prove the involution invariance in Eq. \eqref{invcircle}, we need to recall the scaling property of the multiple gamma function, see Eq. \eqref{scale}.
The transformation in Eq. \eqref{invtranscircle} corresponds to $\kappa=1/\tau.$ We note first that
\begin{equation}
\frac{1}{\tau}(1,\,\tau)=(1, \,\frac{1}{\tau})
\end{equation}
in the sense of the parameters of the double gamma function. We apply Eq. \eqref{scale} to each of the double gamma factors in Eq. \eqref{thefunctioncircle}
under the transformation in Eq. \eqref{invtranscircle}. For example, 
\begin{align}
\Gamma_2\bigl(\frac{1}{\tau}(\tau\lambda_1+\tau\lambda_2+1)+1-\frac{q}{\tau}\,\Big|\,\frac{1}{\tau}\bigr) = & \Gamma_2\Bigl(\frac{1}{\tau}\bigl(\tau(\lambda_1+\lambda_2+1)+1-q\bigr)\,\Big|\,\frac{1}{\tau}\Bigr), \nonumber \\
= & \bigl(\frac{1}{\tau}\bigr)^{-B_{2,2}(\tau(\lambda_1+\lambda_2+1)+1-q)/2} \times\nonumber \\ & \times \Gamma_2\Bigl(\tau(\lambda_1+\lambda_2+1)+1-q\,\Big|\,\tau\Bigr).
\end{align}
Using the formula for $B_{2,2}(x\,|\,a)$ in Eq. \eqref{B22}
with $(a_1=1,\,a_2=\tau),$ we collect all terms and simplify to obtain
\begin{equation}
\mathfrak{M}\bigl(\frac{q}{\tau}\,|\,\frac{1}{\tau},\tau\lambda_1,\tau\lambda_2\bigr) 
\,\Gamma^{\frac{q}{\tau}}(1-\tau) = 
\mathfrak{M}(q\,|\,\tau,\lambda_1,\lambda_2)\Bigl(\frac{
\tau^{\frac{1}{\tau}}}{\Gamma(1-\frac{1}{\tau})}\Bigr)^{-q} 
\frac{\Gamma_2(1-q\,|\tau)\Gamma_2(\tau\,|\tau)}{\Gamma_2(1\,|\tau)\Gamma_2(\tau-q\,|\tau)}.
\end{equation}
It remains to observe that the functional equation of the double gamma function implies the identity
\begin{equation}\label{mydoublegammaidentity}
\tau^{-\frac{q}{\tau}} \frac{\Gamma_2(1-q\,|\tau)\Gamma_2(\tau\,|\tau)}{\Gamma_2(1\,|\tau)\Gamma_2(\tau-q\,|\tau)} = 
\frac{\Gamma(1-q)}{\Gamma(1-\frac{q}{\tau})},
\end{equation}
which gives the result. \qed
\end{proof}

\subsection{The Interval}
\noindent 
We start by noting that the equivalence between the formulas for $\mathfrak{M}(q\,|\,\tau,\lambda_1,\lambda_2)$
in Eqs. \eqref{thefunctioninterval} and \eqref{MGI} follows from Eq. \eqref{GfromG2} and the fact
\begin{equation}
G(\tau\,|\,\tau) = G(1+\tau\,|\,\tau).
\end{equation}

\begin{proof}[Proof of Theorem \ref{InfinFacInt}.]
The starting point is the Shintani identity, cf. Eq. \eqref{ShintaniG}.
Using the functional equations of the Alexeiewsky-Barnes
$G-$function, we first reduce some of the $G-$factors in Eq. \eqref{MGI} as
follows.
\begin{align}
\frac{1}{G(-q+\tau\,|\,\tau)}
\frac{G(2-2q+\tau(2+\lambda_1+\lambda_2)\,|\,\tau)}{G(2-q+\tau(2+\lambda_1+\lambda_2)\,|\,\tau)} 
= & \tau^q 
\frac{\Gamma\bigl(2-2q+\tau(1+\lambda_1+\lambda_2)\bigr)}{\Gamma\bigl(2-q+\tau(1+\lambda_1+\lambda_2)\bigr)}
\times \nonumber \\ & \times
\frac{\Gamma\bigl(1-q/\tau\bigr) }{G(1-q+\tau\,|\,\tau)}
\frac{G(2-2q+\tau(1+\lambda_1+\lambda_2)\,|\,\tau)}{G(2-q+\tau(1+\lambda_1+\lambda_2)\,|\,\tau)}.
\end{align}
We now apply the Shintani identity to each of the resulting
$G-$factors. It is easy to see that the terms that are
quadratic in $q$ all cancel out resulting in Eq. \eqref{InfiniteSelberg}. \qed
\end{proof}
\begin{proof}[Proof of Theorem \ref{IntMoments}.]
These formulas are immediate from Eq.  \eqref{repeated} or equivalently Eq. \eqref{Grepeated}. 
We note that these formulas also follow directly from Eq. \eqref{InfiniteSelberg}, see \cite{MeIMRN}  for details. \qed
\end{proof}

\begin{proof}[Proof of Theorem \ref{IntExist}.]
The proof is based on the
Malmst\'en-type formula for $\log G(z\,|\,\tau)$ in Eq. \eqref{LK}.
Let $\Re(q)<\tau,$ $\tau>1,$ and $\lambda_1,
\lambda_2>-1/\tau.$ Applying Eq. \eqref{LK} to each $G-$factor in Eq. \eqref{MGI}, we can write
\begin{gather}
\log\mathfrak{M}(q\,|\,\mu,\lambda_1,\lambda_2) =
\int\limits_0^\infty \frac{dt}{t}
\Bigl[\frac{1}{(e^t-1)(e^{t\tau}-1)}\Bigl(
e^{t(q-\lambda_1\tau)}+e^{t(q-\lambda_2\tau)}+e^{t(q+1)}+ \nonumber
\\ +
e^{t(q-1-\tau(1+\lambda_1+\lambda_2))}-e^{t(2q-1-\tau(1+\lambda_1+\lambda_2))}\Bigr)
 + A(t)+B(t)\,q+C(t)\,q^2\Bigr]
\end{gather}
for some functions $A(t),$ $B(t),$ and $C(t)$ depending also on $\tau,$ $\lambda_1,$ $\lambda_2.$ 
Note that $C(t)=0$ as it is proportional to $2^2-2-1-1.$ 
Then, we can rewrite this expression in the form
\begin{gather}
\log\mathfrak{M}(q\,|\,\mu,\lambda_1,\lambda_2) =
\int\limits_0^\infty \frac{dt}{t} \Bigl[
\frac{\bigl(e^{tq}-1-tq\bigr)}{(e^t-1)(e^{t\tau}-1)}\bigl(e^{-t\lambda_1\tau}+e^{-t\lambda_2\tau}+e^{t}+e^{-t(1+\tau(1+\lambda_1+\lambda_2))}\bigr)
- \nonumber \\ -\frac{1}{2}(2tq)^2
\frac{e^{-t(1+\tau(1+\lambda_1+\lambda_2))}}{{(e^t-1)(e^{t\tau}-1)}}+A(t)+B(t)\,q\Bigr]
- \int\limits_0^\infty
\frac{dt}{t}\Bigl[\frac{\bigl(e^{2tq}-1-2tq-(2tq)^2/2\bigr)}{(e^t-1)(e^{t\tau}-1)}
e^{-t(1+\tau(1+\lambda_1+\lambda_2))}\Bigr]
\end{gather}
for some appropriately modified $A(t)$ and $B(t).$ It follows that
we have
\begin{equation}
\log\mathfrak{M}(q=0\,|\,\mu,\lambda_1,\lambda_2)=\int\limits_0^\infty
\frac{dt}{t} \, A(t).
\end{equation}
On the other hand, $\mathfrak{M}(q=0\,|\,\mu,\lambda_1,\lambda_2)=1$
by construction. Hence this term vanishes. We now change variables
$t'=2t$ in the second
integral and then bring the two integrals back under the same integral sign. 
\begin{gather}
\log\mathfrak{M}(q\,|\,\mu,\lambda_1,\lambda_2) =
\int\limits_0^\infty \frac{dt}{t} \Bigl[
\bigl(e^{tq}-1-tq\bigr)\Bigl(\frac{\bigl(e^{-t\lambda_1\tau}+e^{-t\lambda_2\tau}+e^{t}+e^{-t(1+\tau(1+\lambda_1+\lambda_2))}\bigr)}
{(e^t-1)(e^{t\tau}-1)} - \nonumber \\
-\frac{e^{-t/2(1+\tau(1+\lambda_1+\lambda_2))}}{(e^{t/2}-1)(e^{t\tau/2}-1)}\Bigr)
+
\frac{q^2t^2}{2}\Bigl(\frac{e^{-t/2(1+\tau(1+\lambda_1+\lambda_2))}}{(e^{t/2}-1)(e^{t\tau/2}-1)}
-
\frac{4\,e^{-t(1+\tau(1+\lambda_1+\lambda_2))}}{{(e^t-1)(e^{t\tau}-1)}}\Bigr)
+ B(t)\,q\Bigr].
\end{gather}
Denote
\begin{gather}
\mathfrak{f}(t)\triangleq \frac{\bigl(e^{-t\lambda_1\tau}+e^{-t\lambda_2\tau}+e^{t}+e^{-t(1+\tau(1+\lambda_1+\lambda_2))}\bigr)}{(e^t-1)(e^{t\tau}-1)} - \frac{e^{-t/2(1+\tau(1+\lambda_1+\lambda_2))}}{(e^{t/2}-1)(e^{t\tau/2}-1)}, \\
\mathfrak{g}(t)\triangleq \frac{t^2}{(e^{t/2}-1)(e^{t\tau/2}-1)} -
\frac{(2t)^2}{{(e^t-1)(e^{t\tau}-1)}}.
\end{gather}
The functions $\mathfrak{f}$ and $\mathfrak{g}$ have the property
$\mathfrak{f}=O\bigl(t^{-1}\bigr),$ $\mathfrak{g}=O(t)$ as
$t\rightarrow 0$ and $\mathfrak{f}, \,\mathfrak{g}$ are
exponentially small as $t\rightarrow +\infty.$ Noticing the
individual existence and equality of the integrals
\begin{equation}
\int\limits_0^\infty \frac{dt}{t}
\Bigl[\frac{t^2\,\bigl(e^{-t/2(1+\tau(1+\lambda_1+\lambda_2))}-1\bigr)}{(e^{t/2}-1)(e^{t\tau/2}-1)}\Bigr]
= \int\limits_0^\infty \frac{dt}{t}
\Bigl[\frac{(2t)^2\,\bigl(e^{-t(1+\tau(1+\lambda_1+\lambda_2))}-1\bigr)}{{(e^t-1)(e^{t\tau}-1)}}\Bigr],
\end{equation}
we can write
\begin{equation}
\log\mathfrak{M}(q\,|\,\mu,\lambda_1,\lambda_2) =
q\int\limits_0^\infty \frac{dt}{t} \,B(t) +\int\limits_0^\infty
\frac{dt}{t} \Bigl[\bigl(e^{tq}-1-tq\bigr)\,\mathfrak{f}(t)\Bigr]
+\frac{q^2}{2}\int\limits_0^\infty \frac{dt}{t} \,\mathfrak{g}(t).
\end{equation}
Denote
\begin{equation}
\sigma^2(\tau)\triangleq \int\limits_0^\infty \frac{dt}{t}
\,\mathfrak{g}(t).
\end{equation}
Define the function $\mathcal{M}_{(\tau, \lambda_1,
\lambda_2)}(u)\triangleq -\int_u^\infty \mathfrak{f}(t)\,dt/t$ for
$u>0$ and $\mathcal{M}_{(\tau, \lambda_1, \lambda_2)}(u)\triangleq 0$
for $u<0.$ We have thus established the decomposition for
$\Re(q)<\tau,$
\begin{equation}
\log\mathfrak{M}(q\,|\,\tau,\lambda_1,\lambda_2) =
q\int\limits_0^\infty \frac{dt}{t} B(t)+\frac{1}{2}q^2\sigma^2(\mu)+
\int\limits_{\mathbb{R}\setminus\{0\}} \bigl(e^{uq}-1-uq\bigr)
\,d\mathcal{M}_{(\tau, \lambda_1, \lambda_2)}(u). \label{Levydecom}
\end{equation}
Finally, since $\tau$ satisfies $\tau>1,$ we have the
inequalities
\begin{gather}
\mathfrak{f}(t)>0 \,\, \text{for} \,\,t>0,   \label{fposit} \\
\sigma^2(\tau)>0.
\end{gather}
Their validity can be established as follows. Let
$a=e^{-t\lambda_1\tau/2},$ $b=e^{-t\lambda_2\tau/2},$ $c=e^{-t/2},$
and $d=e^{-t\tau/2}.$ Note that $a, b\geq 0$ and $0<c, d<1.$ Then,
Eq. \eqref{fposit} is equivalent to
\begin{equation}
a^2+b^2+c^{-2}+a^2b^2c^2d^2-ab(1+c)(1+d)>0. 
\end{equation}
The latter inequality follows from
\begin{align}
a^2+b^2+c^{-2}+a^2b^2c^2d^2-ab(1+c)(1+d) & = (a-b)^2+c^{-2} +
ab(1-c)(1-d) + \nonumber
\\ & +(abcd-1)^2-1\geq c^{-2}-1>0.
\end{align}
The integral for $\sigma^2(\tau)$ can be computed explicitly by means of Eq. \eqref{Frullani}.
\begin{align}
\int\limits_0^\infty \frac{dt}{t} \,\mathfrak{g}(t) & =
\int\limits_0^\infty \frac{dt}{t}
\,\Bigl[\frac{t^2}{(e^{t/2}-1)(e^{t\tau/2}-1)} -
\frac{4}{\tau}e^{-t}\Bigr] - \int\limits_0^\infty \frac{dt}{t}
\,\Bigl[\frac{(2t)^2}{{(e^t-1)(e^{t\tau}-1)}}-\frac{4}{\tau}e^{-t}\Bigr],
\nonumber
\\ & = \frac{4}{\tau}\,\int\limits_0^\infty \frac{dt}{t}
\,\Bigl[e^{-t}-e^{-2t}\Bigr] = \frac{4}{\tau}\log 2.
\end{align}
Hence $\sigma^2(\tau)>0$ and $\mathcal{M}_{(\tau, \lambda_1,
\lambda_2)}(u)$ is continuous and non-decreasing on $(-\infty,\,0)$
and $(0,\,\infty)$ and satisfies the integrability and limit
conditions $\int_{[-1,\,1]\setminus\{0\}} u^2\,d\mathcal{M}_{(\tau,
\lambda_1, \lambda_2)}(u)<\infty,$ $\lim\limits_{u\rightarrow
\pm\infty} \mathcal{M}_{(\tau, \lambda_1,
\lambda_2)}(u) =0$ so that 
$\mathcal{M}_{(\tau, \lambda_1, \lambda_2)}(u)$ is a valid spectral
function.\footnote{Note that it is only the mean $\int_0^\infty B(t)
\,dt/t$ that necessitates $\tau>1.$ The gaussian component
$\sigma^2(\tau)$ and spectral function $\mathcal{M}_{(\mu, \lambda_1,
\lambda_2)}(u)$ satisfy the required properties for all $\tau>0.$ It is this property that allows us to define
the critical Selberg integral distribution, cf. Section \ref{DerivM}.}
The decomposition in Eq. \eqref{Levydecom} assumes the canonical form, confer
Theorem 4.4 in Chap. 4 of \cite{SteVHar}, in the case of purely
imaginary $q$ by a trivial change in the linear term.\qed
\end{proof}

\begin{proof}[Proof of Corollary \ref{IntExistRV}.] 
This follows from Theorem \ref{IntExist}
by some general properties of analytic and infinitely divisible
characteristic functions. Denote the random variable corresponding
to the L\'evy-Khinchine decomposition in Theorem  \ref{IntExist}
by $X.$ As it has a nonzero
gaussian component, it is absolutely continuous with a bounded,
continuous, and zero-free density by Theorem 8.4 and Corollary 8.8
in Chap. 4 of \cite{SteVHar}. Denote the probability density of
$X$ by $f(x),$ $x\in\mathbb{R}.$ The Fourier transform of
$f(x)$ is
$\mathfrak{M}(iq\,|\,\mu,\lambda_1,\lambda_2)$ by construction,
which is analytic as a function of $q$ in the strip $\Im(q)>-\tau.$
Then, by the fundamental theorem of analytic characteristic
functions, confer Theorem 7.1.1 in Chap. 7 of \cite{Lukacs}, we have
for all $q$ in the strip of analyticity $\Im(q)>-\tau$
\begin{equation}
\mathfrak{M}(iq\,|\,\mu,\lambda_1,\lambda_2) =
\int\limits_\mathbb{R} e^{iqx}\,f(x)\,dx.
\end{equation}
On the other hand, the random variable
$M_{(\mu,\lambda_1,\lambda_2)}\triangleq \exp(X)$ has the density $f(\log y)/y,$ $y\in (0, \,\infty),$ so that the right-hand
side of this equation is precisely the Mellin transform of $M_{(\mu,\lambda_1,\lambda_2)}$ for
$\Re(q)<\tau$,
\begin{equation}
\mathfrak{M}(q\,|\,\mu,\lambda_1,\lambda_2) = \int\limits_0^\infty
y^q\,f_{(\mu, \lambda_1, \lambda_2)}(\log y)\,\frac{dy}{y} = {\bf
E}\bigl[M_{(\mu,\lambda_1,\lambda_2)}^q\bigr]
\end{equation}
as seen upon relabeling $iq\rightarrow q$ and changing variables $y
= e^x.$
\end{proof}

\begin{proof}[Proof of Theorem \ref{IntAsymp}.]
The key element in the proof is the integral in Eq. \eqref{Iintegral} and its connection with the ratio of $G(z\,|\,\tau)$ functions in
Eq. \eqref{IfuncG} and its asymptotic expansion in Eq. \eqref{IfuncGAsymptotic}. As in the proof of Theorem \ref{CirAsymp}
 it remains to apply these equations 
to each of the ratios of the $G-$factors in Eq. \eqref{MGI} and collect the terms. 
The only other term in Eq. \eqref{logMellinAsympI}, which has a structure that is
different from the series in Eq. \eqref{IfuncGAsymptotic}, can be treated using the
elementary identity
\begin{equation}
\sum\limits_{r=1}^\infty \frac{\zeta(r+1)}{r+1}/\tau^{r+1} =
\log\Gamma\bigl(1-1/\tau\bigr)+\frac{\psi(1)}{\tau}, \, |\tau|>1.
\end{equation}
The result now follows by a straightforward algebraic reduction.
\qed
\end{proof}


\section{Barnes Beta Distributions}\label{BarnesBeta}
\noindent 
We now proceed to review the theory of Barnes beta probability distributions following \cite{Me13}, \cite{Me14},  and \cite{Me16}. 
For the immediate needs of this paper we only need these distributions of types $M=N=2$ and $M=1, N=0.$ 
Nonetheless, we review the general case $M\leq N+1$ as it requires the same amount of effort as the special cases of interest.
The cases of $M<N,$ $M=N, $ and $M=N+1$ are all constructed in the same way but have somewhat different properties.
For this reason the first two cases $M<N,$ $M=N, $ and the third case $M=N+1$ are treated separately. As we will see below,
Barnes beta distributions have the property that their moments are expressed as products of ratios of multiple
gamma functions, whereas their ratios can have moments in the form of products of ratios of multiple sine functions.  The proofs of all results
in this section are given in the Appendix.

\subsection{$M\leq N$}
Define the action of the combinatorial operator $\mathcal{S}_N$ on a function \(h(x)\) by
\begin{definition}\label{Soperator}
\begin{equation}\label{S}
(\mathcal{S}_Nh)(q\,|\,b) \triangleq \sum\limits_{p=0}^N (-1)^p
\sum\limits_{k_1<\cdots<k_p=1}^N
h\bigl(q+b_0+b_{k_1}+\cdots+b_{k_p}\bigr).
\end{equation}
\end{definition}
\noindent In other words, in Eq. \eqref{S} the action of $\mathcal{S}_N$
is defined as an alternating sum over all combinations of $p$
elements for every $p=0\cdots N.$ 
\begin{definition}\label{bdef}
Given $q\in\mathbb{C}-(-\infty, -b_0],$ \(a=(a_1\cdots a_{M})\), \(b=(b_0,b_1\cdots b_{N}),\) let\footnote{
We will abbreviate $\bigl(\mathcal{S}_N \log\Gamma_M\bigr)(q\,|a,\,b)$ to mean the action of $\mathcal{S}_N$ on
$\log\Gamma_M(x|a),$ \emph{i.e.} $\bigl(\mathcal{S}_N \log\Gamma_M(x|a)\bigr)(q\,|\,b).$}
\begin{equation}\label{eta}
\eta_{M,N}(q\,|a,\,b) \triangleq \exp\Bigl(\bigl(\mathcal{S}_N
\log\Gamma_M\bigr)(q\,|a,\,b) - \bigl(\mathcal{S}_N \log\Gamma_M\bigr)(0\,|a,\,b)\Bigr).
\end{equation}
\end{definition}
The function $\eta_{M,N}(q\,|a,\,b)$ is holomorphic over
$q\in\mathbb{C}-(-\infty, -b_0]$ and equals a product of ratios of
multiple gamma functions by construction. Specifically,
\begin{align}
\eta_{M,N}(q\,|a, b) 
= & \frac{\Gamma_M(q+b_0|a)}{\Gamma_M(b_0|a)}\prod\limits_{j_1=1}^N \frac{\Gamma_M(b_0+b_{j_1}|a)}{\Gamma_M(q+b_0+b_{j_1}|a)}
\prod\limits_{j_1<j_2}^N \frac{\Gamma_M(q+b_0+b_{j_1}+b_{j_2}|a)}{\Gamma_M(b_0+b_{j_1}+b_{j_2}|a)}\times \nonumber \\
&\times\prod\limits_{j_1<j_2<j_3}^N \frac{\Gamma_M(b_0+b_{j_1}+b_{j_2}+b_{j_3}|a)}{\Gamma_M(q+b_0+b_{j_1}+b_{j_2}+b_{j_3}|a)} \cdots,
\end{align}
until all the $N$ indices are exhausted. The function $\log\eta_{M,N}(q\,|a,\,b)$ has an important integral representation that follows from 
that of $\log\Gamma_M(w|a)$ in Eq. \eqref{key}.
\begin{theorem}[Existence and Structure]\label{main}
Given $M, N\in\mathbb{N}$ such that $M\leq N,$ the function
$\eta_{M,N}(q\,|a,\,b)$ is the Mellin transform of a probability
distribution on $(0, 1].$ Denote it by $\beta_{M, N}(a,b).$ Then,
\begin{equation}
{\bf E}\bigl[\beta_{M, N}(a,b)^q\bigr] = \eta_{M, N}(q\,|a,\,b),\;
\Re(q)>-b_0.
\end{equation}
The distribution $-\log\beta_{M, N}(a,b)$ is infinitely divisible on
$[0, \infty)$ and has the L\'evy-Khinchine decomposition for $\Re(q)>-b_0,$
\begin{equation}\label{LKH}
{\bf E}\Bigl[\exp\bigl(q\log\beta_{M, N}(a,b)\bigr)\Bigr] =
\exp\Bigl(\int\limits_0^\infty (e^{-tq}-1) e^{-b_0
t} \frac{
\prod\limits_{j=1}^N (1-e^{-b_j t})}{\prod\limits_{i=1}^M (1-e^{-a_i t})}
\frac{dt}{t} \Bigr).
\end{equation}
$\log\beta_{M,N}(a,b)$ is absolutely continuous if and only if $M=N.$ If $M<N,$
$-\log\beta_{M,N}(a,b)$ is compound Poisson and
\begin{subequations}
\begin{align}
{\bf P}\bigl[\beta_{M,N}(a,b)=1\bigr] & =
\exp\Bigl(-\int\limits_0^\infty e^{-b_0 t}\frac{
\prod\limits_{j=1}^N (1-e^{-b_j t})}{\prod\limits_{i=1}^M (1-e^{-a_i t})}
\frac{dt}{t}\Bigr), \label{Pof1} \\
& = \exp\bigl(-(\mathcal{S}_N \log\Gamma_M)(0\,|a,\,b)\bigr). \label{Pof12}
\end{align}
\end{subequations}
\end{theorem}
\begin{corollary}\label{momentproblembarnes}
The Stieltjes moment problem for the positive integer moments of $\beta_{M,N}(a,b)$ is determinate.
\end{corollary}
It is worth emphasizing that the integral representation of $\log\eta_{M,N}(q\,|a,\,b)$ in Eq. \eqref{LKH} is the main result as it 
automatically implies that $\beta_{M, N}(a,b)$ is a valid probability distribution having
infinitely divisible logarithm, see Chapter 3 of \cite{SteVHar} for background material
on infinitely divisible distributions on $[0, \infty).$ 
\begin{theorem}[Asymptotics]\label{Asymptotics}
If $M<N$ and $|\arg(q)|<\pi,$
\begin{equation}\label{ourasym}
\lim\limits_{q\rightarrow\infty} \eta_{M, N}(q\,|\,a, b) =
\exp\bigl(-(\mathcal{S}_N \log\Gamma_M)(0\,|\,a, b)\bigr).
\end{equation}
If $M=N$ and $|\arg(q)|<\pi,$
\begin{equation}\label{ourasymN}
\eta_{N, N}(q\,|\,a, b) = \exp\bigl(-(b_1\cdots b_N/a_1\cdots a_M) \log(q) +
O(1)\bigr), \;q\rightarrow\infty.
\end{equation}
\end{theorem}
The Mellin transform of Barnes beta distributions satisfies a function equation 
and two remarkable factorizations 
that are inherited from those of the multiple gamma function.
\begin{theorem}[Functional equation]\label{FunctEquat}
$1\leq M\leq N,$ $q\in\mathbb{C}-(-\infty, -b_0],$ $i=1\cdots M,$
\begin{equation}
\eta_{M, N}(q+a_i\,|\,a,\,b) =
\eta_{M, N}(q\,|\,a,\,b)\,\exp\bigl(-(\mathcal{S}_N
\log\Gamma_{M-1})(q\,|\,\hat{a}_i, b)\bigr). \label{fe1}
\end{equation}
\end{theorem}
\begin{corollary}[Symmetries]\label{FunctSymmetry}
$1\leq M\leq N,$ $q\in\mathbb{C}-(-\infty, -b_0],$ $i=1\cdots M,$
$j=1\cdots N,$
\begin{align}
\eta_{M, N}(q\,|\,a,\,b_0+x)\,\eta_{M, N}(x\,|\,a,\,b) & = \eta_{M, N}(q+x\,|\,a,\,b), \label{fe2}\\
\eta_{M, N}(q\,|\,a,\,b)\,\eta_{M, N-1}(q\,|\,a,\,b_0+b_j,
\hat{b}_j) & =
\eta_{M, N-1}(q\,|\,a,\,\hat{b}_j), \label{fe3}\\
\eta_{M, N}(q+a_i\,|\,a,\,b)\,\eta_{M-1, N}(q\,|\,\hat{a}_i,\,b) & =
\eta_{M, N}(q\,|\,a,\,b)\,\eta_{M, N}(a_i\,|\,a,\,b), \label{fe1eq}\\
\eta_{M, N}(q\,|\,a,\,b_j+a_i) \,\eta_{M-1,
N-1}(b_j\,|\,\hat{a}_i,\,\hat{b}_j) & = \eta_{M, N}(q\,|\,a,\,b)\,
\eta_{M-1, N-1}(q+b_j\,|\,\hat{a}_i,\,\hat{b}_j), \label{fe4} \\
\eta_{M, N}(q+a_i\,|\,a,\,b) \, \eta_{M-1,
N-1}(q\,|\,\hat{a}_i,\,\hat{b}_j) & = \eta_{M, N}(q\,|\,a,\,b) \,
\eta_{M-1, N-1}(q+b_j\,|\,\hat{a}_i,\,\hat{b}_j).
\label{funceqsymmetry}
\end{align}
\end{corollary}
\begin{corollary}[Factorizations]\label{FactorizBarnes}
Let $\Omega\triangleq \sum_{i=1}^M n_i \, a_i.$
\begin{align}
\eta_{M,N}(q\,|\,a, b)  = & \prod\limits_{k=0}^\infty \frac{\eta_{M-1,N}(q+k
a_i\,|\,\hat{a}_i, b)}{\eta_{M-1,N}(k
a_i\,|\,\hat{a}_i, b)}, \label{infinprod2} \\
\eta_{M,N}(q\,|\,a, b) =  & \prod\limits_{n_1,\cdots ,n_M=0}^\infty \Bigl[
\frac{b_0+\Omega}{q+b_0+\Omega}\prod\limits_{j_1=1}^N \frac{q+b_0+b_{j_1}+\Omega}
{b_0+b_{j_1}+\Omega} \prod\limits_{j_1<j_2}^N \frac{b_0+b_{j_1}+b_{j_2}+\Omega}
{q+b_0+b_{j_1}+b_{j_2}+\Omega} \times \nonumber \\
 & \times
\prod\limits_{j_1<j_2<j_3}^N \frac{q+b_0+b_{j_1}+b_{j_2}+b_{j_3}+\Omega}{b_0+b_{j_1}+b_{j_2}+b_{j_3}+\Omega}\cdots\Bigr].
\label{infbarnesfac}
\end{align}
Probabilistically, these factorizations are equivalent to, respectively, 
\begin{align}
\beta_{M, N}(a,\,b) \overset{{\rm in \,law}}{=}&
\prod\limits_{k=0}^\infty \beta_{M-1, N}(\hat{a}_i,\,b_0+ka_i, \, \,b_1,\cdots, b_N), \label{probshinfac} \\
\beta_{M, N}(a,\,b) \overset{{\rm in \,law}}{=}&
\prod\limits_{n_1,\cdots ,n_M=0}^\infty \beta_{0, N}(b_0+\Omega,\,\,b_1,\cdots, b_N). \label{probbarnesfac}
\end{align}
\end{corollary}
We note that the factorizations in Eqs. \eqref{infinprod2} and \eqref{infbarnesfac} correspond to the Shintani and Barnes factorizations of the multiple gamma function, see Eqs. \eqref{generalfactorization} and \eqref{barnes}, respectively.  
The functional equation in Eq. \eqref{fe1} gives us the moments.
\begin{corollary}[Moments]\label{moments}
Assume $a_i=1$ for some $i.$ 
Let $k\in\mathbb{N}.$
\begin{align}
{\bf E}\bigl[\beta_{M, N}(a, b)^{k }\bigr] &  =
\exp\Bigl(-\sum\limits_{l=0}^{k-1} \bigl(\mathcal{S}_N
\log\Gamma_{M-1}\bigr)(l \,|\,\hat{a}_i, b)\Bigr), \label{posmom} \\
{\bf E}\bigl[\beta_{M, N}(a, b)^{-k }\bigr] & =
\exp\Bigl(\sum\limits_{l=0}^{k-1} \bigl(\mathcal{S}_N
\log\Gamma_{M-1}\bigr)(-(l+1) \,|\,\hat{a}_i, b)\Bigr), \; k<b_0. \label{negmom}
\end{align}
\end{corollary}
The scaling property in Eq. \eqref{scale} gives us the scaling invariance.
\begin{theorem}[Scaling invariance]\label{barnesbetascaling}
Let $\kappa>0.$ Then,
\begin{equation}
\beta^{\kappa}_{M, N}(\kappa\,a, \kappa\,b) \overset{{\rm in \,law}}{=}\beta_{M,
N}(a,\,b).
\end{equation}
\end{theorem}

\begin{corollary}[Barnes factorization for $a_i=1$]\label{BarnesFactorSpecial} Let $0\leq M\leq N$ and $a_i=1$ for all $i=1\cdots M.$ Then, 
\begin{align}
\eta_{M,N}(q\,|\,1, b) & = \prod\limits_{k=0}^\infty \Bigl[
\frac{b_0+k}{q+b_0+k}\prod\limits_{j_1=1}^N \frac{q+b_0+b_{j_1}+k}
{b_0+b_{j_1}+k} \prod\limits_{j_1<j_2}^N \frac{b_0+b_{j_1}+b_{j_2}+k}
{q+b_0+b_{j_1}+b_{j_2}+k} \times \nonumber \\
& \times \prod\limits_{j_1<j_2<j_3}^N \frac{q+b_0+b_{j_1}+b_{j_2}+b_{j_3}+k}{b_0+b_{j_1}+b_{j_2}+b_{j_3}+k}\cdots\Bigr]^{(k\,|\,M)},
\label{barnesfactorspecial} \\
(k\,|\,M) & \triangleq \sum\limits_{m=1}^M \binom{k-1}{m-1}\binom{M}{m}.
\end{align}
\end{corollary}

We will next present the special case of $M=N=2$ in order to illustrate the general
theory with a concrete yet quite non-trivial example. In addition,
this case is also of a particular interest in the probabilistic
theory of the Selberg and Morris integral probability distributions that we will review in Section \ref{Probabilistic}. Let
$a_1=1$ and $a_2=\tau>0$ and write $\beta_{2, 2}(\tau, b),$
$\eta_{2,2}(q\,|\,\tau, b),$ and
$\Gamma_2\bigl(w\,|\,(1,\tau)\bigr)=\Gamma_2(w\,|\,\tau)$ for
brevity. From Eq. \eqref{eta} and Theorem \ref{main} we have
${\bf E}\bigl[\beta_{2, 2}(\tau, b)^q\bigr] = \eta_{2,2}(q\,|\,\tau,
b)$ for $\Re(q)>-b_0$ and
\begin{equation}
\eta_{2,2}(q\,|\,\tau, b) =
\frac{\Gamma_2(q+b_0\,|\,\tau)}{\Gamma_2(b_0\,|\,\tau)}
\frac{\Gamma_2(b_0+b_1\,|\,\tau)}{\Gamma_2(q+b_0+b_1\,|\,\tau)}
\frac{\Gamma_2(b_0+b_2\,|\,\tau)}{\Gamma_2(q+b_0+b_2\,|\,\tau)}
\frac{\Gamma_2(q+b_0+b_1+b_2\,|\,\tau)}{\Gamma_2(b_0+b_1+b_2\,|\,\tau)}.
\end{equation}
The asymptotic behavior of $\eta_{2,2}(q\,|\,\tau, b)$ 
follows from Theorem \ref{Asymptotics}.
\begin{equation}
\eta_{2, 2}(q\,|\,\tau, b) = \exp\Bigl(-\frac{b_1 b_2}{\tau}\log(q)
+ O(1)\Bigr), \; q\rightarrow\infty,\;|\arg(q)|<\pi.
\end{equation}
Using Eq. \eqref{gamma1}, the functional equation in Theorem
\ref{FunctEquat} takes the form
\begin{align}
\eta_{2, 2}(q+1\,|\,\tau, b) & = \eta_{2, 2}(q\,|\,\tau, b)\,
\frac{\Gamma\bigl((q+b_0+b_1)/\tau\bigr)\Gamma\bigl((q+b_0+b_2)/\tau\bigr)}
{\Gamma\bigl((q+b_0)/\tau\bigr)\Gamma\bigl((q+b_0+b_1+b_2)/\tau\bigr)},
\\
\eta_{2, 2}(q+\tau\,|\,\tau, b) & = \eta_{2, 2}(q\,|\,\tau, b)\,
\frac{\Gamma(q+b_0+b_1)\Gamma(q+b_0+b_2)}
{\Gamma(q+b_0)\Gamma\bigl(q+b_0+b_1+b_2)}.
\end{align}
The positive moments in Corollary \ref{moments} for $k\in\mathbb{N}$
are
\begin{align}
{\bf E}\bigl[\beta_{2, 2}(\tau, b)^k\bigr] & =
\prod\limits_{l=0}^{k-1}
\Bigl[\frac{\Gamma\bigl((l+b_0+b_1)/\tau\bigr)\,\Gamma\bigl((l+b_0+b_2)/\tau\bigr)}{\Gamma\bigl((l
+b_0)/\tau\bigr)\,\Gamma\bigl((l+b_0+b_1+b_2)/\tau\bigr)}\Bigr].
\end{align}
The negative moments are
\begin{align}
{\bf E}\bigl[\beta_{2, 2}(\tau, b)^{-k}\bigr] & =
\prod\limits_{l=0}^{k-1} \Bigl[\frac{\Gamma\bigl((-(l+1)
+b_0)/\tau\bigr)\,\Gamma\bigl((-(l+1)+b_0+b_1+b_2)/\tau\bigr)}{\Gamma\bigl((-(l+1)+b_0+b_1)/\tau\bigr)\,
\Gamma\bigl((-(l+1)+b_0+b_2)/\tau\bigr)}\Bigr], \; k<b_0.
\end{align}
The factorization equations in Corollary \ref{FactorizBarnes}
are
\begin{align} 
\eta_{2,2}(q\,|\,\tau, b) &  =
 \prod\limits_{k=0}^\infty\Bigl[
\frac{\Gamma((q+k+b_0)/\tau) }{\Gamma((k+b_0)/\tau)}
\frac{\Gamma((k+b_0+b_1)/\tau)}{\Gamma((q+k+b_0+b_1)/\tau)}
\frac{\Gamma((k+b_0+b_2)/\tau)}{\Gamma((q+k+b_0+b_2)/\tau)} \times
\nonumber \\ & \times
\frac{\Gamma((q+k+b_0+b_1+b_2)/\tau)}{\Gamma((k+b_0+b_1+b_2)/\tau)}\Bigr],
\\
\eta_{2,2}(q\,|\,\tau, b) & = \prod\limits_{k=0}^\infty\Bigl[
\frac{\Gamma(q+k\tau+b_0)}{\Gamma(k\tau+b_0)}
\frac{\Gamma(k\tau+b_0+b_1)}{\Gamma(q+k\tau+b_0+b_1)}
\frac{\Gamma(k\tau+b_0+b_2)}{\Gamma(q+k\tau+b_0+b_2)} \times
\nonumber \\ & \times
\frac{\Gamma(q+k\tau+b_0+b_1+b_2)}{\Gamma(k\tau+b_0+b_1+b_2)}\Bigr], \\
\eta_{2,2}(q\,|\,\tau, b) & = 
\prod\limits_{n_1, \,n_2=0}^\infty \Bigl[
\frac{b_0+n_1+n_2\tau}{q+b_0+n_1+n_2\tau} \frac{q+b_0+b_1+n_1+n_2\tau}
{b_0+b_1+n_1+n_2\tau} 
 \frac{q+b_0+b_2+n_1+n_2\tau}{b_0+b_2+n_1+n_2\tau}\times
\nonumber \\ & \times 
\frac{b_0+b_1+b_2+n_1+n_2\tau}{q+b_0+b_1+b_2+n_1+n_2\tau} \Bigr]
.
\end{align}
Finally, in the special case of $\tau=1$ we get from Corollary \ref{BarnesFactorSpecial},
\begin{equation}
\eta_{2,2}(q\,|\,1, b)  = 
\prod\limits_{k=0}^\infty \Bigl[
\frac{b_0+k}{q+b_0+k} \frac{q+b_0+b_1+k}
{b_0+b_1+k} 
 \frac{q+b_0+b_2+k}{b_0+b_2+k}
\frac{b_0+b_1+b_2+k}{q+b_0+b_1+b_2+k} \Bigr]^{k+1}
.\label{specialfactorization}
\end{equation}

\begin{remark} The analytic structure of $\eta_{M,N}(a,b)$ is not fully understood, even in the case of $M=0.$ The latter with $b_i=i,$ $i=1\cdots N,$ is of particular interest as it occurs in the context of the Riemann xi function, cf. Section 7 in \cite{Me14}.
\end{remark}


\subsection{$N=M-1$}

Let $M\in\mathbb{N},$ $a=(a_1,\cdots, a_{M}),$ and $b=(b_0, b_1,\cdots,b_{M-1}),$ all assumed to be positive.
In other words, $N=M-1$ in the sense of the previous subsection. 
Define
\begin{equation}
\eta_{M, M-1}(q|a, b) \triangleq \exp\Bigl( \bigl(\mathcal{S}_{M-1}
\log\Gamma_M\bigr)(q\,|a,\,b) - \bigl(\mathcal{S}_{M-1} \log\Gamma_M\bigr)(0\,|a,\,b) \Bigr). \label{etaL}
\end{equation}
For example, in the case of $M=2$ we have
\begin{equation}
\eta_{2, 1}(q|a, b) =
\frac{\Gamma_2(q+b_0\,|\,a)}{\Gamma_2(b_0\,|\,a)}
\frac{\Gamma_2(b_0+b_1\,|\,a)}{\Gamma_2(q+b_0+b_1\,|\,a)}.
\end{equation}
The case of $M=2$ was first treated in \cite{Kuz} and then studied in depth in \cite{LetSim} in the context of
the Mellin transform of certain functionals of the stable L\'evy process. We extended the theory to general $M$ in
\cite{Me16}.
\begin{theorem}[Existence and Structure]\label{mainsine}
Assume
\begin{equation}
\Re(q)>-b_0.
\end{equation}
$\eta_{M, M-1}(q\,|a,\,b)$ is the Mellin transform of a probability
distribution $\beta_{M, M-1}(a,b)$  on $(0, \infty).$
\begin{equation}
{\bf E}\bigl[\beta_{M, M-1}(a,b)^q\bigr] = \eta_{M, M-1}(q\,|a,\,b). 
\end{equation}
The distribution $\log\beta_{M, M-1}(a,b)$ is infinitely divisible and absolutely continuous on
$\mathbb{R}$ and has the L\'evy-Khinchine decomposition 
\begin{align}
{\bf E}\Bigl[\exp\bigl(q\log\beta_{M, M-1}(a,b)\bigr)\Bigr] = 
\exp&\Bigl(
\int\limits_0^\infty (e^{-tq}-1+qt) e^{-b_0
t} \frac{
\prod\limits_{j=1}^{M-1} (1-e^{-b_j t})}{\prod\limits_{i=1}^M (1-e^{-a_i t})}
\frac{dt}{t} +\nonumber \\&+ q
\int\limits_0^\infty \Bigl[\frac{e^{-t}}{t} \frac{\prod\limits_{j=1}^{M-1} b_j}{\prod\limits_{i=1}^{M} a_i}
-e^{-b_0
t}\frac{
\prod\limits_{j=1}^{M-1} (1-e^{-b_j t})}{\prod\limits_{i=1}^M (1-e^{-a_i t})}\Bigr]
dt
\Bigr).\label{LKHsine}
\end{align} 
$\eta_{M, M-1}(q\,|a,\,b)$ satisfies the functional equation in Eq. \eqref{fe1} and moment formulas in Eqs. \eqref{posmom}
and \eqref{negmom}. Given $\kappa>0,$ the scaling invariance property of $\beta_{M, M-1}(a,b)$ is
\begin{equation}\label{scalinvgen}
\beta^{\kappa}_{M, M-1}(\kappa\,a, \kappa\,b) \overset{{\rm in \,law}}{=}\kappa^{\prod\limits_{j=1}^{M-1} b_j/\prod\limits_{i=1}^M a_i} \;\beta_{M, M-1}(a,\,b).
\end{equation}
\end{theorem}
\begin{theorem}[Asymptotics]\label{newetaasympt}
Given  $|\arg(q)|<\pi,$
\begin{equation}\label{ourasymN}
\eta_{M, M-1}(q\,|\,b) = \exp\Bigl(\bigl(\prod\limits_{j=1}^{M-1} b_j/\prod\limits_{i=1}^M a_i\bigr) q\log(q) +
O(q)\Bigr), \;q\rightarrow\infty.
\end{equation}
\end{theorem}
It is expected that the Stieltjes moment problem for $\beta_{M, M-1}(a,b)$ is determinate (unique solution) iff 
\begin{equation}\label{condition}
\prod\limits_{j=1}^{M-1} b_j \leq 2\prod\limits_{i=1}^M a_i  .
\end{equation}
This is a classical result for $M=1.$ In general, the asymptotic behavior in Eq. \eqref{ourasymN} coincides with 
the asymptotic behavior of generalized gamma distributions and we expect that Eq. \eqref{condition} should follow from
the known solution to the Stieltjes moment problem for these distributions, see \cite{Stoyanov}.

It is also interesting to look at the ratio of two independent Barnes beta distributions of type $(M, M-1)$ as this ratio
has remarkable factorization properties. Let 
\begin{equation}\label{bdefine}
\bar{b} \triangleq \bigl(\bar{b}_0, b_1, \cdots b_{M-1}\bigr)
\end{equation}
for some fixed $\bar{b}_0>0,$
define
\begin{equation}
\beta_{M, M-1}(a,b, \bar{b})\triangleq \beta_{M, M-1}(a,b)\,\beta^{-1}_{M, M-1}(a,\bar{b}),
\end{equation}
and denote its Mellin transform by $\eta_{M, M-1}(q\,|\,a, b, \bar{b}).$
Then, the Mellin transform of 
$\beta_{M, M-1}(a,b, \bar{b})$ 
satisfies 
two factorizations and scaling invariance that are similar to those of Barnes beta distributions for $M\leq N.$
\begin{theorem}[Properties]\label{FunctEquatSine}
Let $M\in\mathbb{N},$ $\bar{b}_0>\Re(q)>-b_0,$ $i=1\cdots M,$  $\Omega\triangleq \sum_{i=1}^M n_i \, a_i.$
\begin{align}
\eta_{M, M-1}(q\,|\,a, b, \bar{b})  = & \prod\limits_{k=0}^\infty \frac{\eta_{M-1,M-1}(q+k
a_i\,|\,\hat{a}_i, b)}{\eta_{M-1,M-1}(k
a_i\,|\,\hat{a}_i, b)} \,\frac{\eta_{M-1,M-1}(-q+k
a_i\,|\,\hat{a}_i, \bar{b})}{\eta_{M-1,M-1}(k
a_i\,|\,\hat{a}_i, \bar{b})}, \label{infinprod2L} \\
\eta_{M, M-1}(q\,|\,a, b, \bar{b}) = & \prod\limits_{n_1,\cdots ,n_M=0}^\infty \Bigl[
\frac{b_0+\Omega}{q+b_0+\Omega}\,\frac{\bar{b}_0+\Omega}{-q+\bar{b}_0+\Omega}
\times \nonumber \\
& \times
\prod\limits_{j_1=1}^{M-1} \frac{q+b_0+b_{j_1}+\Omega}
{b_0+b_{j_1}+\Omega} \,\frac{-q+\bar{b}_0+b_{j_1}+\Omega}
{\bar{b}_0+b_{j_1}+\Omega}
\times \nonumber \\
& \times\prod\limits_{j_1<j_2}^{M-1} \frac{b_0+b_{j_1}+b_{j_2}+\Omega}
{q+b_0+b_{j_1}+b_{j_2}+\Omega} \, \frac{\bar{b}_0+b_{j_1}+b_{j_2}+\Omega}
{-q+\bar{b}_0+b_{j_1}+b_{j_2}+\Omega} \cdots\Bigr].\label{infinprod1L}
\end{align}
Probabilistically, these factorizations are equivalent to, respectively, 
\begin{align}
\beta_{M, M-1}(a,\,b, \bar{b}) \overset{{\rm in \,law}}{=}&
\prod\limits_{k=0}^\infty \beta_{M-1, M-1}(\hat{a}_i,\,b_0+ka_i,\,\,b_1,\cdots, b_{M-1})\times \nonumber \\ & \times \beta_{M-1, M-1}^{-1}(\hat{a}_i,\,\bar{b}_0+ka_i,\,b_1,\cdots, b_{M-1}), \label{probshin}\\
\beta_{M, M-1}(a,\,b, \bar{b}) \overset{{\rm in \,law}}{=}&
\prod\limits_{n_1,\cdots ,n_M=0}^\infty \beta_{0, M-1}(b_0+\Omega,\,b_1,\cdots, b_{M-1})\times \nonumber \\ & \times\beta_{0, M-1}^{-1}(\bar{b}_0+\Omega, \,\,b_1,\cdots, b_{M-1}). \label{probbarnes}
\end{align}
Let $\kappa>0.$ Then,
\begin{equation}\label{scalinvL}
\beta^{\kappa}_{M, M-1}(\kappa\,a, \kappa\,b, \kappa\,\bar{b}) \overset{{\rm in \,law}}{=}\beta_{M, M-1}(a,\,b, \bar{b}).
\end{equation}
\end{theorem}

We will illustrate the general theory with the special case of $\beta_{1,0}.$ 
\begin{align}
{\bf E}[\beta_{1,0}(a, b)^q] = & \frac{\Gamma_1(q+b_0\,|\,a)}{\Gamma_1(b_0\,|\,a)}, \\
= & a^{\frac{q}{a}} \frac{\Gamma(\frac{q+b_0}{a})}{\Gamma(\frac{b_0}{a})}, 
\end{align}
by Eq. \eqref{gamma1} so that $\beta_{1,0}$ is a Fr\'echet distribution. This distribution plays an important role
in the structure of both Selberg and Morris integral probability distributions. For example, the $Y$ distribution in Eq. \eqref{Ydist} is
\begin{equation}
Y = \tau^{1/\tau}\,\beta^{-1}_{1,0}(a=\tau, b_0=\tau). 
\end{equation}

\subsection{Ratios and Multiple Sine Functions}
\noindent 
Recall the definition of the multiple sine function in Eq. \eqref{msinedef}.
Let 
\begin{equation}\label{b0spec}
\bar{b}_0 = |a|-\sum_{j=0}^{N} b_j>0.
\end{equation}
We then have the identity
\begin{equation}
(\mathcal{S}_N \log S_M)(q\,|\,a, b) = (-1)^{N+M} (\mathcal{S}_N \log\Gamma_M)(-q\,|\,a, \bar{b}) - (\mathcal{S}_N \log\Gamma_M)(q\,|\,a, b),
\end{equation} 
where $\bar{b}$ is as in Eq. \eqref{bdefine}.
It follows that we obtain the interesting identity,
\begin{align}
\exp\Bigl(\bigl(\mathcal{S}_{N}
\log S_M\bigr)(0\,|a,\,b) - \bigl(\mathcal{S}_{N} \log S_M\bigr)(q\,|a,\,b)\Bigr) = &  \frac{S_M(b_0|a)}{S_M(q+b_0|a)}\prod\limits_{j_1=1}^{N} \frac{S_M(q+b_0+b_{j_1}|a)}{S_M(b_0+b_{j_1}|a)} \times\nonumber \\
&\times
\prod\limits_{j_1<j_2}^{N} \frac{S_M(b_0+b_{j_1}+b_{j_2}|a)}{S_M(q+b_0+b_{j_1}+b_{j_2}|a)} \times\nonumber \\
&\times\prod\limits_{j_1<j_2<j_3}^{N} \frac{S_M(q+b_0+b_{j_1}+b_{j_2}+b_{j_3}|a)}{S_M(b_0+b_{j_1}+b_{j_2}+b_{j_3}|a)} \cdots, 
\nonumber \\
= & \eta_{M,N}(q\,|\,a, b) \eta_{M,N}(-q\,|\,a, \bar{b})^{(-1)^{M+N+1}}.
\end{align}
In particular, when $N=M-1,$ this identity becomes
\begin{align}
\exp\Bigl(\bigl(\mathcal{S}_{M-1}
\log S_M\bigr)(0\,|a,\,b) - \bigl(\mathcal{S}_{M-1} \log S_M\bigr)(q\,|a,\,b)\Bigr)  = & \eta_{M,M-1}(q\,|\,a, b) \eta_{M,M-1}(-q\,|\,a, \bar{b}), 
\nonumber \\
= & \eta_{M, M-1}(q|a, b, \bar{b}),
\end{align}
and so is the Mellin transform of the ratio of two independent Barnes beta random variables of type $(M, M-1).$ 

We will illustrate this result with the special cases of $\beta_{1,0}$ and $\beta_{2,1}.$ Let $M=1.$
Let $\bar{b}_0=a-b_0$ as in Eq. \eqref{b0spec},
\begin{equation}
{\bf E}[\beta_{1,0}(a, b, \bar{b})^q] = \frac{\sin(\frac{\pi b_0}{a})}{\sin(\frac{\pi (q+b_0)}{a})}.
\end{equation}
Now, let $M=2,$ $a=(a_1, a_2),$  $b=(b_0, b_1),$ $\bar{b}=(a_1+a_2-b_0-b_1, b_1).$
\begin{equation}
{\bf E}[\beta_{2,1}(a, b, \bar{b})^q] =  \frac{S_2(b_0\,|\,a)}{S_2(q+b_0\,|\,a)}\frac{S_2(q+b_0+b_1\,|\,a)}{S_2(b_0+b_1\,|\,a)}.
\end{equation}

\section{Morris and Selberg Integral Distributions: Probabilistic Approach}\label{Probabilistic}
\noindent 
In this section we will re-consider the problem of finding positive probability distributions having the property
that their positive integer moments are given by the Morris and Selberg integrals, respectively. Throughout this section we let $\tau>1$
and restrict $\lambda_1, \lambda_2\geq 0$ for simplicity.  As we did in Sections \ref{CirAnalytical} and \ref{IntAnalytical}, we write $\tau$ as an abbreviation of $a=(1,\tau)$ in the list of parameters of the double gamma function and $\beta_{2,2}(a, b).$ The proofs of all results
in this section are given in Section \ref{ProbProofs}.

\subsection{Morris Integral Probability Distribution}

\begin{theorem}[Existence and Properties]\label{theoremcircle}
The function 
\begin{align}\label{thefunctioncopy}
\mathfrak{M}(q\,|\tau,\,\lambda_1,\,\lambda_2)=&\frac{(\tau^{\frac{1}{\tau}})^q}{\Gamma^q\bigl(1-\frac{1}{\tau}\bigr)}
\frac{\Gamma_2(\tau(\lambda_1+\lambda_2+1)+1-q\,|\,\tau)}{\Gamma_2(\tau(\lambda_1+\lambda_2+1)+1\,|\,\tau)}
\frac{\Gamma_2(-q+\tau\,|\,\tau)}{\Gamma_2(\tau\,|\,\tau)}\times \nonumber \\ & \times
\frac{\Gamma_2(\tau(1+\lambda_1)+1\,|\,\tau)}{\Gamma_2(\tau(1+\lambda_1)+1-q\,|\,\tau)}
\frac{\Gamma_2(\tau(1+\lambda_2)+1\,|\,\tau)}{\Gamma_2(\tau(1+\lambda_2)+1-q\,|\,\tau)}
\end{align}
is the Mellin transform of the distribution
\begin{align}
M_{(\tau, \lambda_1, \lambda_2)} = & \frac{ \tau^{1/\tau}}{\Gamma(1-1/\tau)} \,\beta^{-1}_{22}(\tau, b_0=\tau,\,b_1=1+\tau \lambda_1, \,b_2=1+\tau \lambda_2) \times \nonumber \\ & \times
\beta_{1,0}^{-1}(\tau, b_0=\tau(\lambda_1+\lambda_2+1)+1), \label{thedecompcircle}
\end{align}
where $\beta^{-1}_{22}(a, b)$ is the inverse Barnes beta of type $(2,2)$ and $\beta^{-1}_{1,0}(a, b)$ is the independent inverse Barnes beta of type $(1,0).$ In particular, $\log M_{(\tau, \lambda_1, \lambda_2)}$ is infinitely divisible and the Stieltjes moment problem for $M^{-1}_{(\tau, \lambda_1, \lambda_2)}$ is determinate (unique solution).
In the special case of $\lambda_1=\lambda_2=0,$ we have
\begin{gather}
\mathfrak{M}(q\,|\tau, 0, 0)=\frac{1}{\Gamma^q\bigl(1-\frac{1}{\tau}\bigr)}\Gamma\bigl(1-\frac{q}{\tau}\bigr), \\
M_{(\tau, 0, 0)} = \frac{ \tau^{1/\tau}}{\Gamma(1-1/\tau)}\beta_{1,0}^{-1}(\tau, b_0=\tau).
\end{gather}
\end{theorem}
The special case of $\lambda_1=\lambda_2=0$ was first treated in \cite{FyoBou} and corresponds to
the Dyson integral. The general case first appeared in \cite{Me16}.

\subsection{Selberg Integral Probability Distribution}

We begin with a ``master'' theorem (Theorem \ref{general}), which as we will see below is a corollary of Barnes multiplication in Eq. \eqref{multiplic}, 
followed by the main result in Theorem \ref{BSM} and a corollary pertaining to the Stieltjes moment problem for the Selberg integral distribution.

\begin{theorem}\label{general}
Let $a\triangleq(a_1,\,a_2),$ $a_i>0,$ and $x\triangleq
(x_1,\,x_2),$ $x_1,\,x_2>0,$ then 
\begin{equation}\label{eqgeneral}
{\bf E}\big[M_{(a, x)}^q\bigr] \triangleq
\frac{\Gamma_2(x_1-q\,|\,a)}{\Gamma_2(x_1\,|\,a)}
\frac{\Gamma_2(x_2-q\,|\,a)}{\Gamma_2(x_2\,|\,a)}
\frac{\Gamma_2(a_1+a_2-q\,|\,a)}{\Gamma_2(a_1+a_2\,|\,a)}
\frac{\Gamma_2(x_1+x_2-q\,|\,a)}{\Gamma_2(x_1+x_2-2q\,|\,a)}
\end{equation}
is the Mellin transform of a probability distribution $M_{(a, x)}$
on $(0,\,\infty).$ Let $L$ be lognormal
\begin{equation}\label{Ldefin}
L \triangleq \exp\bigl(\mathcal{N}(0,\,4\log 2/a_1a_2)\bigr),
\end{equation}
and let $X_1,\,X_2,\,X_3$ have the $\beta^{-1}_{2, 2}(a, b)$
distribution with the parameters
\begin{align}
X_1 &\triangleq \beta_{2,2}^{-1}\Bigl(a,
b_0=x_1,\,b_1=b_2=(x_2-x_1)/2\Bigl), \label{X1}\\
X_2 & \triangleq \beta_{2,2}^{-1}\Bigl(a,
b_0=(x_1+x_2)/2,\,b_1=a_1/2,\,b_2=a_2/2\Bigr), \label{X2}\\
X_3 & \triangleq \beta_{2,2}^{-1}\Bigl(a, b_0=a_1+a_2,\,
b_1=b_2=(x_1+x_2-a_1-a_2)/2\Bigl). \label{X3}
\end{align}
Then, $M_{(a, x)}$ has the factorization
\begin{equation}\label{generaldecomp}
M_{(a, x)} \overset{{\rm in \,law}}{=}
2^{-\bigl(2(x_1+x_2)-(a_1+a_2)\bigr)/a_1a_2}\, L\,X_1\,X_2\,X_3.
\end{equation}
In particular, $\log M_{(a, x)}$ is absolutely continuous and
infinitely divisible.
\end{theorem}

\begin{theorem}[Selberg Integral Probability Distribution]\label{BSM}
Let $\mathfrak{M}(q\,|\,\tau,\lambda_1,\lambda_2)$ be as in Eq. \eqref{thefunctioninterval}
for $\Re(q)<\tau.$ Then, $\mathfrak{M}(q\,|\,\tau,\lambda_1,\lambda_2)$ is the Mellin transform of a
probability distribution $M_{(\tau, \lambda_1, \lambda_2)}$ on $(0,\infty),$
\begin{equation}\label{Mint}
\mathfrak{M}(q\,|\,\tau,\lambda_1,\lambda_2) = {\bf E}\bigl[M_{(\tau,\lambda_1,\lambda_2)}^q\bigr], \;\Re(q)<\tau,
\end{equation}
$\log M_{(\tau, \lambda_1, \lambda_2)}$ is absolutely continuous and
infinitely divisible. 
Let 
\begin{equation}\label{xidef}
x_i(\tau,\lambda_i) \triangleq 1+\tau(1+\lambda_i),\,i=1,2,\,a_1=1, \,a_2=\tau,
\end{equation}
and $L$ and $X_i$ be as in Theorem \ref{general}
corresponding to $a(\tau)=(1,\tau)$ and $x(\tau,\lambda)=\bigl(x_1(\tau,\lambda_1), x_2(\tau,\lambda_2)\bigr),$
\begin{align}
L(\tau) & \triangleq \exp\bigl(\mathcal{N}(0,\,4\log 2/\tau)\bigr), \\
X_1(\tau, \lambda) &\triangleq \beta_{2,2}^{-1}\Bigl(\tau,
b_0=1+\tau+\tau\lambda_1,\,b_1=\tau(\lambda_2-\lambda_1)/2, \,
b_2=\tau(\lambda_2-\lambda_1)/2\Bigr),\\
X_2(\tau, \lambda) & \triangleq \beta_{2,2}^{-1}\Bigl(\tau,
b_0=1+\tau+\tau(\lambda_1+\lambda_2)/2,\,b_1=1/2,\,b_2=\tau/2\Bigr),\\
X_3(\tau, \lambda) & \triangleq \beta_{2,2}^{-1}\Bigl(\tau, b_0=1+\tau,\,
b_1=\frac{1+\tau+\tau\lambda_1+\tau\lambda_2}{2}, \,
b_2=\frac{1+\tau+\tau\lambda_1+\tau\lambda_2}{2}\Bigr),
\end{align}
\emph{i.e.} $\log L$ is a zero-mean normal with variance $4\log
2/\tau$ and $X_1,\,X_2,\,X_3$ have the $\beta^{-1}_{2,
2}(\tau, b)$ distribution with the specified parameters.
Define also the distribution $Y$ to be a power of the exponential,
\begin{align}
Y(\tau) \triangleq &\tau\,y^{-1-\tau}\exp\bigl(-y^{-\tau}\bigr)\,dy,\; y>0.\label{Ydist}
\end{align}
Then,
\begin{equation}\label{Decomposition}
M_{(\tau, \lambda_1, \lambda_2)} \overset{{\rm in \,law}}{=} 2\pi\,
2^{-\bigl[3(1+\tau)+2\tau(\lambda_1+\lambda_2)\bigr]/\tau}\,\Gamma\bigl(1-1/\tau\bigr)^{-1}\,
L\,X_1\,X_2\,X_3\,Y.
\end{equation}
\end{theorem}
\begin{corollary}\label{Uniqueness}
The Stieltjes moment problem for $M_{(\tau, \lambda_1, \lambda_2)}^{-1}$ is indeterminate.
Let $\widetilde{M}_{(\tau, \lambda_1, \lambda_2)}$ be a probability distribution on $(0,\,\infty)$ such that
\begin{equation}
\widetilde{M}_{(\tau, \lambda_1, \lambda_2)} \overset{{\rm in \,law}}{=} L\,N,
\end{equation}
\begin{equation}\label{NYdef}
L \triangleq \exp\bigl(\mathcal{N}(0,\,4\log 2/\tau)\bigr),
\end{equation}
\emph{i.e.} $\log L$ is a zero-mean normal with variance $4\log
2/\tau,$
and $N$ is some distribution that is independent of $L.$ If the negative moments of
$\widetilde{M}_{(\tau, \lambda_1, \lambda_2)}$ equal those of $M_{(\tau, \lambda_1, \lambda_2)}$ in Eq. \eqref{IntNegMoments},
then $\widetilde{M}_{(\tau, \lambda_1, \lambda_2)}\overset{{\rm in \,law}}{=}M_{(\tau, \lambda_1, \lambda_2)}.$
\end{corollary}

The last result that we will consider in this section has to do with the scaling invariance
in the context of the analytic extension of the Selberg integral in Theorem \ref{BSM}. Specifically, we are interested in the behavior of the decomposition in Eq. \eqref{Decomposition} under the involution $\tau\rightarrow 1/\tau.$ Let $a$ and $x$ be as in Eq. \eqref{xidef}.
\begin{theorem}\label{scalinvar} 
The distributions $L(\tau),$ $X_i(\tau,\lambda),$ and $M_{(a, x)}$ are involution invariant under $\tau\rightarrow 1/\tau$ and
$\lambda\rightarrow \tau\lambda.$
\begin{align}
L^{1/\tau}\bigl(\frac{1}{\tau}\bigr) &\overset{{\rm in \,law}}{=} L(\tau), \label{Linvar} \\
X^{1/\tau}_i\bigl(\frac{1}{\tau}, \tau\lambda\bigr) &\overset{{\rm in \,law}}{=}
X_i(\tau, \lambda), \;i=1,2,3, \label{Xinvar} \\
M^{1/\tau}_{\bigl(a(1/\tau), \,x(1/\tau,\tau\lambda)\bigr)} &\overset{{\rm in \,law}}{=} M_{\bigl(a(\tau), \,x(\tau,\lambda)\bigr)}. \label{Minvar}
\end{align}
\end{theorem}
\begin{corollary}\label{mycorollary}
The Mellin transform of the Selberg integral distribution satisfies the identity in Theorem \ref{Mtransforminvol}.
\end{corollary}
Theorem \ref{BSM} and Corollary \ref{Uniqueness} first appeared in \cite{MeIMRN}, where we discovered the decomposition in Eq. \eqref{Decomposition}, which led to the development of the theory of Barnes beta distributions in \cite{Me13}.
The probabilistic approach based on Theorem \ref{general} first appeared in \cite{Me14}, where we gave a new, purely probabilistic proof of  Eq. \eqref{Decomposition} and also established Theorem \ref{scalinvar}. 

\begin{remark}
It is interesting to point out that the $L$ and $X_i,$ $i=1,2,3$ distributions
on the one hand and $Y$ on the other in the structure of $M_{(\tau,\lambda_1,\lambda_2)}$ are intrinsically different. This can be seen on three levels. First, the proof of Theorem \ref{BSM} indicates that $LX_1X_2X_3$
appears as one block from Theorem \ref{general}, while $Y$ is only needed to match the moments given by Selberg's formula. Second, Theorem \ref{scalinvar} shows that $L$ and $X_i,$ are involution invariant, whereas the proof of Corollary \ref{mycorollary} shows that
$Y$ is not. Finally, the law of the total mass of the Bacry-Muzy measure on the circle was conjectured in \cite{FyoBou} and verified in \cite{Remy} to be precisely
the same as $Y,$ while our conjecture for the law of the total mass of this measure on the unit interval is $LX_1X_2X_3$ times $Y$ so that it appears
that $Y$ comes from the circle and $LX_1X_2X_3$ is a superstructure that is needed to transform the law of the total mass from the circle to the unit interval.
\end{remark}


\section{Proofs of Probabilistic Results}\label{ProbProofs}

\begin{proof}[Proof of Theorem \ref{theoremcircle}]
The inverse Barnes beta distribution $\beta^{-1}_{2,2}(a, b)$ with parameters
$a=(1,\tau)$ and $b=(\tau, 1+\tau\lambda_1, 1+\tau\lambda_2)$ has the Mellin transform
\begin{align}
{\bf E}\bigl[\beta^{-q}_{2,2}(\tau, \tau, 1+\tau\lambda_1, 1+\tau\lambda_2)\bigr] = &\frac{\Gamma_2(-q+\tau\,|\,\tau)}{\Gamma_2(\tau\,|\,\tau)}
\frac{\Gamma_2(\tau(1+\lambda_1)+1\,|\,\tau)}{\Gamma_2(\tau(1+\lambda_1)+1-q\,|\,\tau)}\times \nonumber \\ & \times
\frac{\Gamma_2(\tau(1+\lambda_2)+1\,|\,\tau)}{\Gamma_2(\tau(1+\lambda_2)+1-q\,|\,\tau)}
\times \nonumber \\ & \times
\frac{\Gamma_2(\tau(\lambda_1+\lambda_2+1)+2-q\,|\,\tau)}{\Gamma_2(\tau(\lambda_1+\lambda_2+1)+2\,|\,\tau)}.
\end{align}
The difference between this expression and that in Eq. \eqref{thefunctioncopy} is in the last factor. Applying the functional
equation, we get
\begin{align}
\frac{\Gamma_2(\tau(\lambda_1+\lambda_2+1)+1-q\,|\,\tau)}{\Gamma_2(\tau(\lambda_1+\lambda_2+1)+1\,|\,\tau)}
=&\tau^{-\frac{q}{\tau}} 
\frac{\Gamma(\lambda_1+\lambda_2+1+\frac{1-q}{\tau})}{\Gamma(\lambda_1+\lambda_2+1+\frac{1}{\tau})}\times \nonumber \\ & \times
\frac{\Gamma_2(\tau(\lambda_1+\lambda_2+1)+2-q\,|\,\tau)}{\Gamma_2(\tau(\lambda_1+\lambda_2+1)+2\,|\,\tau)}.
\end{align}
Recalling the definition of the Mellin transform of the inverse Barnes beta  $\beta^{-1}_{1,0}(a, b)$ with parameters
$a=\tau$ and $b=1+\tau(1+\lambda_1+\lambda_2),$ 
\begin{equation}
{\bf E} \bigl[\beta^{-q}_{1,0}\bigl((\tau, 1+\tau(1+\lambda_1+\lambda_2)\bigr)\bigr] = 
 \tau^{-\frac{q}{\tau}} \frac{\Gamma(-\frac{q}{\tau}+1+\lambda_1+\lambda_2+\frac{1}{\tau})}{\Gamma(1+\lambda_1+\lambda_2+\frac{1}{\tau})},
\end{equation}
we see that the Mellin transform of the distribution $M_{(\tau, \lambda_1, \lambda_2)}$ in Eq. \eqref{thedecompcircle} coincides with the expression in Eq. \eqref{thefunctioncopy}.
The infinite divisibility of $\log M_{(\tau, \lambda_1, \lambda_2)}$ follows from that of $\log\beta^{-1}_{2,2}(a, b)$ and $\log\beta^{-1}_{1,0}(a, b).$
The determinacy of the Stieltjes moment problem for $M^{-1}_{(\tau, \lambda_1, \lambda_2)}$ follows from 
the fact $\beta_{2,2}(a, b)$ is compactly supported and so has a determinate moment problem and $\beta_{1,0}(a, b)$ is Fr\'echet and so
too has a determinate moment problem, cf. \cite{Char}, Sections 2.2 and 2.3. \qed
\end{proof}

\begin{proof}[Proof of Theorem \ref{general}.]
Recalling the Mellin transform of $\beta_{2,2}(a, b)$ and
definition of $X_1, X_2, X_3$ in Eqs. \eqref{X1}--\eqref{X3}, we
can write for ${\bf E}\bigl[X_1^q\bigr]{\bf E}\bigl[X_2^q\bigr]{\bf
E}\bigl[X_3^q\bigr]$ after some simplification
\begin{gather}
\frac{\Gamma_2(x_1-q\,|\,a)}{\Gamma_2(x_1\,|\,a)}
\frac{\Gamma_2(x_2-q\,|\,a)}{\Gamma_2(x_2\,|\,a)}
\frac{\Gamma_2(a_1+a_2-q\,|\,a)}{\Gamma_2(a_1+a_2\,|\,a)}
\frac{\Gamma_2(x_1+x_2-q\,|\,a)}{\Gamma_2(x_1+x_2\,|\,a)} \times
\nonumber \\
\times
\frac{\Gamma_2((x_1+x_2)/2\,|\,a)}{\Gamma_2((x_1+x_2)/2-q\,|\,a)}
\frac{\Gamma_2((x_1+x_2+a_1)/2\,|\,a)}{\Gamma_2((x_1+x_2+a_1)/2-q\,|\,a)}
\frac{\Gamma_2((x_1+x_2+a_2)/2\,|\,a)}{\Gamma_2((x_1+x_2+a_2)/2-q\,|\,a)}\times
\nonumber \\
\times
\frac{\Gamma_2((x_1+x_2+a_1+a_2)/2\,|\,a)}{\Gamma_2((x_1+x_2+a_1+a_2)/2-q\,|\,a)}
.
\end{gather}
Eq. \eqref{multiplic} gives us
\begin{gather}
\Gamma_2(x_1+x_2-2q\,|\,a) = 2^{-B_{2, 2}(x_1+x_2-2q\,|\,a)/2}
\Gamma_2((x_1+x_2)/2-q\,|\,a)
\times\nonumber \\ \times
\Gamma_2((x_1+x_2+a_1)/2-q\,|\,a)\times\Gamma_2((x_1+x_2+a_2)/2-q\,|\,a)
\times\nonumber \\ \times
\Gamma_2((x_1+x_2+a_1+a_2)/2-q\,|\,a).
\end{gather}
Hence, we can write for ${\bf E}\bigl[X_1^q\bigr]{\bf
E}\bigl[X_2^q\bigr]{\bf E}\bigl[X_3^q\bigr]$
\begin{align}
{\bf E}\bigl[X_1^q\bigr]{\bf
E}\bigl[X_2^q\bigr]{\bf E}\bigl[X_3^q\bigr] & =
\frac{2^{B_{2, 2}(x_1+x_2\,|\,a)/2}}{2^{B_{2,
2}(x_1+x_2-2q\,|\,a)/2}}
\frac{\Gamma_2(x_1-q\,|\,a)}{\Gamma_2(x_1\,|\,a)}
\frac{\Gamma_2(x_2-q\,|\,a)}{\Gamma_2(x_2\,|\,a)}\times \nonumber \\ & \times
\frac{\Gamma_2(a_1+a_2-q\,|\,a)}{\Gamma_2(a_1+a_2\,|\,a)}
\frac{\Gamma_2(x_1+x_2-q\,|\,a)}{\Gamma_2(x_1+x_2-2q\,|\,a)}.
\end{align}
Now, using the formula for $B_{2,2}(x\,|\,a)$ in Eq. \eqref{B22},
and the definition of $L$ in Eq. \eqref{Ldefin}, we can write
\begin{align}
{\bf E}\bigl[L^q\bigr]{\bf E}\bigl[X_1^q\bigr]{\bf
E}\bigl[X_2^q\bigr]{\bf E}\bigl[X_3^q\bigr]= & 2^{\bigl(2(x_1+x_2)-(a_1+a_2)\bigr)q/a_1a_2}\frac{\Gamma_2(x_1-q\,|\,a)}{\Gamma_2(x_1\,|\,a)}
\frac{\Gamma_2(x_2-q\,|\,a)}{\Gamma_2(x_2\,|\,a)}
\times\nonumber \\ & \times
\frac{\Gamma_2(a_1+a_2-q\,|\,a)}{\Gamma_2(a_1+a_2\,|\,a)}
\frac{\Gamma_2(x_1+x_2-q\,|\,a)}{\Gamma_2(x_1+x_2-2q\,|\,a)}.
\end{align}
This proves that the expression on the right-hand side of
Eq. \eqref{eqgeneral} is in fact the Mellin transform of the probability
distribution on $(0,\,\infty)$ that is given by the right-hand side of Eq. \eqref{generaldecomp}.
Its properties follow from the known properties of the normal and $\beta_{2,2}(a, b)$ distributions.\qed
\end{proof}

We now proceed to give a probabilistic proof of the existence of the Selberg integral probability distribution that is based on Theorem \ref{general}. 

\begin{proof}[Proof of Theorem \ref{BSM} and Corollary \ref{Uniqueness}]
Recall the definition of $a=(a_1, a_2)$ and $x_i(\tau,\lambda)$ in Eq. \eqref{xidef},
and let $L(\tau),$ $X_i(\tau,\lambda),$ and $Y(\tau)$ be as in the statement of Theorem \ref{BSM}.
By the functional equation of $\Gamma_2$ in Eq. \eqref{feq}
and the definition of $Y(\tau)$ in Eq. \eqref{Ydist}, we have
\begin{align}
{\bf E}[Y(\tau)^q] & =\Gamma(1-q/\tau), \label{YMellin} \\
\Gamma_{2}(\tau-q\,|\,\tau) & =
\frac{\tau^{(\tau-q)/\tau-1/2}}{\sqrt{2\pi}}{\bf E}[Y(\tau)^q]\,\Gamma_2\bigl(1+\tau-q\,|\,\tau\bigr). \label{Ytransform}
\end{align}
Then, given Theorem \ref{general}, the difference between Eqs. \eqref{thefunctioninterval} and \eqref{eqgeneral} is in the third double gamma ratio. Define
\begin{equation}
M_{(\tau, \lambda_1, \lambda_2)} \triangleq 2\pi\,
2^{-\bigl[3(1+\tau)+2\tau(\lambda_1+\lambda_2)\bigr]/\tau}\,
\frac{
L(\tau)\,X_1(\tau,\lambda)\,X_2(\tau,\lambda)\,X_3(\tau,\lambda)\,Y(\tau)}{\Gamma\bigl(1-1/\tau\bigr)}.
\end{equation}
It is now elementary to see that the Mellin transform of $M_{(\tau, \lambda_1, \lambda_2)}$ equals the expression in Eq. \eqref{thefunctioninterval}. The appearance of the $Y$ distribution in Eq. \eqref{Decomposition} is to account for this difference in the third gamma ratio.
The
remaining computation, which determines the overall constant in Eq. \eqref{Decomposition}, is straightforward. 
The proof of Corollary \ref{Uniqueness} follows from Eq. \eqref{Decomposition} due to
the determinacy of the Stieltjes moment problems for $\beta_{2,2}$ (compactly supported) and $Y^{-1}$ (Carleman's criterion)
and its indeterminacy for $L^{-1}$ (lognormal), cf. \cite{Char}, Sections 2.2 and 2.3.
\qed
\end{proof}

\begin{proof}[Proof of Theorem \ref{scalinvar}.]
We have by Eq. \eqref{Ldefin},
\begin{align}
{\bf E}\Bigl[L^{q/\tau} \bigl(\frac{1}{\tau}\bigr)\Bigr] & = {\bf E}\Bigl[e^{(q/\tau)\mathcal{N}(0,\,4\tau\log 2)}\Bigr], \nonumber \\
& = e^{4q^2\log 2/2\tau} \equiv {\bf E}\bigl[L^{q} (\tau)\bigr].
\end{align}
To prove Eq. \eqref{Xinvar}, observe the identity
\begin{equation}\label{xinvar}
x_i(\tau, \lambda_i)/\tau = x_i(1/\tau, \tau\lambda_i), \;i=1,2.
\end{equation}
Hence, by the definition of $X_i(\tau,\lambda)$ and
Theorem \ref{barnesbetascaling}, we have the identity
\begin{align}
X^{1/\tau}(1/\tau, \tau\lambda) & \overset{{\rm in \,law}}{=} \beta_{2,2}^{-1/\tau}\bigl(a/\tau, b/\tau\bigr), \nonumber \\
& \overset{{\rm in \,law}}{=} \beta_{2,2}^{-1}\bigl(a, b\bigr),
\end{align}
where $a=(1,\tau)$ and $b=(b_0, b_1, b_2)$ is given in terms of
$a$ and $x_i, \,i=1, 2$ in Eqs. \eqref{X1}--\eqref{X3}.
The proof of Eq. \eqref{Minvar} follows from Eqs. \eqref{Linvar} and \eqref{Xinvar}
(or can be seen directly from Eq. \eqref{eqgeneral} by Eqs. \eqref{scale} and \eqref{B22}). \qed
\end{proof}

\begin{proof}[Proof of Theorem \ref{Mtransforminvol} and Corollary \ref{mycorollary}.]
The $Y$ distribution in Eq. \eqref{Decomposition} is not involution invariant, confer Eq. \eqref{YMellin}. Instead,
we have by Eq. \eqref{YMellin} the identity
\begin{equation}
{\bf E}\bigl[Y^{q/\tau}(1/\tau)\bigr] = \frac{\Gamma(1-q)}{\Gamma(1-q/\tau)}
\,{\bf E}\bigl[Y^q(\tau)\bigr].  
\end{equation}
The proof now follows from Eq. \eqref{Minvar}. \qed
\end{proof}

\section{Critical Morris and Selberg Integral Distributions}\label{DerivM}
\noindent
In this section we will consider conjectured laws of the derivative martingales of the Bacry-Muzy GMC on the circle and interval.
These laws are obtained from the sub-critical laws that we studied in previous sections in the 
limit of $\tau\rightarrow 1,$ \emph{i.e.} in the limit of the so-called 
critical temperature and are of a particular interest in the theory of critical GMC chaos, confer \cite{barraletal} and \cite{dupluntieratal}.

\begin{definition}[Critical Morris Integral Distribution]
Let $M_{(\tau,\lambda_1,\lambda_2)}$ denote the Morris integral distribution as in Theorem \ref{theoremcircle}.
\begin{equation}
M_{(\tau=1,\lambda_1,\lambda_2)}\triangleq \lim\limits_{\tau\downarrow 1} \Gamma\bigl(1-1/\tau\bigr)\,M_{(\tau,\lambda_1,\lambda_2)}.
\end{equation} 
\end{definition}
This limit exists by Eq. \eqref{thedecompcircle}. We will refer to $M_{(\tau=1,\lambda_1,\lambda_2)}$ as the critical Morris integral distribution.

\begin{definition}[Critical Selberg Integral Distribution]
Let $M_{(\tau,\lambda_1,\lambda_2)}$ denote the Selberg integral distribution as in Theorem \ref{BSM}.
\begin{equation}\label{Selbergcrit}
M_{(\tau=1,\lambda_1,\lambda_2)}\triangleq \lim\limits_{\tau\downarrow 1} \Gamma\bigl(1-1/\tau\bigr)\,M_{(\tau,\lambda_1,\lambda_2)}.
\end{equation}
\end{definition}
This limit exists by Eq. \eqref{Decomposition}. We will refer to $M_{(\tau=1,\lambda_1,\lambda_2)}$ as the critical Selberg integral distribution.

Recall the definition of the Barnes $G(z)$ function in Eq. \eqref{Gdef}. Throughout this section we let $\Re(q)<1.$
\begin{theorem}[Critical Morris integral distribution]\label{criticalMorris}
\begin{gather}
{\bf E}\bigl[M_{(\tau=1,\lambda_1,\lambda_2)}^q\bigr] = 
\frac{G(2-q+\lambda_1)}{G(2+\lambda_1)}
\frac{G(2-q+\lambda_2)}{G(2+\lambda_2)} \frac{G(1)}{G(-q+1)}
\frac{G(2+\lambda_1+\lambda_2)}{G(2-q+\lambda_1+\lambda_2)}. \label{MGcrit}
\end{gather}
\begin{align}
\mathfrak{M}(q\,|\tau=1,\,\lambda_1,\,\lambda_2) = & \Gamma(1-q) \prod\limits_{m=1}^\infty \Bigl[
\frac{\Gamma(1-q+m)}{\Gamma(1+m)}  \frac{\Gamma(1+\lambda_1+m)}{\Gamma(1-q+\lambda_1+m)} 
\frac{\Gamma(1+\lambda_2+m)}{\Gamma(1-q+\lambda_2+m)} \nonumber \times \\  & \times \frac{\Gamma(1-q+\lambda_1+\lambda_2+m)}{\Gamma(1+\lambda_1+\lambda_2+m)}\Bigr], \\
= & \frac{\Gamma(\lambda_1+\lambda_2+2-q)}{\Gamma(\lambda_1+\lambda_2+2)} \times \nonumber \\
& \times 
\prod\limits_{k=0}^\infty \Bigl[
\frac{1+k}{1+k-q} \frac{2+\lambda_1+k-q}
{2+\lambda_1+k} 
 \frac{2+\lambda_2+k-q}{2+\lambda_2+k}
\frac{3+\lambda_1+\lambda_2+k}{3+\lambda_1+\lambda_2+k-q} \Bigr]^{k+1}
.\label{Morrisspecial}
\end{align}
The negative moments of the critical Morris integral distribution are
\begin{equation}
{\bf E}[M_{(\tau=1,\lambda_1,\lambda_2)}^{-n}]   =  \prod\limits_{j=0}^{n-1} \frac{\Gamma(2+\lambda_1+j) \,\Gamma(2+\lambda_2+j)}{\Gamma(2+\lambda_1+\lambda_2+j)\,\Gamma(1+j)}. \label{CirNegMomentscrit}
\end{equation}
The distribution is a product of two Barnes beta distributions,
\begin{align}
M_{(\tau=1, \lambda_1, \lambda_2)} = & \beta^{-1}_{22}(a=(1,1), b_0=1,\,b_1=1+ \lambda_1, \,b_2=1+\lambda_2) \times \nonumber \\ & \times
\beta_{1,0}^{-1}(a=1, b_0=\lambda_1+\lambda_2+2), \label{thedecompcirclecrit}
\end{align}
where $\beta^{-1}_{22}(a, b)$ is the inverse Barnes beta of type $(2,2)$ and $\beta^{-1}_{1,0}(a, b)$ is the independent inverse Barnes beta of type $(1,0).$ In particular, $\log M_{(\tau=1, \lambda_1, \lambda_2)}$ is infinitely divisible.
In the special case of $\lambda_1=\lambda_2=0,$ we have
\begin{equation}
M_{(\tau=1, \lambda_1, \lambda_2)} = \Gamma(1-q).
\end{equation}
\end{theorem}

We will now describe the critical Selberg integral distribution. 
\begin{theorem}[Critical Selberg integral distribution]\label{criticalSelberg}
\begin{gather}
{\bf E}\bigl[M_{(\tau=1,\lambda_1,\lambda_2)}^q\bigr] = 
\frac{G(2+\lambda_1)}{G(2-q+\lambda_1)}
\frac{G(2+\lambda_2)}{G(2-q+\lambda_2)} \frac{G(1)}{G(-q+1)}
\frac{G(4-2q+\lambda_1+\lambda_2)}{G(4-q+\lambda_1+\lambda_2}. \label{MGIcrit}
\end{gather}
\begin{gather}
{\bf E}[M_{(\tau=1,\lambda_1,\lambda_2)}^q] = \Gamma\bigl(1-q\bigr)
\frac{\Gamma\bigl(3-2q+\lambda_1+\lambda_2\bigr)}{
\Gamma\bigl(3-q+\lambda_1+\lambda_2\bigr)}
\prod\limits_{m=1}^\infty m^{2q}
\frac{\Gamma\bigl(1-q+m\bigr)}{\Gamma\bigl(1+m\bigr)}
 \times \nonumber \\
\times
\frac{\Gamma\bigl(1-q+\lambda_1+m\bigr)}{\Gamma\bigl(1+\lambda_1+m\bigr)}
\frac{\Gamma\bigl(1-q+\lambda_2+m\bigr)}{\Gamma\bigl(1+\lambda_2+m\bigr)}
\frac{\Gamma\bigl(2-q+\lambda_1+\lambda_2+m\bigr)}{\Gamma\bigl(2-2q+\lambda_1+\lambda_2+m\bigr)}.\label{InfiniteSelbergcrit}
\end{gather}
\begin{equation}
{\bf E}[M_{(\tau=1,\lambda_1,\lambda_2)}^{-n}]   =  \prod_{k=0}^{n-1}
\frac{\Gamma\bigl(4+\lambda_1+\lambda_2+n+k\bigr)
 }{
\Gamma\bigl(2+\lambda_1+k\bigr)\Gamma\bigl(2+\lambda_2+k\bigr)
\Gamma\bigl(1+k\bigr) }. \label{IntNegMomentscrit}
\end{equation}
Define the distributions
\begin{align}
L \triangleq &\exp\bigl(\mathcal{N}(0,\,4\log 2)\bigr), \\ Y
\triangleq &\,y^{-2}\exp\bigl(-y^{-1}\bigr)\,dy,\; y>0,\label{Ydistcrit}
\end{align}
\emph{i.e.} $\log L$ is a zero-mean normal with variance $4\log
2$ and $Y$ is Fr\'echet. Let $X_1,\,X_2,\,X_3$ have the $\beta^{-1}_{2,
2}(a=(1,1), b)$ distribution with the parameters 
\begin{align}
X_1 &\triangleq \beta_{2,2}^{-1}\Bigl(a=(1,1),
b_0=2+\lambda_1,\,b_1=(\lambda_2-\lambda_1)/2, \,
b_2=(\lambda_2-\lambda_1)/2\Bigr),\\
X_2 & \triangleq \beta_{2,2}^{-1}\Bigl(a=(1,1),
b_0=2+(\lambda_1+\lambda_2)/2,\,b_1=1/2,\,b_2=1/2\Bigr),\\
X_3 & \triangleq \beta_{2,2}^{-1}\Bigl(a=(1,1), b_0=2,\,
b_1=1+(\lambda_1+\lambda_2)/2, \,
b_2=1+(\lambda_1+\lambda_2)/2\Bigr).
\end{align}
Then,
\begin{equation}\label{Decompositioncrit}
M_{(\tau=1, \lambda_1, \lambda_2)} \overset{{\rm in \,law}}{=} 2\pi\,
2^{-\bigl[6+2(\lambda_1+\lambda_2)\bigr]}\,
L\,X_1\,X_2\,X_3\,Y.
\end{equation}
In particular, $M_{(\tau=1,\lambda_1,\lambda_2)}$ is infinitely divisible and absolutely continuous.
\end{theorem}

Unlike the critical Morris integral distribution, which becomes Fr\'echet in the special case of $\lambda_1=\lambda_2=0,$ 
the critical Selberg integral distribution remains non-trivial. 
\begin{theorem}\label{Deriv}
Let $\lambda_1=\lambda_2=0.$ Then, the critical Selberg integral distribution satisfies 
\begin{equation}\label{Mderiv}
{\bf E}[M_{(\tau=1,0,0)}^q] = 
\frac{G(4-2q)}{G(1-q)\, G^2(2-q)\,G(4-q)}
\end{equation}
for $\Re(q)<1.$ The Mellin transform satisfies the infinite product representation
\begin{equation}
{\bf E}\bigl[M_{(\tau=1,0,0)}^q\bigr] =
\frac{\Gamma\bigl(1-q\bigr)\Gamma\bigl(3-2q\bigr)}{
\Gamma\bigl(3-q\bigr)}
\prod\limits_{m=1}^\infty m^{2q}
\frac{\Gamma^3\bigl(1-q+m\bigr)}{\Gamma^3\bigl(1+m\bigr)}
\frac{\Gamma\bigl(2-q+m\bigr)}{\Gamma\bigl(2-2q+m\bigr)}.
\end{equation}
The negative moments for $l\in\mathbb{N}$ are
\begin{equation}
{\bf E}\bigl[M_{(\tau=1,0,0)}^{-l}\bigr] =
\prod_{k=0}^{l-1}
\frac{(3+l+k)!}{(k+1)!^2 \,k!}.
\end{equation}
$M_{(\tau=1,0,0)}$ has the factorization
\begin{equation}\label{McDecomposition}
M_{(\tau=1,0,0)} = \frac{\pi}{32}\,L\,\,X_2\,X_3\,Y,
\end{equation}
where
\begin{align}
L & = \exp\bigl(\mathcal{N}(0, 4\log 2)\bigr) \;(\text{Lognormal}), \\
X_2 & = \beta^{-1}_{2,2}\bigl(a=(1,1), \,b_0=2, \,b_1=b_2=1/2\bigr), \\
X_3 & = \frac{2}{y^3}\,dy, \;y>1 \;(\text{Pareto}),\\
Y & = \frac{1}{y^2}\,e^{-1/y}\,dy,\;y>0 \;(\text{Fr\'echet}).
\end{align}
\end{theorem}
This result first appeared in \cite{Me14}.

Theorems \ref{criticalMorris}--\ref{Deriv} are straightforward corollaries of the corresponding results for the Morris and Selberg 
integral distributions in Sections \ref{CirAnalytical}, \ref{IntAnalytical}, and \ref{Probabilistic}.
We note that Eq. \eqref{Morrisspecial} follows from Eq. \eqref{thedecompcirclecrit} by means of Eq. \eqref{specialfactorization}. \qed

\begin{remark}
It is clear from Theorem \ref{Deriv} that the most nontrivial component of
the critical Selberg distribution is $X_2.$ Consider more generally a family of $\beta_{2,2}$ distributions
that is parameterized by $\delta>0.$
\begin{equation}
\beta_{2,2}(\delta) \triangleq \beta_{2,2}\bigl(a=(1,1), \,b_0=\delta, \,b_1=b_2=1/2\bigr).
\end{equation}
Its Mellin transform satisfies, 
\begin{align}
{\bf E}\bigl[\beta_{2,2}^q(\delta)\bigr] & = \frac{G(\delta)}{G(q+\delta)}
\frac{G^2(q+\delta+1/2)}{G^2(\delta+1/2)}
\frac{G(\delta+1)}{G(q+\delta+1)}, \\
& = \prod\limits_{k=0}^\infty \Bigl[
\frac{\delta+k}{q+\delta+k} \frac{(q+\delta+1/2+k)^2}
{(\delta+1/2+k)^2} \frac{\delta+1+k}
{q+\delta+1+k}\Bigr]^{k+1}
\end{align}
by Corollary \ref{BarnesFactorSpecial}.
Remarkably, this distribution is intrinsically related to the Riemann xi function, cf. Section 6 in \cite{Me14}, which suggests that 
the critical Selberg integral distribution is also related to the xi function.
\end{remark}

\section{Analytic Continuation of the Complex Selberg Integral}\label{AnalyticalComplexSelberg}
\noindent The complex Selberg Integral, also known as the Dotsenko-Fateev integral, was computed independently by Aomoto \cite{Aom} and 
Dotsenko and Fateev \cite{DF}.
It is a two-dimensional generalization of the classical Selberg integral. Unlike the Morris and Selberg integrals, it does not correspond
to the moments of total mass of a GMC measure and so its analytic continuation is not the Mellin transform of a probability distribution. It can however be
naturally interpreted as a rescaled limit of the moments of total mass of the two-dimensional GMC measure with the kernel
$-\log|\vec{r}_1-\vec{r}_2|$ in the limit of the
domain, over which it is defined, going to infinity, cf.  \cite{cao17} for details. Then, the analytic continuation of the
complex Selberg integral becomes the mod-Gaussian limit of the REM (random energy model) that is associated with the
GMC measure. This observation explains the interest in the complex Selberg integral, for its analytic continuation provides a 
glimpse at the behavior of the total mass in two-dimensions. Our goal in this section is to write down the analytic continuation and 
prove its involution invariance.  We note that the physicists' approach to this problem can be found in \cite{cao17},
which leads to interesting conjectures about the maximum of the underlying gaussian field. 

Let $\vec{u}$ be an arbitrary unit vector. Following \cite{ForresterWarnaar}, define the complex Selberg integral,
\begin{gather}
\int\limits_{(\mathbb{R}^2)^n} \prod_{i=1}^n |\vec{r}_i|^{2\lambda_1} |\vec{u}-\vec{r}_i|^{2\lambda_2}\, \prod\limits_{i<j}^n |\vec{r}_i-\vec{r}_j|^{-4/\tau} d\vec{r}_1\cdots d\vec{r}_n , \nonumber \\
= \frac{1}{n!} \Bigl[\prod_{k=0}^{n-1}\frac{\Gamma(1-(k+1)/\tau)
\Gamma(1+\lambda_1-k/\tau)\Gamma(1+\lambda_2-k/\tau)}
{\Gamma(1-1/\tau)\Gamma(2+\lambda_1+\lambda_2-(n+k-1)/\tau)}\Bigr]^2\times \nonumber \\
\times \prod_{k=0}^{n-1} \frac{\sin\pi((k+1)/\tau)
\sin\pi(1+\lambda_1-k/\tau)\sin\pi(1+\lambda_2-k/\tau)}
{\sin(\pi/\tau)\sin\pi(2+\lambda_1+\lambda_2-(n+k-1)/\tau)}.\label{ComplexSelberg}
\end{gather}

\begin{theorem}[Analytic continuation of the complex Selberg integral]\label{analcomplexS}
Let $\lambda_1,$ $\lambda_2,$ and $\tau$ satisfy 
\begin{gather}
\tau>1,\; \lambda_1,\lambda_2<0, \label{bound1} \\
1+\tau(1+\lambda_i) >0, \, i=1,2,\label{bound2}\\
1+\tau(\lambda_1+\lambda_2)< 0. \label{bound3}
\end{gather}
Let $q$ belong to the strip
\begin{equation}\label{qcomplexstrip}
\max\Big\{\frac{1}{2}\bigl(1+\tau(1+\lambda_1+\lambda_2)\bigr),\,      1+\tau(1+\lambda_1+\lambda_2)\Big\}<\Re(q)<\min\Big\{\tau, 1+\tau(1+\lambda_1), 1+\tau(1+\lambda_2)\Big\}.
\end{equation}
Let $\mathfrak{M}(q\,|\,\tau,\lambda_1,\lambda_2)$ denote the analytic continuation of the Selberg integral in Eq. \eqref{thefunctioninterval}
and $S_2(z\,|\,\tau)$ denote the double sine function as in Eq. \eqref{msinedef} ($M=2,\; a=(1,\tau)).$
Then, the function
\begin{align}
\mathfrak{CM}(q\,|\,\tau,\lambda_1,\lambda_2)=& \frac{ \mathfrak{M}(q\,|\,\tau,\lambda_1,\lambda_2)^2}{\Gamma(1+q) \, \bigl(4 \sin(\pi/\tau)\bigr)^q}  \, 
\frac{S_2(1-q+\tau(1+\lambda_1)\,|\,\tau)}{S_2(1+\tau(1+\lambda_1)\,|\,\tau)}
\frac{S_2(1-q+\tau(1+\lambda_2)\,|\,\tau)}{S_2(1+\tau(1+\lambda_2)\,|\,\tau)}\times
\nonumber \\  \times
& \frac{S_2(\tau-q\,|\,\tau)}{S_2(\tau\,|\,\tau)}
\frac{S_2(2-q+\tau(2+\lambda_1+\lambda_2)\,|\,\tau)}{S_2(2-2q+\tau(2+\lambda_1+\lambda_2)\,|\,\tau)} \label{complexS1}
\end{align}
recovers the expression in Eq. \eqref{ComplexSelberg} when $q=n\in\mathbb{N}$ satisfies Eq. \eqref{qcomplexstrip}. The function
$\mathfrak{CM}(q\,|\,\tau,\lambda_1,\lambda_2)$ satisfies
\begin{align}
\mathfrak{CM}(q\,|\,\tau,\lambda_1,\lambda_2) = & 
\Bigl(\frac{\Gamma(1/\tau)\,\pi\,\tau^{\frac{1}{\tau}}}{\Gamma\bigl(1-1/\tau\bigr)}\Bigr)^q\, \Gamma\bigl(1-\frac{q}{\tau}\bigr)\,
\frac{\Gamma_2(1-q+\tau(1+\lambda_1)\,|\,\tau)}{\Gamma_2(1+\tau(1+\lambda_1)\,|\,\tau)}
\frac{\Gamma_2(1-q+\tau(1+\lambda_2)\,|\,\tau)}{\Gamma_2(1+\tau(1+\lambda_2)\,|\,\tau)}\times
\nonumber \\ & \times
\frac{\Gamma_2(1-q+\tau\,|\,\tau)}{\Gamma_2(1+\tau\,|\,\tau)}
\frac{\Gamma_2(2-q+\tau(2+\lambda_1+\lambda_2)\,|\,\tau)}{\Gamma_2(2-2q+\tau(2+\lambda_1+\lambda_2)\,|\,\tau)}
\frac{\Gamma_2(q+1+\tau\,|\,\tau)}{\Gamma_2(1+\tau\,|\,\tau)}
\times
\nonumber \\ & \times
\frac{\Gamma_2(q-\tau\lambda_1\,|\,\tau)}{\Gamma_2(-\tau\lambda_1\,|\,\tau)}
\frac{\Gamma_2(q-\tau\lambda_2\,|\,\tau)}{\Gamma_2(-\tau\lambda_2\,|\,\tau)}
\frac{\Gamma_2(q-1-\tau(1+\lambda_1+\lambda_2)\,|\,\tau)}{\Gamma_2(2q-1-\tau(1+\lambda_1+\lambda_2)\,|\,\tau)}.\label{complexS2}
\end{align}
\end{theorem}
\begin{corollary}\label{MComplextransforminvol}
The function
$\mathfrak{CM}(q\,|\,\tau,\lambda_1,\lambda_2)$ satisfies the following involution invariance under
\begin{equation}\label{involutiondef}
\tau\rightarrow \frac{1}{\tau},\; q\rightarrow \frac{q}{\tau}, \; \lambda_i\rightarrow \tau\lambda_i.
\end{equation}
\begin{align}
\mathfrak{CM}\bigl(\frac{q}{\tau}\,|\,\frac{1}{\tau},\tau\lambda_1,\tau\lambda_2\bigr) \,
\Bigl(\frac{\Gamma(1-\tau)}{\pi\,\Gamma(\tau)}\Bigr)^{\frac{q}{\tau}} \Gamma(1-\frac{q}{\tau}) = &
\mathfrak{CM}(q\,|\,\tau,\lambda_1,\lambda_2)  \Bigl(\frac{\Gamma(1-\frac{1}{\tau})}{\pi\,\Gamma(1/\tau)}\Bigr)^{q}
 \Gamma(1-q).\label{involutionintcomplex}
\end{align}
\end{corollary}
\begin{corollary}
In the critical case $\tau \downarrow 1,$ the function $\mathfrak{CM}(q\,|\, \tau,\lambda_1,\lambda_2)$ has the limit,
\begin{align}
\lim\limits_{\tau\downarrow 1} \Bigl[\Gamma(1-1/\tau)^q\,\mathfrak{CM}(q\,|\,\tau,\lambda_1,\lambda_2)\Bigr] = & 
\pi^q\, \Gamma\bigl(1-q\bigr)\,
\frac{G(2+\lambda_1)}{G(2-q+\lambda_1)}
\frac{G(2+\lambda_2)}{G(2-q+\lambda_2)}\times
\nonumber \\ & \times
\frac{G(2)}{G(2-q)}
\frac{G(4-2q+\lambda_1+\lambda_2)}{G(4-q+\lambda_1+\lambda_2)}
\frac{G(2)}{G(q+2)}
\times
\nonumber \\ & \times
\frac{G(-\lambda_1)}{G(q-\lambda_1)}
\frac{G(-\lambda_2)}{G(q-\lambda_2)}
\frac{G(2q-2-(\lambda_1+\lambda_2))}{G(q-2-(\lambda_1+\lambda_2))}.\label{critcomplexS}
\end{align}
\end{corollary}
\begin{remark}
While the analytic continuation of the complex Selberg integral is not the Mellin transform of a random variable,
it is nonetheless possible to decompose it into the product of the Mellin transform of a random variable and an extra factor.
Let $M_{(a, x)}=L\,X_1\,X_2\,X_3$ denote the random variable that is described in Theorem \ref{general} and $Y$ be the Fr\'echet factor
as in Eq. \eqref{Ydist}. Let
\begin{align}
M_1 = & M_{(a, x)} \; \text{with}\; a=(1,\tau), \; x=\bigl(1+\tau(1+\lambda_1), \, 1+\tau(1+\lambda_2)\bigr), \\
M_2 = & M_{(a, x)} \; \text{with}\; a=(1,\tau), \; x=(-\tau\lambda_1, \,-\tau\lambda_2).
\end{align}
Then, up to a constant $C,$ Eq. \eqref{complexS2} 
is equivalent to the identity
\begin{align}
\mathfrak{CM}(q\,|\,\tau,\lambda_1,\lambda_2) = & e^{Cq}  {\bf E}[Y^q]\,
{\bf E}[M_1^q]\,{\bf E}[M_2^{-q}]\,\frac{\Gamma_2(q-1-\tau(1+\lambda_1+\lambda_2)\,|\,\tau)}{\Gamma_2(2q-1-\tau(1+\lambda_1+\lambda_2)\,|\,\tau)}\,\frac{\Gamma_2(2q-\tau(\lambda_1+\lambda_2)\,|\,\tau)}{\Gamma_2(q-\tau(\lambda_1+\lambda_2)\,|\,\tau)}, \nonumber \\ 
= &  e^{Cq}  {\bf E}[Y^q]\,
{\bf E}[M_1^q]\,{\bf E}[M_2^{-q}]\,\frac{\Gamma\bigl(\frac{q-1-\tau(1+\lambda_1+\lambda_2)}{\tau}\bigr)}{\Gamma\bigl(\frac{2q-1-\tau(1+\lambda_1+\lambda_2)}{\tau}\bigr)}\,\frac{\Gamma(q-\tau(1+\lambda_1+\lambda_2))}{\Gamma(2q-\tau(1+\lambda_1+\lambda_2))}.
\end{align}
Thus, the random variable is the product of an independent Selberg integral distribution $Y\,M_1,$ cf. Theorem \ref{BSM}, and $M_2^{-1}.$
\end{remark}

\begin{proof}[Proof of Theorem \ref{analcomplexS}.]
It is elementary to see that the bounds in Eqs. \eqref{bound1}--\eqref{bound3} guarantee that the strip in Eq. \eqref{qcomplexstrip} is non-empty. Now, we observe that the structure of the product of sines in Eq. \eqref{ComplexSelberg} is essentially the same as the product
of gamma factors in the Selberg integral. Moreover, the functional equation of the double sine function in Eq. \eqref{feqsine} is the same
as that of the double gamma function in Eq. \eqref{feq}, except $\Gamma_1$ is replaced with $S_1.$ It follows that
the analytic continuation in Eq. \eqref{complexS1} follows from Eq. \eqref{Srepeated} in the same way as Eq. \eqref{thefunctioninterval}
followed from Eq. \eqref{repeated}. It remains to verify that the expression is Eq. \eqref{complexS1} is well-defined under the
conditions in Eq. \eqref{qcomplexstrip}, \emph{i.e.} that the arguments of $\Gamma_2$ factors satisfy $\Re(z)>0$ and those of
$S_2$ factors satisfy $0<\Re(z)<1+\tau.$ For $\Re(z)>0$ this is true by the upper bound in Eq. \eqref{qcomplexstrip}.
To satisfy $\Re(z)<1+\tau, $ in addition to the lower bound in Eq. \eqref{qcomplexstrip} we also need $\Re(q)>\tau\lambda_i$ and $\Re(q)>-(1+\tau).$\footnote{$\Re(q)=-1$ is a removable singularity as both $S_2(\tau-q\,|\,\tau)$ and $\Gamma(1+q)$ have simple poles there, see Eq. \eqref{simplepole}.} However, these are automatically satisfied due to $\max\{\tau\lambda_i, -(1+\tau)\} < 1+\tau(1+\lambda_1
+\lambda_2)$ by Eq. \eqref{bound2}.

The proof of Eq. \eqref{complexS2} follows from the definition of the double sine function
in Eq. \eqref{msinedef} and the identity in Eq. \eqref{mydoublegammaidentity}. One obtains,
\begin{align}
\frac{1}{\Gamma(1+q)} \frac{S_2(\tau-q\,|\,\tau)}{S_2(\tau\,|\,\tau)} = & \frac{\Gamma_2(1+q+\tau\,|\,\tau)}{\Gamma_2(1-q+\tau\,|\,\tau)} \,
\frac{\tau^{q/\tau}}{\Gamma(1-q/\tau)}, \label{simplepole} \\
\tau^{q/\tau} \frac{\Gamma_2(\tau-q\,|\,\tau)}{\Gamma_2(\tau\,|\,\tau)} =& \Gamma(1-\frac{q}{\tau})\,\frac{\Gamma_2(1-q+\tau\,|\,\tau)}{\Gamma_2(1+\tau\,|\,\tau)}.
\end{align}
Finally, we need the identity
\begin{equation}
\Gamma(1-\frac{1}{\tau})\sin\pi/\tau = \frac{\pi}{\Gamma(1/\tau)}.
\end{equation}
The result now follows by substituting these identities into Eq. \eqref{complexS1} and using the definition of the double sine function.  \qed
\end{proof}
\begin{proof}[Proof of Corollary \ref{MComplextransforminvol}.]
The proof is very similar to that of Theorem \ref{CicInvolution}. One starts with Eq. \eqref{complexS2} and observes that
all the double gamma 
terms are invariant under the involution in Eq. \eqref{involutiondef} due to 
the scaling property of the double gamma function in Eq. \eqref{scale}. 
In our case $\kappa=1/\tau.$ It remains to collect the power of $\kappa$ that comes from the pre-factor
in Eq. \eqref{scale}. A direct calculation shows that it is
\begin{equation}
\bigl(\frac{1}{\tau}\bigr)^{-(q+\frac{q}{\tau})}.
\end{equation}
The result follows. \qed
\end{proof}

\section{Applications}\label{SomeApplications}
\noindent In this section we will consider three conjectured applications of our results. 

\subsection{Maximum Distribution of the Gaussian Free Field on the Interval and Circle}
In this section we will formulate precise conjectures about the distribution of the maximum of the discrete 2D Gaussian
Free Field (GFF) with a non-random logarithmic potential restricted to the unit interval and circle.  We will not attempt to review the GFF here but rather refer the reader to \cite{FyoBou} for the circle case and to \cite{FLDR} and Section 3 of \cite{Menon} for the interval case. Suffice it to say that our approach to the 
GFF construction is essentially based on the construction of Bacry-Muzy, cf. \cite{MRW}, \cite{BM1}, \cite{BM}. Our results in this section first appeared in \cite{Me16}.

Let 
\begin{equation}
N=1/\varepsilon.
\end{equation}
Let the gaussian field $V_\varepsilon(x)$ be as in Eq. \eqref{covk}.
The limit $\lim\limits_{\varepsilon\rightarrow 0} V_{\varepsilon}(x)$ is
called the continuous GFF on the interval $x\in[0, \,1]$ and its discretized version $V_\varepsilon(x_i),$  
$x_i=i\varepsilon,$ $i=0\cdots N,$ is the discrete GFF on the interval. 
We note that in applications, see \cite{Menon} and subsection \ref{modG} below, for example, the GFF construction arises is a slightly more general form of
\begin{align}
{\bf{Cov}}\left[V_{\varepsilon}(u), \,V_{\varepsilon}
(v)\right]  = & 
\begin{cases}\label{covkk}
 -
2 \, \log|u-v|, \, \varepsilon\ll|u-v|\leq 1,  \\
 2
\left(\kappa-\log\varepsilon\right),\, u=v,
\end{cases} \nonumber \\
& + O(\varepsilon),
\end{align}
where $\kappa\geq0$ is some fixed constant and the details of regularization are relegated to the $O(\varepsilon)$ term.
It is worth emphasizing that the choice of covariance regularization for $|u-v|\leq \varepsilon$ has no effect on
distribution of the maximum, so long as the variance behaves as in Eq. \eqref{covk},  due to Theorem 6 in \cite{BM1} as explained below.
The same remark applies to the GFF on the circle. 
Let the gaussian field $V_\varepsilon(\psi)$ be as in Eq. \eqref{covkc} or more generally satisfy
\begin{align}
{\bf{Cov}}\left[V_{\varepsilon}(\psi), \,V_{\varepsilon}
(\xi)\right]  = &
\begin{cases}\label{covkcir}
 -
2 \, \log|e^{2\pi i\psi}-e^{2\pi i\xi}|, \,  |\xi-\psi|\gg \varepsilon,  \\
 2
\left(\kappa-\log\varepsilon\right), \psi=\xi,
\end{cases} \nonumber \\
& + O(\varepsilon),
\end{align}
where $\kappa\geq 0$ is some fixed constant. The limit $\lim\limits_{\varepsilon\rightarrow 0} V_{\varepsilon}(\psi)$ is
called the GFF on the circle $\psi\in[-\frac{1}{2}, \frac{1}{2})$ and its discretized version $V_\varepsilon(\psi_j),$  
$\psi_j=j\varepsilon,$ $j=-N/2\cdots N/2,$ is the discrete GFF on the circle. 
The existence of such objects follows from the general theory of \cite{RajRos} as shown in \cite{BM1} and \cite{BM}. 
\begin{definition}[Problem formulation for the interval]
Let $\lambda_1,\,\lambda_2\geq 0.$  
What is the distribution of 
\begin{equation}
V_N\triangleq\max\Big\{V_{\varepsilon}(x_i)+\lambda_1\,\log(x_i)+\lambda_2\,\log(1-x_i),\,i=1\cdots N\Big\}
\end{equation}
in the form of an asymptotic expansion in $N$ in the limit $N\rightarrow \infty?$
\end{definition}
\begin{definition}[Problem formulation for the circle]
Let $\lambda\geq 0.$ 
What is the distribution of 
\begin{equation}
V_N\triangleq\max\Big\{V_{\varepsilon}(\psi_j)+2\lambda\log|1+e^{2\pi i\psi_j}|,\,j=-N/2\cdots N/2\Big\}
\end{equation}
in the form of an asymptotic expansion in $N$ in the limit $N\rightarrow \infty?$
\end{definition}

We will consider the case of the GFF on the interval first. Recall the critical Selberg integral probability 
distribution $M_{(\tau=1,\lambda_1,\lambda_2)}$ that we defined in Section \ref{DerivM}. 
\begin{conjecture}[Maximum of the GFF on the Interval]\label{maxint}
The leading asymptotic term in the expansion of the Laplace transform of $V_N$ in $N$ is 
\begin{equation}
{\bf E}[e^{q\,V_N}] \approx e^{q(2\log N-(3/2)\log\log N+{\rm const})}\,{\bf E}\bigl[M^q_{(\tau=1,\lambda_1,\lambda_2)}
\bigr],\;N\rightarrow\infty.
\end{equation}
Probabilistically, let $X_i,$ $i=1,2,3$ and $Y$ be as in Theorem \ref{criticalSelberg}.
Then, as $N\rightarrow \infty,$
\begin{align}
V_N = & 2\log N-\frac{3}{2}\log\log N+{\rm const}+\mathcal{N}(0,\,4\log 2) + \log X_1+\log X_2+\log X_3+\nonumber \\ & +
\log Y+\log Y'+o(1),
\end{align}
where $Y'$ is an independent copy of $Y.$ 
\end{conjecture}
This conjecture at the level of the Mellin transform is due to \cite{FLDR} in the case of $\lambda_1=\lambda_2=0$ and \cite{FLD} for
general $\lambda_1$ and $\lambda_2.$ Our probabilistic re-formulation of their conjecture was first given in \cite{Me16}.

Similarly, to formulate our conjecture for the maximum of the GFF on the circle, we need to recall the critical Morris integral
probability distribution $M_{(\tau=1,\lambda_1,\lambda_2)}$ that we considered in Theorem \ref{criticalMorris}. 
\begin{conjecture}[Maximum of the GFF on the Circle]\label{maxintcircle}
The leading asymptotic term in the expansion of the Laplace transform of $V_N$ in $N$ is 
\begin{equation}
{\bf E}[e^{q\,V_N}] \approx e^{q(2\log N-(3/2)\log\log N+{\rm const})}\,{\bf E}\bigl[M^q_{(\tau=1,\lambda,\lambda)}
\bigr],\;N\rightarrow\infty.
\end{equation}
Probabilistically, let $X \triangleq \beta_{2,2}^{-1}\bigl(\tau=1, b_0=1,\,b_1=1+\lambda,\,b_2=1+\lambda\bigr)$ and
$Y \triangleq \beta_{1,0}^{-1}(\tau=1, b_0=2\lambda+2)$  be as in Theorem \ref{criticalMorris} 
and $Y' \triangleq \beta_{1,0}^{-1}\bigl(\tau=1, b_0=1\bigr).$
Then,
\begin{equation}
V_N = 2\log N-\frac{3}{2}\log\log N+{\rm const}+\log X+ \log Y+\log Y'+o(1).
\end{equation}
If $\lambda=0,$ 
\begin{equation}
V_N = 2\log N-\frac{3}{2}\log\log N+{\rm const}+ \log Y+\log Y'+o(1),
\end{equation}
where 
\begin{equation}
Y\overset{{\rm in \,law}}{=}Y'=\beta_{1,0}^{-1}\bigl(\tau=1, b_0=1\bigr).
\end{equation}
\end{conjecture}
This conjecture is due to \cite{FyoBou} in the case of $\lambda=0.$ The extension to general $\lambda$ was first given in \cite{Me16}.

In the rest of this section we will give a heuristic derivation of our conjectures. 
Our approach is based on the freezing hypothesis and calculations of \cite{FyoBou} and \cite{FLDR} as well as our Conjectures
\ref{ourmainconjcircle} and \ref{ourmainconjinterval} and the involution invariance of the Morris and Selberg integral distributions.

Let $0\leq \beta<1$ and  $\tau=1/\beta^2>1.$ 
Following \cite{FLDR}, define the exponential functional
\begin{equation}\label{Zdef}
Z_{\lambda_1,\lambda_2,\varepsilon}(\beta) \triangleq \sum\limits_{i=1}^N  x_i^{\beta\lambda_1}(1-x_i)^{\beta\lambda_2} e^{\beta V_\varepsilon(x_i)}.
\end{equation}
Using the identity
\begin{equation}\label{keymax}
{\bf P} \bigl(V_N < s \bigr) = \lim\limits_{\beta\rightarrow \infty} {\bf E}\Bigl[
\exp\Bigl(-e^{-\beta s}\,Z_{\lambda_1,\lambda_2,\varepsilon}(\beta)/C\Bigr)\Bigr],
\end{equation}
which is applicable to any sequence of random variables and an arbitrary \emph{$\beta$-independent} constant $C,$
the distribution of the maximum is reduced to the Laplace transform of the exponential functional in the limit $\beta\rightarrow \infty.$
Now, by Eq. \eqref{chaosinterval}, it is known that for $0<\beta<1$ the exponential functional converges\footnote{
Theorem 6 in \cite{BM1} shows that the laws of the total mass of the continuous and discrete multiplicative chaos measures are the same
provided  the $\varepsilon$ parameter coincides with the discretization step.}  as 
$N\rightarrow \infty$ to the total mass of the Bacry-Muzy measure on the unit interval with a logarithmic potential,
which is conjectured to be given by the Selberg integral distribution, resulting 
in the approximation\footnote{The validity of approximating the finite $N$ quantity with the $N\rightarrow\infty$
limit is discussed in \cite{FyoBou}.}
\begin{equation}\label{keyapprox}
Z_{\lambda_1,\lambda_2,\varepsilon}(\beta) \approx N^{1+\beta^2}\, e^{\beta^2\kappa}\,M_{(\tau,\beta\lambda_1,\beta\lambda_2)},\;N\rightarrow\infty.
\end{equation}
Next, we recall the involution invariance of the Mellin transform of the Selberg integral distribution, see Eq. \eqref{involutionint}.
Denoting the Mellin transform as in Theorem \ref{BSM} by $\mathfrak{M}(q\,|\,\tau,\lambda_1,\lambda_2)$
and introducing the function
\begin{equation}\label{Func}
F(q\,|\,\beta, \lambda_1,\lambda_2) \triangleq \mathfrak{M}\bigl(\frac{q}{\beta}\,|\,\frac{1}{\beta^2},\beta\lambda_1,\beta\lambda_2\bigr) (2\pi)^{-\frac{q}{\beta}}\,\Gamma^{\frac{q}{\beta}}(1-\beta^2) \Gamma(1-\frac{q}{\beta}),
\end{equation}
one observes that by the involution invariance in Eq. \eqref{involutionint} this function satisfies the identity
\begin{equation}\label{selfdual}
F(q\,|\,\beta, \lambda_1,\lambda_2) = F(q\,\big|\,\frac{1}{\beta}, \lambda_1,\lambda_2),
\end{equation}
first discovered in \cite{FLDR} in the special case of $\lambda_1=\lambda_2=0,$ then formulated in general in the form of Eq. \eqref{involutionint} in
\cite{Me14}, and in \cite{FLD} in this form. 
On the other hand, as shown in \cite{FLDR}, one has the general identity
\begin{equation}\label{generalidentity}
\int_\mathbb{R} e^{yq} \frac{d}{dy}\exp\Bigl(-e^{-\beta y}X\Bigr)  dy = X^{\frac{q}{\beta}} \,\Gamma(1-\frac{q}{\beta}),\;\Re(q)<0, \,X>0.
\end{equation}
Letting 
\begin{equation}\label{X}
X = Z_{\lambda_1,\lambda_2,\varepsilon}(\beta) \frac{e^{\kappa}\,\Gamma(1-\beta^2)}{2\pi}
\end{equation}
one sees by means of Eq. \eqref{keyapprox} that for $0<\beta<1,$
\begin{equation}
{\bf E}\bigl[X^{\frac{q}{\beta}}\bigr]\approx e^{q\kappa(\beta+\frac{1}{\beta})} N^{q(\beta+\frac{1}{\beta})}\,\mathfrak{M}\bigl(\frac{q}{\beta}\,|\,\frac{1}{\beta^2},\beta\lambda_1,\beta\lambda_2\bigr) (2\pi)^{-\frac{q}{\beta}}\,\Gamma^{\frac{q}{\beta}}(1-\beta^2),\; N\rightarrow\infty,
\end{equation}
so that
\begin{equation}\label{rhs}
{\bf E}\bigl[X^{\frac{q}{\beta}}\bigr]\, \Gamma(1-\frac{q}{\beta}) \approx e^{q\kappa(\beta+\frac{1}{\beta})} N^{q(\beta+\frac{1}{\beta})}\,F(q\,|\,\beta, \lambda_1,\lambda_2),\;
N\rightarrow\infty.
\end{equation}
On the other hand, by letting the constant $C$ in Eq. \eqref{keymax} to be taken to be
\begin{equation}
C = \frac{2\pi}{e^{\kappa}\Gamma(1-\beta^2)},
\end{equation}
and combining Eq. \eqref{keymax} with Eq. \eqref{generalidentity} , one obtains
\begin{equation}
{\bf E}[e^{q V_N}] = \lim_{\beta\rightarrow \infty} \Bigl[{\bf E}\bigl[X^{\frac{q}{\beta}}\bigr] \Gamma(1-\frac{q}{\beta})\Bigr].
\end{equation}
The right-hand side of this equation has only been determined for $0<\beta<1,$ see Eq. \eqref{rhs}.
Due to the self-duality of the right-hand side, one assumes that it gets frozen at $\beta=1,$ as first formulated in \cite{FLDR}.
\begin{conjecture}[The Freezing Hypothesis]
Let $\beta>1.$
\begin{equation}
{\bf E}\bigl[X^{\frac{q}{\beta}}\bigr] \Gamma(1-\frac{q}{\beta}) = 
{\bf E}\bigl[X^{\frac{q}{\beta}}\bigr] \Gamma(1-\frac{q}{\beta})\Big|_{\beta=1}.
\end{equation}
\end{conjecture}
One must note however that $C$ is not \emph{$\beta$-independent}. It is argued in \cite{YO} and \cite{FLDR2}  
that the $\Gamma(1-\beta^2)$ term 
shifts the maximum by  $-(3/2) \log \log N,$ see also \cite{Ding}. 
Overall, we then obtain by Eq. \eqref{rhs},
\begin{equation}
{\bf E}[e^{q\,V_N}] \approx e^{q(2\log N-(3/2)\log\log N+\,{\rm const})}\,F(q\,|\,\beta=1, \lambda_1,\lambda_2),\;
N\rightarrow\infty
\end{equation}
for some constant.\footnote{As remarked in \cite{FyodSimm}, this procedure only determines the distribution of the maximum up
to a constant term.}
Finally, recalling definitions of the critical Selberg integral distribution in Eq. \eqref{Selbergcrit} and of $F(q\,|\,\beta, \lambda_1,\lambda_2)$
in Eq. \eqref{Func}, and appropriately adjusting the constant, 
\begin{equation}
{\bf E}[e^{q\,V_N}] \approx e^{q(2\log N-(3/2)\log\log N+\,{\rm const})}\,{\bf E}\Bigl[M^q_{(\tau=1,\lambda_1,\lambda_2)}\Bigr]\Gamma(1-q),
\end{equation}
so that $Y'$ comes from the $\Gamma(1-q)$ factor and has the same law as $Y.$

The argument for the GFF on the circle goes through verbatim so we will only point out the key steps and omit redundant details.
Define the exponential functional 
\begin{equation}\label{Zdefc}
Z_{\lambda,\varepsilon}(\beta) = \sum\limits_{j=-N/2}^{N/2}  |1+e^{2\pi i\psi_j}|^{2\lambda\,\beta} e^{\beta V_\varepsilon(\psi_j)}.
\end{equation}
To describe its limit as $N\rightarrow \infty$ we need to compute the distribution of the total mass of the Bacry-Muzy measure on the circle
with a logarithmic potential that was defined in Eq. \eqref{chaoscircle}. Assuming Conjecture \ref{ourmainconjcircle},
we have
\begin{equation}\label{keyapproxc}
Z_{\lambda,\varepsilon}(\beta) \approx N^{1+\beta^2}\, e^{\beta^2\kappa}\,M_{(\tau,\beta\lambda,\beta\lambda)},\;N\rightarrow\infty.
\end{equation}
The rest of the argument is the same as for the GFF on the interval, with Eq. \eqref{invcircle} replacing Eq. \eqref{involutionint}.

\subsection{Inverse Participation Ratios of the Fyodorov-Bouchaud Model}\label{IPRsection}
In this section we will compute inverse participation ratios (IPR) of the Fyodorov-Bouchaud model based on 
Conjecture \ref{ourmainconjcircle}.

Let the gaussian field $V_\varepsilon(\psi)$ be as in Eq. \eqref{covkc}, $\beta\in(0,1)$ denote the inverse temperature, and
$\tau=1/\beta^2.$ 
Consider the associated partition function corresponding to the field with a non-random logarithmic
potential,
\begin{equation}\label{Zdefc}
Z_{\lambda,\varepsilon}(\beta) = \sum\limits_{j=-N/2}^{N/2}  |1+e^{2\pi i\psi_j}|^{2\lambda} e^{-\beta V_\varepsilon(\psi_j)}.
\end{equation}
When $\lambda=0,$ we will simply write
\begin{equation}\label{Zdef}
Z_{\varepsilon}(\beta) = \sum\limits_{j=-N/2}^{N/2}  e^{-\beta V_\varepsilon(\psi_j)}.
\end{equation}
The problem of computing the IPR of the Fyodorov-Bouchaud model
is that of computing
\begin{equation}
{\bf E}\Bigl[\frac{Z_{\varepsilon}(n\beta)}{Z^n_{\varepsilon}(\beta)}\Bigr],\;N\rightarrow\infty,
\end{equation}
for positive integer $n,$ see \cite{Fyo09}.
We will consider here a more general problem of computing the analytic
continuation in the form
\begin{equation}
{\bf E} \Bigl[Z_{\varepsilon}(q\beta) 
Z^s_{\varepsilon}(\beta)\Bigr],\;N\rightarrow\infty
\end{equation}
for real $q$ and generally complex $s.$ 

Our results are as follows. Let $\mathfrak{M}(q\,|\,\tau,\lambda)$  denote the Mellin transform of the Morris integral probability distribution
with $\lambda_1=\lambda_2=\lambda,$ see Eq. \eqref{thefunctioncircle}. Then,
\begin{equation}
{\bf E} \Bigl[Z_{\varepsilon}(q\beta) 
Z^s_{\varepsilon}(\beta)\Bigr] \approx N^{1+q^2\beta^2+(1+\beta^2)s} \;
\mathfrak{M}(s\,|\,\tau, \lambda=-q\beta^2), \, N\rightarrow\infty.
\end{equation}
Explicitly, 
\begin{align}\label{qsfunction}
{\bf E} \Bigl[Z_{\varepsilon}(q\beta) 
Z^s_{\varepsilon}(\beta)\Bigr] \approx N^{1+q^2\beta^2+(1+\beta^2)s} 
&\frac{\tau^{\frac{s}{\tau}}}{\Gamma^s\bigl(1-\beta^2\bigr)}
\frac{\Gamma_2((-2q\beta^2+1)/\beta^2+1-s\,|\,\tau)}{\Gamma_2((-2q\beta^2+1)/\beta^2+1\,|\,\tau)}
\times \nonumber \\ & \times
\frac{\Gamma_2(-s+1/\beta^2\,|\,\tau)}{\Gamma_2(\tau\,|\,\tau)}
\frac{\Gamma_2((1-q\beta^2)/\beta^2+1\,|\,\tau)^2}{\Gamma_2((1-q\beta^2)/\beta^2+1-s\,|\,\tau)^2}.
\end{align}
In particular, when $s=-n$ and $n\in\mathbb{N},$ we can use Eq. \eqref{CirNegMoments} to simplify this expression to
\begin{equation}\label{qn}
{\bf E} \Bigl[Z_{\varepsilon}(q\beta) 
Z^{-n}_{\varepsilon}(\beta)\Bigr] \approx N^{1+q^2\beta^2-(1+\beta^2)n}
\prod\limits_{j=0}^{n-1} \frac{\Gamma(1-q\beta^2+(j+1)\beta^2)^2 \,\Gamma(1-\beta^2)}{\Gamma(1-2q\beta^2+(j+1)\beta^2)\,\Gamma(1+j\beta^2)}. 
\end{equation}
Finally, when $q=n,$ we obtain the expression for the IPR,
\begin{equation}\label{IPRformula}
{\bf E}\Bigl[\frac{Z_{\varepsilon}(n\beta)}{Z^n_{\varepsilon}(\beta)}\Bigr] 
\approx 
 N^{1+n^2\beta^2-(1+\beta^2)n}
\prod\limits_{j=0}^{n-1} \frac{\Gamma(1-n\beta^2+(j+1)\beta^2)^2 \,\Gamma(1-\beta^2)}{\Gamma(1-2n\beta^2+(j+1)\beta^2)\,\Gamma(1+j\beta^2)}.
\end{equation}
For example, when $n=2,$ we recover the result of \cite{Fyo09},
\begin{equation}
{\bf E}\Bigl[\frac{Z_{\varepsilon}(2\beta)}{Z^2_{\varepsilon}(\beta)}\Bigr] 
\approx N^{2\beta^2-1}
 \frac{\Gamma(1-\beta^2)^4}{\Gamma(1-3\beta^2)
\Gamma(1-2\beta^2) \, \Gamma(1+\beta^2)}.
\end{equation}
The conditions of validity of our approximation are as follows. It is clear from Eq. \eqref{qsfunction} that 
the condition on $q$ is 
\begin{equation}
1-2q\beta^2+\beta^2 >0.
\end{equation}
Thus, the range of allowed values of $q$ is
\begin{equation}
q< \frac{1+\beta^2}{2\beta^2}.
\end{equation}
The range of $s$ is determined by 
\begin{equation}
\Re(s) < \min\Bigl(\frac{1}{\beta^2},  \frac{1}{\beta^2} + 1 - 2q, \frac{1}{\beta^2} + 1 - q\Bigr).
\end{equation}
In particular, for $q\in \mathbb{N},$ 
\begin{equation}\label{rescondition}
\Re(s) < \frac{1}{\beta^2} + 1 - 2q.
\end{equation}
For example, the expression in Eq. \eqref{qsfunction} holds for all $q<0$ and $\Re(s)<1/\beta^2$ and
the expression in Eq. \eqref{qn} holds for all $q<0$ and $n\in \mathbb{N}.$ In the specific case of IPR,
$s=-n,$ $q=n,$ so that Eq. \eqref{IPRformula} holds for the range
\begin{equation}
n < \frac{1+\beta^2}{2\beta^2}
\end{equation}
as the condition in Eq. \eqref{rescondition} is automatically satisfied.

We now proceed to explain our calculations. Following \cite{FyoBou}, we observe
\begin{equation}\label{keyapproxc}
Z_{\lambda,\varepsilon}(\beta) \approx N^{1+\beta^2}\,\int_{-\frac{1}{2}}^{\frac{1}{2}} |1+e^{2\pi i s}|^{2\lambda} \,M_\beta(ds),\;
N\rightarrow\infty,
\end{equation}
where $M_\beta(ds)$ denotes the Bacry-Muzy GMC on the circle. On the other hand,
we can use the following elementary application of the Girsanov theorem for gaussian fields.
Consider the change of measure
\begin{equation}
\frac{d\mathcal{Q}}{d\mathcal{P}} = e^{q^2\beta^2\log\varepsilon} e^{-q\beta V_\varepsilon(\phi)},
\end{equation}
where $\phi$ is some fixed angle. Then, we have the identity of gaussian processes in law
 viewed as functions of $\psi,$
\begin{equation}
V_\varepsilon(\psi)-q\beta {\bf Cov}(V_\varepsilon(\psi), V_\varepsilon(\phi))\big|_\mathcal{P} = 
V_\varepsilon(\psi)\big|_\mathcal{Q}.
\end{equation}
which is verified by a straightforward calculation of their characteristic functions. 
We now choose $\phi=1/2,$ and write by the rotational invariance of the field,\footnote{The idea
of using rotational invariance in this context is due to \cite{Fyo09}.}
\begin{align}
{\bf E} \Bigl[Z_{\varepsilon}(q\beta) 
Z^s_{\varepsilon}(\beta)\Bigr] = & N {\bf E} \Bigl[e^{-q\beta V_\varepsilon(1/2)}\,
Z^s_{\varepsilon}(\beta)\Bigr], \nonumber \\
= & N e^{-q^2\beta^2\log\varepsilon} {\bf E} \Bigl[\Bigl(  \sum_j e^{-\beta(V_\varepsilon(\psi_j)
+2q\beta \log |e^{2\pi i\psi_j}+1|)}\Bigr)^s\Bigr], \nonumber \\
= & N^{1+q^2\beta^2} {\bf E} \Bigl[Z^s_{\lambda=-q\beta^2, \varepsilon}(\beta)\Bigr],
\end{align}
and the result follows from Conjecture \ref{ourmainconjcircle}.

\subsection{Mod-Gaussian Limit Theorems}\label{modG}
\noindent 
The idea of reformulating the law of total mass of the Bacry-Muzy measure on the interval as a mod-Gaussian limit
was first introduced in \cite{Menon}. It is based on the observation that the smoothed indicator function of a linear statistic,
which converges to the $\mathcal{H}^{1/2}$ gaussian noise, converges to the \emph{centered} GFF on the interval, \emph{i.e.} 
the process $V_\varepsilon(x)-V_\varepsilon(0),$ where $V_\varepsilon(x)$ is as in Eq. \eqref{covkk}, also known
as Fractional Brownian motion with $H=0,$ cf. \cite{FKS}. There are many known examples of linear statistics
that converge to the $\mathcal{H}^{1/2}$ gaussian noise: counting statistics of Riemann zeroes \cite{BK}, \cite{Rodg},
the CLT of Soshnikov for the CUE ensemble \cite{Sosh} and the recent work on log-absolute
value of the characteristic polynomial of suitably scaled GUE matrices \cite{FKS}, see \cite{joint} for more examples. 
Hence, the results that are conjectured below are expected to be highly universal, \emph{i.e.} independent of the origin
of the linear statistic. For concreteness, we will assume in this section that the statistic comes from the Riemann zeroes following \cite{Menon}.

Consider the class of $\mathcal{H}^{1/2}$  test functions that was considered in \cite{BK}.
\begin{align}
\langle f,\,g\rangle \triangleq & \Re\int |w| \hat{f}(w)\overline{\hat{g}(w)}\,dw, \label{scalarf} \\
  = & -\frac{1}{2\pi^2} \int f'(x) g'(y)\log|x-y|dx\,dy \label{scalar}
\end{align}
plus some mild conditions on the growth of $f(x)$ and its Fourier transform 
$\hat{f}(w) \triangleq 1/2\pi \int f(x) e^{-iwx} \, dx$
at infinity. 
Assuming the Riemann hypothesis, we write non-trivial zeroes of the Riemann zeta function in the form $\{1/2+i\gamma\},$
$\gamma\in\mathbb{R}.$ Let $\lambda(t)$ be a function of $t>0$ that satisfies
the asymptotic condition
\begin{equation}\label{lt}
1\ll \lambda(t) \ll \log t
\end{equation}
in the limit $t\rightarrow \infty,$ where the number theoretic notation $a(t)\ll b(t)$ means  $a(t)=o\bigl(b(t)\bigr).$
Let $\omega$ denote a uniform random variable over $(1, 2),$ $\gamma(t) \triangleq \lambda(t) (\gamma-\omega t),$ 
and define the statistic
\begin{equation}\label{St}
S_t(f) \triangleq \sum\limits_{\gamma} f\bigl(\gamma(t)\bigr) -\frac{\log t}{2\pi \lambda(t)} \int f(u) du
\end{equation}
given a test function $f(x)$ in the $\mathcal{H}^{1/2}$ class. We note that $S_t(f)$ is centered
in the limit $t\rightarrow \infty$ as it is well known that the number of Riemann zeroes in the interval $[t,\,2t]$ is asymptotic to $t\log t/2\pi$ in this limit. The principal result of \cite{BK} and \cite{Rodg}
and the starting point of our construction is the following theorem, which is the number theoretic equivalent of Soshnikov's CLT for CUE.
\begin{theorem}[Convergence to a gaussian vector]\label{strong}
Given test function $f_1,$ $\cdots$ $f_k$ in $\mathcal{H}^{1/2},$
the random vector $\bigl(S_t(f_1),\cdots, S_t(f_k)\bigr)$ converges in law
in the limit $t\rightarrow\infty$ to
a centered gaussian vector $\bigl(S(f_1),\cdots S(f_k)\bigr)$ having the covariance 
\begin{equation}
{\bf Cov} \bigl(S(f_i), \,S(f_j)\bigr) = \langle f_i,\,f_j\rangle.
\end{equation}
\end{theorem}
The significance of the condition $\lambda(t)\ll \log t$ is that the number of zeroes that are visited by $f$ as $t\rightarrow \infty$ goes to infinity, \emph{i.e.} Theorem \ref{strong} is a mesoscopic central limit theorem. 

We can now summarize our results. Let $0<u<1$ and $\chi_u(x)$ denote the indicator function of the interval $[0, \,u].$
Let $\phi(x)$ be a smooth bump function supported on $(-1/2, \,1/2),$ 
and denote
\begin{equation}\label{kappa}
\kappa\triangleq -\int
\phi(x)\phi(y)\log|x-y|\,dxdy. 
\end{equation}
Define the $\varepsilon$-rescaled bump function by $\phi_\varepsilon(x)\triangleq 1/\varepsilon\phi(x/\varepsilon),$
and let $f_{\varepsilon, u}(x)$ be the smoothed indicator function of 
the interval $[0, \,u]$ given by 
the convolution of $\chi_u(x)$ with $\phi_\varepsilon(x),$
\begin{equation}\label{fu}
f_{\varepsilon, u}(x) \triangleq (\chi_u\star\phi_\varepsilon)(x) = \frac{1}{\varepsilon}\int \chi_u(x-y)\phi(y/\varepsilon)\,dy.
\end{equation}
Theorem \ref{strong} applies to $f_{\varepsilon, u}(x)$ for all $u>\varepsilon>0.$ Fix $\varepsilon>0$ and define the statistic 
$S_t(\mu,u,\varepsilon),$
\begin{equation}\label{Su}
S_t(\mu, u, \varepsilon)\triangleq \pi\sqrt{2\mu}\Bigl[\sum\limits_{\gamma} f_{\varepsilon, u}\bigl(\gamma(t)\bigr) - \frac{\log t}{2\pi \lambda(t)} \int f_{\varepsilon, u}(x) dx\Bigr].
\end{equation}
Then, we have the following key result, cf. \cite{Menon}. By Theorem \ref{strong},
the process $u\rightarrow S_t(\mu, u, \varepsilon),$ $u\in (0, 1),$
converges in law in the limit $t\rightarrow\infty$  to the centered gaussian field having the asymptotic covariance
\begin{align}\label{limcov1}
\begin{cases}
& -\mu 
\bigl(\log\varepsilon - \kappa + \log|u-v| - \log|u| - \log|v| \bigr) + O(\varepsilon), \; \text{if $|u-v|\gg \varepsilon$}, 
 \\
 &  -2\mu 
\bigl(\log\varepsilon - \kappa - \log|u| \bigr) + O(\varepsilon),
 \; \text{if $u=v.$}
\end{cases}
\end{align}
Thus, recalling the process $V_\varepsilon(u)$ in Eq. \eqref{covkk}, we have shown that the smoothed counting statistic $S_t(\mu,u,\varepsilon)$ converges to
the \emph{centered} GFF on the interval,
\begin{equation}
S_t(\mu,u,\varepsilon) \rightarrow \sqrt{\mu/2}\,\bigl(V_\varepsilon(u)-V_\varepsilon(0)\bigr).
\end{equation}
Moreover, we also showed in \cite{Menon} that in the case of $\varepsilon$ varying with $t$ the 
$\mathcal{H}^{1/2}$ covariance in Theorem \ref{strong}
is preserved under the following natural slow decay condition,
\begin{equation}
\frac{\lambda(t)}{\log t}\ll\varepsilon(t) \ll 1, \label{vart}
\end{equation}
so that under this condition we have the approximation,
\begin{equation}
S_t\bigl(\mu,u,\varepsilon(t)\bigr) \approx \sqrt{\mu/2}\,\bigl(V_{\varepsilon(t)}(u)-V_{\varepsilon(t)}(0)\bigr), \; t\rightarrow \infty.
\end{equation}
We emphasize that these results are believed to be universal in that they only require $\mathcal{H}^{1/2}-$Gaussianity of the limiting statistic
and refer the reader to \cite{joint} for the corresponding CUE calculations.

We will now formulate some conjectured mod-Gaussian limit theorems for the centered GFF on the interval and the associated smoothed 
counting statistic $S_t(\mu,u,\varepsilon)$ in the limit $t\rightarrow \infty.$
\begin{conjecture}[Weak version]\label{modcentered}
Let $-(\tau+1)/2<\Re(q)<\tau,$ $\tau=2/\mu,$ $0<\mu<2,$ then
\begin{align}
& \lim\limits_{\varepsilon\rightarrow 0} e^{\mu(\log \varepsilon-\kappa)\frac{q(q+1)}{2}} \Big[\lim\limits_{t\rightarrow\infty}  
{\bf E} \Bigl[\Bigl(\int_0^1 e^{S_t(\mu,u,\varepsilon)} du\Bigr)^q\Bigr]\Bigr], \label{M1transf}\\
& = \lim\limits_{\varepsilon\rightarrow 0} e^{\mu(\log \varepsilon-\kappa)\frac{q(q+1)}{2}} \Big[
{\bf E} \Bigl[\Bigl(\int_0^1 e^{\sqrt{\mu/2} \bigl(V_{\varepsilon}(u)-V_{\varepsilon}(0)\bigr)
} du\Bigr)^q\Bigr]\Bigr], \label{M1transfV}\\
& =  \Bigl(\frac{2\pi\tau^{1/\tau}}{\Gamma\bigl(1-\frac{1}{\tau}\bigr)}\Bigr)^q
\frac{\Gamma_2(1+q+\tau\,|\,\tau)}{\Gamma_2(1+2q+\tau\,|\,\tau)}
\frac{\Gamma_2(1-q+\tau\,|\,\tau)}{\Gamma_2(1+\tau\,|\,\tau)}
\frac{\Gamma_2(-q+\tau\,|\,\tau)}{\Gamma_2(\tau\,|\,\tau)}
\frac{\Gamma_2(2+q+2\tau\,|\,\tau)}{\Gamma_2(2+2\tau\,|\,\tau)}.
\label{M1}
\end{align}
\end{conjecture}
\begin{conjecture}[Strong version]
Let $-(\tau+1)/2<\Re(q)<\tau,$ $\tau=2/\mu,$ $0<\mu<2,$ and $\varepsilon(t)$ satisfy Eq. \eqref{vart}. Then
\begin{align}
\lim\limits_{t\rightarrow\infty}   e^{\mu(\log \varepsilon(t)-\kappa)\frac{q(q+1)}{2}} \Big[
{\bf E} \Bigl[\Bigl(\int_0^1 e^{S_t(\mu,u,\varepsilon(t))} du\Bigr)^q\Bigr]\Bigr] =
\lim\limits_{\varepsilon\rightarrow 0} e^{\mu(\log \varepsilon-\kappa)\frac{q(q+1)}{2}} \Big[\lim\limits_{t\rightarrow\infty}  
{\bf E} \Bigl[\Bigl(\int_0^1 e^{S_t(\mu,u,\varepsilon)} du\Bigr)^q\Bigr]\Bigr].
\end{align}
\end{conjecture}
The calculations behind these conjectures are all based on applications of the Girsanov theorem similar to those in Subsection \ref{IPRsection} and
Conjecture \ref{ourmainconjinterval}. In particular, the expression in Eq. \eqref{M1} corresponds to the expression for $\mathfrak{M}(q\,|\,\tau,
\lambda_1, \lambda_2)$ in Eq. \eqref{thefunctioninterval} with $\lambda_1=\mu\,q$ and $\lambda_2=0.$ 
The reader can find details of these calculations in Theorem 4.5 and Lemma 4.6 of \cite{Menon}.

The interest in the strong version of the conjecture is that it contains information about the statistical distribution of the zeroes at large but finite $t,$
whereas the weak version only describes the distribution at $t=\infty.$ Moreover,
as the strong conjecture fits into the general framework of mod-Gaussian convergence, cf. Jacod \emph{et. al.} \cite{Jacod},
the results of \cite{Feray} and \cite{Meliot} and the explicit knowledge of the limiting function make it possible to quantify the normality zone, 
\emph{i.e.} the scale up to which the tails of our exponential functionals are normal,
and the breaking of symmetry near the edges of the normality zone thereby quantifying precise deviations at large $t.$
We refer the interested reader to \cite{joint} for some rigorous results for positive integer $q$ that partially verify
our conjectures when the underlying statistic comes from CUE.

The same type of result can be formulated for the GFF on the circle. 
\begin{conjecture}
Let $V_\varepsilon(\psi)$ be the GFF on the circle as in Eq. \eqref{covkcir}.
\begin{align}
\lim\limits_{\varepsilon\rightarrow 0} e^{\mu(\log \varepsilon-\kappa)\frac{q(q+1)}{2}} \Big[
{\bf E} \Bigl[\Bigl(\int_0^1 e^{\sqrt{\mu/2} \bigl(V_{\varepsilon}(\psi)-V_{\varepsilon}(1/2)\bigr)
} d\psi\Bigr)^q\Bigr]\Bigr] = &
\frac{\tau^{\frac{q}{\tau}}}{\Gamma^q\bigl(1-\frac{1}{\tau}\bigr)}
\frac{\Gamma^3_2(q+1+\tau\,|\,\tau)}{\Gamma^2_2(\tau+1\,|\,\tau)\Gamma_2(2q+1+\tau\,|\,\tau)} \times \nonumber \\  
& \times
\frac{\Gamma_2(-q+\tau\,|\,\tau)}{\Gamma_2(\tau\,|\,\tau)}.\label{M1cir}
\end{align}
\end{conjecture}
The limiting function in this case corresponds to the Mellin transform of the Morris integral distribution in Eq. \eqref{thefunctioncircle}
with $\lambda=\mu q/2.$

We end this section with another conjecture, which combines Conjectures \ref{maxint} and \ref{maxintcircle} with Conjecture \ref{modcentered}, Eq. \eqref{M1transfV}. It
is a mod-Gaussian statement about the maximum of the centered GFF on the circle and interval. For simplicity, we let $\kappa=0.$
\begin{conjecture}
Let the gaussian field $V_\varepsilon(u)$ be as in Eq. \eqref{covkk} and let $V_\varepsilon(\psi)$ be the corresponding field on the circle.
Let $N=1/\varepsilon$ and consider the discretizations as in Conjectures \ref{maxint} and \ref{maxintcircle}. Let $-1<\Re(q)<1.$
\begin{align}
\lim\limits_{N\rightarrow \infty}  N^{-q^2-2q} (\log N)^{3q/2} {\bf E} \Bigl[e^{q \max\big\{V_\varepsilon(x_j)-V_\varepsilon(0), \,j=1\cdots N\big\}}\Bigr] 
= &  e^{q\,\text{const}}\,\Gamma(1-q)
\frac{G(2+2q)}{G(2+q)} 
\frac{G(2)}{G(2-q)} \times\nonumber \\ & \times
\frac{G(1)}{G(1-q)}
\frac{G(4)}{G(4+q)}, \label{M1crit}
\\
\lim\limits_{N\rightarrow \infty}  N^{-q^2-2q} (\log N)^{3q/2} {\bf E} \Bigl[e^{q \max\big\{V_\varepsilon(\psi_j)-V_\varepsilon(1/2), \,j=-N/2\cdots N/2\big\}}\Bigr] = & e^{q\,\text{const}}\,\Gamma(1-q) 
\frac{G(2+2q)\,G^2(2)}{G^3(2+q)}
\times
\nonumber \\  & \times
\frac{G(1)}{G(1-q)}. \label{M1critcir}
\end{align}
\end{conjecture}
It should be emphasized that the expressions on the right-hand side of Eqs. \eqref{M1}, \eqref{M1cir}, \eqref{M1crit}, and \eqref{M1critcir} are
not Mellin transforms of probability distributions. We refer the interested reader to \cite{cao17} and \cite{cao18} for deep results on the subtle nature of the distribution of the maximum of the centered GFF fields. In addition, as shown in \cite{cao17}, the distribution of the maximum of the two-dimensional gaussian field with the covariance $-\log|\vec{r}_1-\vec{r}_2|$ can be similarly quantified 
in terms of the critical analytic continuation of the complex Selberg integral in Eq. \eqref{critcomplexS}.


\section{Conclusions}
\noindent We have reviewed conjectured laws of the Bacry-Muzy GMC measures on the circle and interval with logarithmic potentials.
We have described both the analytical and probabilistic approaches to the Morris and Selberg integral
probability distributions that are believed to give these laws. The building blocks of the Morris and Selberg integral distributions are
the so-called  Barnes beta distributions, whose theory we have reviewed in detail. We have also described critical Morris and Selberg integral
distributions, which are conjectured to be the distributions of derivative martingales of the Bacry-Muzy GMC measures on the circle and interval.

Our analytical methods are not limited to the Morris and Selberg integrals. We have given the analytic continuation of the complex 
Selberg integral and established its involution invariance property.

We have considered three applications of our conjectures. The first application is the calculations of the distribution of
the maximum of the discrete Gaussian Free Field restricted to the circle and interval in terms of the critical Morris and Selberg integral
distributions, respectively. The second application is the calculation of inverse participation ratios of the Fyodorov-Bouchaud model.
In the third application we have conjectured two kinds of mod-Gaussian limit theorems. The first kind relates linear statistics 
that converge to the $\mathcal{H}^{1/2}-$Gaussian noise to the Selberg integral distribution. The second kind relates
the distribution of the maximum of the centered Gaussian Free Field restricted to the circle and interval to the 
critical Morris and Selberg integral distributions.

\section*{Acknowledgments}
\noindent The author wishes to thank Y. V. Fyodorov for bringing refs. \cite{cao17} and \cite{Fyo09} to our attention.

\appendix
\section{Appendix: Proofs of Results on Barnes Beta Distributions}
\renewcommand{\theequation}{\Alph{section}.\arabic{equation}}
\setcounter{equation}{0}  
\noindent 
In this section we will give proofs of the results in Section \ref{BarnesBeta}. 
The proofs rely on Eqs. \eqref{key} and \eqref{asym}, properties of infinitely divisible
distributions, and Lemma \ref{mylemma} below.
We note that Eqs. \eqref{key} and \eqref{asym} continue to hold for any $f(t)$
of the Ruijsenaars class, \emph{i.e.} analytic for
$\Re(t)>0$ and at $t=0$ and of at worst polynomial growth as
$t\rightarrow \infty,$ cf. Section 2 in \cite{Ruij}, provided
we define the corresponding generalized
Bernoulli polynomials by
\begin{equation}\label{Bdefa}
B^{(f)}_{m}(x) \triangleq \frac{d^m}{dt^m}\Big\vert_{t=0} \bigl[f(t)
e^{-xt}\bigr].
\end{equation}
The main
example that corresponds to the case of Barnes multiple gamma
functions is the function $f(t)$ that is defined in Eq. \eqref{fdef}.
Many of the proofs go through with a general $f(t)$ so we add the superscript
$(f)$ and drop $a$ from the list of arguments to indicate that the given formula does not require $f(t)$ to be as in Eq. \eqref{fdef}.  
\subsection{$M\leq N$}
\begin{lemma}[Main Lemma]\label{mylemma}
Let $f(t)$ be of the Ruijsenaars class.
Let $\lbrace b_k\rbrace,$
$k\in\mathbb{N},$ be a sequence of real
numbers, 
$n, r\in\mathbb{N},$ and $q\in\mathbb{C}.$ Define the function
$g(t)$ by
\begin{equation}\label{gfunction}
g(t) \triangleq f(t) e^{-qt} \frac{d^r}{dt^r}\bigl[e^{-b_0
t}\prod\limits_{j=1}^N (1-e^{-b_j t})\bigr].
\end{equation}
Then,
\begin{align}
g^{(n)}(0) & = \sum\limits_{m=0}^n \binom{n}{m} B^{(f)}_{n-m}(q)
\frac{d^{m+r}}{dt^{m+r}}\Big\vert_{t=0}\bigl[e^{-b_0 t}\prod\limits_{j=1}^N
(1-e^{-b_j t})\bigr], \label{line1} \\
& = (-1)^r \sum\limits_{p=0}^N (-1)^p
\sum\limits_{k_1<\cdots<k_p=1}^N \bigl(b_0+\sum b_{k_j}\bigr)^r \,
B^{(f)}_n\bigl(q+b_0+\sum b_{k_j}\bigr), \label{line2} \\
&= 0,\;{\rm if}\;r+n<N, \label{deriv1}\\
&= f(0)\,N! \prod\limits_{j=1}^N b_j,\;{\rm
if}\;r+n=N.\label{deriv2}
\end{align}
\end{lemma}
\begin{proof}
The expression in Eq. \eqref{line1} follows from Eq. \eqref{Bdefa}. Using
the identity
\begin{equation}
\prod\limits_{j=1}^N (1-e^{-b_j t}) = \sum\limits_{p=0}^N (-1)^p
\sum\limits_{k_1<\cdots<k_p=1}^N
\exp\bigl(-(b_{k_1}+\cdots+b_{k_p})t\bigr),
\end{equation}
we can write
\begin{equation}\label{exponID}
\frac{d^r}{dt^r}\bigl[e^{-b_0 t}\prod\limits_{j=1}^N (1-e^{-b_j
t})\bigr] = (-1)^r \sum\limits_{p=0}^N (-1)^p
\sum\limits_{k_1<\cdots<k_p=1}^N \bigl(b_0+\sum b_{k_j}\bigr)^r
\exp\bigl(-(b_0+\sum b_{k_j})t\bigr).
\end{equation}
Substituting this expression into Eq. \eqref{gfunction} and recalling
Eq. \eqref{Bdefa}, we obtain Eq. \eqref{line2}. Eq. \eqref{deriv1} is immediate
from the definition of $g(t)$ in Eq. \eqref{gfunction}, and
Eq. \eqref{deriv2} follows from Eq. \eqref{line1}.
\end{proof}
\begin{corollary}\label{mainauxID}
\begin{align}
\bigl(\mathcal{S}_N\,B^{(f)}_n\bigr)(q\,|\,b)& =0,\; n=0\cdots N-1,
\label{mainID} \\
\bigl(\mathcal{S}_N\,B^{(f)}_N\bigr)(q\,|\,b)&=f(0) \,N!
\prod\limits_{j=1}^N b_j, \label{mainIDN} \\
\bigl(\mathcal{S}_N \,x^n\bigr)(q\,|\,b)&=0,\;n=0\cdots N-1, \label{auxID}\\
\bigl(\mathcal{S}_N \,x^N\bigr)(q\,|\,b)&=(-1)^N N!\prod\limits_{j=1}^N
b_j. \label{auxIDN}
\end{align}
\end{corollary}
\begin{proof}
Eqs. \eqref{mainID} and \eqref{mainIDN} follow from Lemma \ref{mylemma}
by setting $r=0$ and recalling Eq. \eqref{S}. Eqs. \eqref{auxID} and
\eqref{auxIDN} follow from Eqs. \eqref{mainID} and \eqref{mainIDN} by
letting $f(t)=1$ in Lemma \ref{mylemma} so that the corresponding
Bernoulli polynomials are $B^{(f)}_n(x) = (-x)^n.$
\end{proof}
\begin{proof}[Proof of Theorem \ref{main}]
Let $M\leq N$ and $\Re(q)>-b_0.$ We start with the definition of 
$\eta_{M,N}(q\,|a,\,b)$ in Eq. \eqref{eta}, except we allow for
a general $L_M^{(f)},$ 
and substitute Eq. \eqref{key} for $\log\Gamma^{(f)}_M(w).$ By Eq. \eqref{mainID} in
Corollary \ref{mainauxID} and linearity of $\mathcal{S}_N,$ we
obtain
\begin{equation}\label{inter}
\eta_{M,N}(q\,|\,b) = \exp\Bigl(\int\limits_0^\infty
\Bigl[\bigl(\mathcal{S}_N \,\exp(-xt)\bigr)(q\,|\,b)-\bigl(\mathcal{S}_N
\,\exp(-xt)\bigr)(0\,|\,b)\Bigr]f(t) dt/t^{M+1}\Bigr).
\end{equation}
Letting $r=0$ in Eq. \eqref{exponID}, we have the identity
\begin{equation}
e^{-b_0 t} e^{-qt}\prod\limits_{j=1}^N (1-e^{-b_j t}) =
\bigl(\mathcal{S}_N \exp(-xt)\bigr)(q\,|\,b)
\end{equation}
so that Eq. \eqref{inter} can be simplified to
\begin{equation}
\eta_{M,N}(q\,|\,b) = \exp\Bigl(\int\limits_0^\infty
(e^{-qt}-1)e^{-b_0t}\prod\limits_{j=1}^N (1-e^{-b_j
t})f(t)\,dt/t^{M+1}\Bigr).
\end{equation}
\noindent This is the canonical representation of the Laplace
transform of an infinitely divisible distribution on $[0,\,\infty),$
confer Theorem 4.3 in Chapter 3 of \cite{SteVHar}.
\begin{align}
\eta_{M,N}(q\,|\,b) & = \exp\Bigl(-\int\limits_0^\infty
(1-e^{-tq})dK^{(f)}_{M,N}(t\,|\,b)/t\Bigr), \label{Keq}\\
dK^{(f)}_{M,N}(t\,|\,b) & \triangleq e^{-b_0t} \prod\limits_{j=1}^N
(1-e^{-b_j t}) f(t)\,dt/t^M. \label{Kdef}
\end{align}
It remains to note that $dK^{(f)}_{M,N}(t\,|\,b)$ satisfies the
required integrability condition
\begin{equation}
\int\limits_0^\infty e^{-st} dK^{(f)}_{M,N}(t\,|\,b) =
\int\limits_0^\infty e^{-st} e^{-b_0t}\prod\limits_{j=1}^N
(1-e^{-b_j t}) f(t)\,dt/t^M<\infty, \;s>0.
\end{equation}
Denote this non-negative distribution by $-\log\beta_{M,N}(b)$ so
that $\beta_{M,N}(b)\in(0, \,1]$ and
\begin{equation}
{\bf E}\Bigl[\exp\bigl(q\log\beta_{M,N}(b)\bigr)\Bigr] =
\eta_{M,N}(q\,|\,b), \;\Re(q)>-b_0.
\end{equation}
This is equivalent to Eq. \eqref{LKH}. 

Now, we note that
\begin{equation}
\int\limits_0^\infty dK^{(f)}_{M,N}(t\,|\,b)/t <
\infty\;\text{iff}\; M<N.
\end{equation}
It follows from Proposition 4.13 in Chapter 3 of \cite{SteVHar} that
$\log\beta_{M,N}(b)$ is absolutely continuous if $M=N.$ If $M<N,$
$-\log\beta_{M,N}(b)$ is compound Poisson by Theorem \ref{main} and
Proposition 4.4 in Chapter 3 of \cite{SteVHar}. In particular,
\begin{equation}\label{lambda}
{\bf P}\bigl[\log\beta_{M,N}(b)=0\bigr] =
\exp\Bigl(-\int\limits_0^\infty dK^{(f)}_{M,N}(t\,|\,b)/t\Bigr).
\end{equation}
The result follows from Eq. \eqref{Kdef}. \qed
\end{proof}
\begin{proof}[Proof of Corollary \ref{momentproblembarnes}]
It is sufficient to note that $\beta_{M,N}$ is compactly supported, confer Section 2.2 in \cite{Char}. \qed
\end{proof}
\begin{proof}[Proof of Theorem \ref{Asymptotics}]
The starting point of the proof is Eq. \eqref{asym}. Substituting
Eq. \eqref{asym} into Eq. \eqref{eta} and using linearity of
$\mathcal{S}_N,$ we can write in the limit of $q\rightarrow \infty,$
$|\arg(q)|<\pi,$
\begin{align}
\eta_{M,N}(q\,|\,b) & = \exp\Bigl(-\bigl(\mathcal{S}_N
\log\Gamma_M\bigr)(0\,|\,b)\Bigr)\exp\Bigl(-\frac{1}{M!}
\mathcal{S}_N\bigl(B^{(f)}_M(w)\,\log(w)\bigr)(q\,|\,b)+\nonumber\\
& + \sum\limits_{k=0}^M \frac{B^{(f)}_k(0)
\bigl(\mathcal{S}_N(-w)^{M-k}\bigr)(q\,|\,b)}{k!(M-k)!}\sum\limits_{l=1}^{M-k}
\frac{1}{l}+O(q^{-1})\Bigr).\label{sum2}
\end{align}
Now, to compute
$\mathcal{S}_N\bigl(B^{(f)}_M(w)\,\log(w)\bigr)(q\,|\,b),$ we expand
the logarithm in powers of $1/q,$ resulting in terms of the form
Eq. \eqref{line2} with $n=M.$ By Eq. \eqref{line1} in Lemma \ref{mylemma},
if $r+m>M,$ then such terms are of order $O(1/q).$ If $r+m\leq M$
and $M<N,$ they are all zero by Eq. \eqref{line1}. If $r+m\leq M$ and
$M=N,$ the only non-zero terms satisfy $r+m=N$ so that they have
degree zero in $q.$ Hence, we have the estimate
\begin{align}
\mathcal{S}_N\bigl(B^{(f)}_M(w)\,\log(w)\bigr)(q\,|\,b) & =
\log(q)\,\mathcal{S}_N\bigl(B^{(f)}_M(w)\bigr)(q\,|\,b) + O(q^{-1}),
\;
{\rm if}\; M<N, \label{estimate}\\
& = \log(q)\,\mathcal{S}_N\bigl(B^{(f)}_M(w)\bigr)(q\,|\,b) + O(1),
\; {\rm if}\; M=N.
\end{align}
If $M<N,$ the expression in Eq. \eqref{estimate} is zero by
Eq. \eqref{mainID} and the sum in Eq. \eqref{sum2} is zero by Eq. \eqref{auxID}
so that Eq. \eqref{ourasym} follows from Eq. \eqref{sum2}. If $M=N,$ the
result follows from Eqs. \eqref{mainIDN} and \eqref{auxIDN}. \qed
\end{proof}
From now on we restrict ourselves to $f(t)$ as in Eq. \eqref{fdef}.
\begin{proof}[Proof of Theorem \ref{FunctEquat}]
It is sufficient to substitute Eq. \eqref{feq}, written in the form
\begin{equation}
L_M\bigl(w+a_i\,|\,a\bigr) = L_{M}(w\,|\,a) -
L_{M-1}(w\,|\,\hat{a}_i),
\end{equation}
into Eq. \eqref{eta} and recall the definition of
$\eta_{M-1,N}(q\,|\,\hat{a}_i, b).$ \qed 
\end{proof}
\begin{proof}[Proof of Corollary \ref{FunctSymmetry}.]
To prove Eq. \eqref{fe2}, note that Eq. \eqref{S} implies
the identity
\begin{equation}
(\mathcal{S}_Nf)(q\,|\,b_0+x) = (\mathcal{S}_Nf)(q+x\,|\,b),
\end{equation}
and the result follows from definition of $\eta_{M,N}(q\,|a,\,b) $ in Eq. \eqref{eta}. Eq. \eqref{fe3} is
immediate from Eq. \eqref{eta}. Eq. \eqref{fe1eq} is equivalent to
Eq. \eqref{fe1} due to the special case of $q=0$ in Eq. \eqref{fe1},
\begin{equation}\label{auxidentity}
\eta_{M, N}(a_i\,|\,a, b) = \exp\bigl(-(\mathcal{S}_N
L_{M-1})(0\,|\,\hat{a}_i, b)\bigr).
\end{equation}
The proof of Eq. \eqref{fe4} follows from Eqs. \eqref{fe3}, \eqref{fe2}, and
\eqref{fe1eq}, in this order.
\begin{align}
\eta_{M, N}(q\,|\,a,\,b_j+a_i) & = \frac{\eta_{M, N-1}(q\,|\,a,\,\hat{b}_j)}{\eta_{M, N-1}(q\,|\,a,\,b_0+b_j+a_i, \hat{b}_j)}, \nonumber \\
& = \eta_{M, N-1}(q\,|\,a,\,\hat{b}_j)\frac{\eta_{M, N-1}(a_i\,|\,a,\,b_0+b_j, \hat{b}_j)}{\eta_{M, N-1}(q+a_i\,|\,a,\,b_0+b_j, \hat{b}_j)}, \nonumber \\
& =  \eta_{M-1, N-1}(q\,|\,\hat{a}_i,\,b_0+b_j,
\hat{b}_j)\frac{\eta_{M, N-1}(q\,|\,a,\,\hat{b}_j)}{\eta_{M,
N-1}(q\,|\,a,\,b_0+b_j, \hat{b}_j)}.
\end{align}
The result follows by yet another application of Eqs. \eqref{fe2} and
\eqref{fe3}. Finally, to verify Eq. \eqref{funceqsymmetry}, we combine
Eqs. \eqref{fe3} and \eqref{fe1eq} to obtain
\begin{align}
\eta_{M, N}(q+a_i\,|\,a,\,b) & = \eta_{M, N}(q\,|\,a,\,b) \,\eta_{M,
N}(a_i\,|\,a,\,b) \frac{\eta_{M-1, N-1}(q\,|\,\hat{a}_i,\,b_0+b_j,
\hat{b}_j)}{\eta_{M-1, N-1}(q\,|\,\hat{a}_i,\,\hat{b}_j)}, \nonumber
\\ & = \eta_{M, N}(q\,|\,a,\,b) \,\frac{\eta_{M,
N}(a_i\,|\,a,\,b)}{\eta_{M-1, N-1}(b_j\,|\,\hat{a}_i, \hat{b}_j)}
\frac{\eta_{M-1, N-1}(q+b_j\,|\,\hat{a}_i, \hat{b}_j)}{\eta_{M-1,
N-1}(q\,|\,\hat{a}_i,\,\hat{b}_j)}
\end{align}
by Eq. \eqref{fe2}. It remains to notice that Eqs. \eqref{fe3} and
\eqref{auxidentity} imply
\begin{equation}
\eta_{M, N}(a_i\,|\,a,\,b) = \eta_{M-1, N-1}(b_j\,|\,\hat{a}_i,
\hat{b}_j). \qed
\end{equation} 
\end{proof}
\begin{proof}[Proof of Corollary \ref{FactorizBarnes}.]
To prove Eq. \eqref{infinprod2} we need to recall the Shintani factorization of multiple gamma functions, cf. Eq. \eqref{generalfactorization}.
It is sufficient to note that Lemma \ref{mylemma} implies
\begin{equation}\label{basicid}
\bigl(\mathcal{S}_N x^n\bigr)(q\,|\,b) = \bigl(\mathcal{S}_N x^n\bigr)(0\,|\,b),\; n\leq N.
\end{equation}
and \(\Psi_{M}(w,y\,|\,a)\) and \(\phi_{M}(w,y\,|a, a_M)\) in Eq. \eqref{generalfactorization} are polynomials in $w$ of degree $M.$ Hence,
for $M\leq N$ we can write by Eq. \eqref{generalfactorization},
\begin{align}
\bigl(\mathcal{S}_N L_M\bigr)(q\,|\,a, b) - \bigl(\mathcal{S}_N L_M\bigr)(0\,|\,a, b) & = \sum\limits_{k=0}^\infty \Bigl[\bigl(\mathcal{S}_N L_{M-1}\bigr)(q+ka_M\,|\,a, b) - \nonumber \\
& -\bigl(\mathcal{S}_N L_{M-1}\bigr)(ka_M\,|\,a, b)\Bigr],
\end{align}
which is equivalent to Eq. \eqref{infinprod2}. Eq. \eqref{infinprod2} is equivalent to
\begin{equation}
\eta_{M,N}(q\,|\,a, b) = \prod\limits_{k=0}^\infty
\eta_{M-1,N}(q\,|\,\hat{a}_i, b_0+ka_i), \label{infinprod1}
\end{equation}
by Eq. \eqref{fe2} hence verifying Eq. \eqref{probshinfac}. Using Eq. \eqref{basicid} in the same way, Eq. \eqref{infbarnesfac} follows
from the Barnes factorization in Eq. \eqref{barnes}. Finally, 
to prove Eq. \eqref{probbarnesfac}, 
it is sufficient to note that Eq. \eqref{infbarnesfac} is equivalent to
\begin{equation}
\eta_{M,N}(q\,|\,a, b) = \prod\limits_{n_1,\cdots ,n_M=0}^\infty \eta_{0,N}(q\,|\, b_0+\Omega),
\end{equation}
which in turn is equivalent to Eq. \eqref{probbarnesfac}. 
Alternatively, Eq. \eqref{probbarnesfac} follows from Eq. \eqref{probshinfac} directly by induction.\qed
\end{proof}
\begin{proof}[Proof of Corollary \ref{moments}]
Repeated application of Eq. \eqref{fe1} gives the identity
\begin{equation}
\eta_{M,N}(q+k a_i\,|\,a, b) = \eta_{M,N}(q\,|\,a, b)
\exp\Bigl(-\sum_{l=0}^{k-1} \bigl(\mathcal{S}_N
L_{M-1}\bigr)(q+la_i\,|\,\hat{a}_i, b)\Bigr).
\end{equation}
Equations Eqs. \eqref{posmom} and \eqref{negmom} now follow by letting
$q=0$ and $q=-ka_i,$ respectively.
\end{proof}
\begin{proof}[Proof of Theorem \ref{barnesbetascaling}.]
By Eq. \eqref{scale} and the definition of the operator $\mathcal{S}_N$ in Eq. \eqref{S}, we have the identity
\begin{equation}
\bigl(\mathcal{S}_N L_M\bigr)(\kappa\,q\,|\,\kappa\,a,\kappa\,b) = \bigl(\mathcal{S}_N L_M\bigr)(q\,|\,a, b) - \frac{\log\kappa}{M!}
\bigl(\mathcal{S}_N B_{M,M}\bigr)(q\,|\,b)
\end{equation}
so that by Lemma \ref{mylemma} we obtain for $M\leq N$
\begin{align}
\log\eta_{M, N}(\kappa\,q\,|\,\kappa\,a,\kappa\,b) & = \bigl(\mathcal{S}_N L_M\bigr)(\kappa\,q\,|\,\kappa\,a,\kappa\,b) - \bigl(\mathcal{S}_N L_M\bigr)(0\,|\,\kappa\,a,\kappa\,b), \nonumber \\
& = \log\eta_{M, N}(q\,|\,a,\,b).
\end{align}
The result follows.\qed
\end{proof}

\begin{proof}[Proof of Corollary \ref{BarnesFactorSpecial}.]
This follows from Eq. \eqref{infbarnesfac} using the formula for the number of compositions of $k$ into exactly $m$ parts, confer \cite{Bressoud}, \emph{i.e.} $(k\,|\,M)$ equals the number of non-negative integer solutions to $n_1+\cdots +n_M=k.$
\end{proof}

\subsection{$N=M-1$}
We begin with a lemma on the action of $\mathcal{S}_{M-1}$ on polynomials that complements Lemma \ref{mylemma} .
\begin{lemma}\label{MyLemma}
Let $f(t)$ be a fixed function of Ruijsenaars type.
\begin{align}
\bigl(\mathcal{S}_{M-1} B^{(f)}_{M}\bigr)(q\,|\,b)  - \bigl(\mathcal{S}_{M-1} B^{(f)}_{M}\bigr)(0\,|\,b) 
 = -qf(0)\,M! \prod_{j=1}^{M-1} b_j. \label{id3}
\end{align}     
\end{lemma}       
\begin{proof}
Define the function
$g(t),$
\begin{equation}\label{gfunctionsine}
g(t) \triangleq f(t) e^{-qt} e^{-b_0
t}\prod\limits_{j=1}^N (1-e^{-b_j t}).
\end{equation}
Then, using the identity
\begin{align}
e^{-(q+b_0) t}\prod\limits_{j=1}^N (1-e^{-b_j t}) = & e^{-(q+b_0) t}\sum\limits_{p=0}^N (-1)^p
\sum\limits_{k_1<\cdots<k_p=1}^N
\exp\bigl(-(b_{k_1}+\cdots+b_{k_p})t\bigr), \nonumber \\  = & \bigl(\mathcal{S}_{N} e^{-xt}\bigr)(q|b), \label{theid}
\end{align}
one obtains
\begin{align}
g^{(k)}(0) & = 
(\mathcal{S}_{N}
B^{(f)}_k)(q\,|\,b). \label{line2orig} 
\end{align}
To verify Eq. \eqref{id3}, letting $N=M-1,$ $k=M,$ one observes that Eq. \eqref{line2orig} implies the identity
\begin{equation}
(\mathcal{S}_{M-1}
B^{(f)}_M)(q\,|\,b) = f(0)\frac{d^M}{dt^M}\Big\vert_{t=0}\prod\limits_{j=1}^{M-1} (1-e^{-b_j t}) +
M! \prod\limits_{j=1}^{M-1} b_j \,B^{(f)}_{1}(q+b_0).
\end{equation}
Hence, 
\begin{gather}
\bigl(\mathcal{S}_{M-1} B^{(f)}_{M}(x)\bigr)(q\,|\,b) - \bigl(\mathcal{S}_{M-1} B^{(f)}_{M}(x)\bigr)(0\,|\,b) 
= M! \prod\limits_{j=1}^{M-1} b_j \Bigl( B^{(f)}_{1}(q+b_0)- B^{(f)}_{1}(b_0)
\Bigr).
\label{alter}
\end{gather}
Finally, it remains to notice that $B^{(f)}_{1}(x)$ satisfies the identity
\begin{equation}
B^{(f)}_{1}(x) - B^{(f)}_{1}(x+y) = f(0)\,y,
\end{equation}
and Eq. \eqref{id3} follows. 
\qed
\end{proof}
\begin{proof}[Proof of Theorem \ref{mainsine}]
We wish to establish the identity
\begin{equation}\label{keyid}
\eta_{M, M-1}(q|a, b) = \exp\Bigl(
\int\limits_0^\infty \frac{dt}{t}\Bigl[(e^{-tq}-1)e^{-b_0 t}\frac{\prod\limits_{j=1}^{M-1} (1-e^{-b_j t})}{\prod\limits_{i=1}^M (1-e^{-a_i t})} + 
q e^{-t} \frac{\prod\limits_{j=1}^{M-1} b_j}{\prod\limits_{i=1}^M a_i}\Bigr]
\Bigr).
\end{equation}
Once it is established, Eq. \eqref{LKHsine} follows immediately by adding and subtracting $qt$ in the integrand,
which allows us to split the integral in Eq. \eqref{keyid} into two individual integrals and thereby bring it 
to the required L\'evy-Khinchine form. To verify Eq. \eqref{keyid} we need to recall the Ruijsenaars formula in Eq. \eqref{key}.
We see from Lemma \ref{mylemma} that for any $k=0\cdots M-1$ we have the identity
\begin{align}
\bigl(\mathcal{S}_{M-1} B_{M, k}(x|a)\bigr)(q\,|\,b) & - \bigl(\mathcal{S}_{M-1} B_{M, k}(x|a)\bigr)(0\,|\,b) 
 = 0.
\end{align}
It follows upon substituting Eq. \eqref{key} into Eq. \eqref{etaL} and using Eq. \eqref{id3} that the only non-vanishing terms are
\begin{align}
\eta_{M, M-1}(q|a, b) = \exp\Bigl( \int\limits_0^\infty \frac{dt}{t} \Bigl[\frac{
\bigl(\mathcal{S}_{M-1} e^{-xt}\bigr)(q\,|\,b)-\bigl(\mathcal{S}_{M-1} e^{-xt}\bigr)(0\,|\,b)}{\prod\limits_{i=1}^M (1-e^{-a_i t})}
 + qe^{-t}  \frac{\prod\limits_{j=1}^{M-1} b_j}{\prod\limits_{i=1}^M a_i}
\Bigr]
\Bigl),
\end{align}
and the result follows from Eq. \eqref{theid}. Now, we have the obvious identities
\begin{gather}
\int\limits_0^\infty e^{-b_0 t} \frac{\prod\limits_{j=1}^{M-1} (1-e^{-b_j t})}{\prod\limits_{i=1}^M (1-e^{-a_i t})} \frac{dt}{t} = \infty, \\
\int\limits_0^\infty e^{-b_0 t} \frac{\prod\limits_{j=1}^{M-1} (1-e^{-b_j t})}{\prod\limits_{i=1}^M (1-e^{-a_i t})} dt = \infty,
\end{gather}
which imply that $\log\beta_{M, M-1}(a,b)$ is absolutely continuous and supported on $\mathbb{R}$ by
Theorem 4.23 and Proposition 8.2 in Chapter 4 of \cite{SteVHar}, respectively. 
The scaling invariance in Eq. \eqref{scalinvgen}
is a corollary of Eqs. \eqref{scale} and \eqref{id3}. \qed
\end{proof}
\begin{proof}[Proof of Theorem \ref{newetaasympt}]
The asymptotic expansion of $\log\Gamma_M(w\,|\, a)$ in Eq. \eqref{asym} consists of
the leading term $-B_{M, M}(w\,|\,a)\,\log(w)/M!$ plus a polynomial remainder. Lemmas \ref{mylemma}
and \ref{MyLemma} show that the remainder term contributes $O(q)$ to $\log\eta_{M, M-1}(q|a, b).$ It remains to
show that as $q\rightarrow \infty,$
\begin{equation}
\bigl(\mathcal{S}_{M-1} B_{M, M}(x\,|\,a)\,\log(x)\bigr)(q|b) 
= - M!\bigl(\prod\limits_{j=1}^{M-1} b_j/\prod\limits_{i=1}^M a_i\bigr) q\log(q) +
O(q).
\end{equation}
Slightly generalizing the calculation in Lemma \ref{MyLemma}, let
\begin{equation}\label{gfunctionasym}
g(t) \triangleq f(t) e^{-qt} \frac{d^r}{dt^r}\bigl[e^{-b_0
t}\prod\limits_{j=1}^{M-1} (1-e^{-b_j t})\bigr].
\end{equation}
Then,
\begin{align}
g^{(n)}(0) & = \sum\limits_{m=0}^n \binom{n}{m} B^{(f)}_{n-m}(q)
\frac{d^{m+r}}{dt^{m+r}}\Big\vert_{t=0}\bigl[e^{-b_0 t}\prod\limits_{j=1}^{M-1}
(1-e^{-b_j t})\bigr], \label{line1sine} \\
& = (-1)^r \sum\limits_{p=0}^{M-1} (-1)^p
\sum\limits_{k_1<\cdots<k_p=1}^{M-1} \bigl(b_0+\sum b_{k_j}\bigr)^r \,
B^{(f)}_n\bigl(q+b_0+\sum b_{k_j}\bigr). \label{line2sine} 
\end{align}
In our case $n=M,$ $f(t)$ as in Eq. \eqref{fdef}, and the expression in Eq. \eqref{line2sine} is the coefficient of
$1/q^r$ that one gets by expanding $\bigl(\mathcal{S}_{M-1} B_{M, M}(x\,|\,a)\,\log(x)\bigr)(q|b)$ in powers of $1/q.$ 
By Eq. \eqref{line1sine} we can restrict ourselves to $m+r\geq M-1.$  On the other hand, 
the overall power of $q$ is $M-m-r$ so that we are only interested in $M-m-r\geq 1,$
as the other terms contribute $O(1).$  Hence, $m=M-r-1,$ and the contribution of such terms
is $O(q),$
\begin{equation}
\bigl(\mathcal{S}_{M-1} B_{M, M}(x\,|\,a)\,\log(x)\bigr)(q|b) = \log(q)\bigl(\mathcal{S}_{M-1} B_{M, M}(x\,|\,a)\,\bigr)(q|b)+  O(q),
\end{equation}
and the result follows from Eq. \eqref{id3}. \qed
\end{proof}
\begin{proof}[Proof of Theorem \ref{FunctEquatSine}]
The factorizations in Eqs. \eqref{infinprod2L} and
\eqref{infinprod1L} are corollaries of Lemma \ref{mylemma} and Shintani and Barnes factorizations of the multiple gamma functions, see Eqs. \eqref{generalfactorization} and \eqref{barnes}, respectively. A direct proof can be given as follows. The Mellin transform of $\beta_{M, M-1}(a, b, \bar{b})$ is
\begin{equation}\label{keyidb}
\eta_{M, M-1}(q|a, b, \bar{b}) = \exp\Bigl(
\int\limits_0^\infty \frac{dt}{t}\Bigl((e^{-tq}-1)e^{-b_0 t}+(e^{tq}-1)e^{-\bar{b}_0 t}\Bigr)\frac{\prod\limits_{j=1}^{M-1} (1-e^{-b_j t})}{\prod\limits_{i=1}^M (1-e^{-a_i t})}\Bigr).
\end{equation}
Let $i=M$ without any loss
of generality. The formula in Eq. \eqref{keyidb} can be written in the form
\begin{align}
\eta_{M, M-1}(q|a, b, \bar{b})  = &e^{
\int\limits_0^\infty \frac{dt}{t}\sum\limits_{k=0}^\infty \Bigl[(e^{-tq}-1)e^{-(b_0+ka_M) t} + (e^{tq}-1) e^{-(\bar{b}_0+ka_M)t} \Bigr] \prod\limits_{j=1}^{M-1} \frac{(1-e^{-b_j t})}{(1-e^{-a_j t})}}, \nonumber \\
= & \prod\limits_{k=0}^\infty  \exp\Bigl(
\int\limits_0^\infty \frac{dt}{t} (e^{-tq}-1)e^{-(b_0+ka_M) t} \prod\limits_{j=1}^{M-1} \frac{(1-e^{-b_j t})}{(1-e^{-a_j t})}\Bigr)
\times \nonumber \\  &\times \exp\Bigl(
\int\limits_0^\infty \frac{dt}{t} (e^{tq}-1) e^{-(\bar{b}_0+ka_M)t} \prod\limits_{j=1}^{M-1} \frac{(1-e^{-b_j t})}{(1-e^{-a_j t})}\Bigr)
, \label{keyid2}
\end{align}
which is exactly the same as the expression in Eq. \eqref{infinprod2L} if one recalls Eq. \eqref{LKH} and identity in Eq. \eqref{fe2}.
The expression in Eq. \eqref{keyid2} is equivalent to Eq. \eqref{probshin}.
If we now apply Eq. \eqref{probbarnesfac} to each Barnes beta factor in Eq. \eqref{probshin}, we obtain Eq. \eqref{probbarnes},
which is equivalent to Eq. \eqref{infinprod1L}. The scaling invariance in Eq. \eqref{scalinvL} follows from Eq. \eqref{scalinvgen}.
\qed
\end{proof}


\end{document}